\documentclass[11pt, reqno]{amsart}

\usepackage[utf8]{inputenc}
\usepackage[T1]{fontenc}
\usepackage{lmodern}
\usepackage{microtype} 

\usepackage{geometry} 

\usepackage{graphicx} 
\usepackage{caption}
\usepackage{subcaption}

\usepackage[dvipsnames]{xcolor} 
\usepackage{pgf,tikz} 
\usetikzlibrary{quotes}
\usetikzlibrary{cd}
\usetikzlibrary{babel}
\usetikzlibrary{decorations.markings}
\usetikzlibrary{patterns.meta}

\usepackage{amsmath} 
\usepackage{amssymb} 
\usepackage{amsthm} 
\usepackage{mathtools} 
\usepackage{xfrac} 
\usepackage{bm}
\usepackage{thmtools} 
\usepackage{thm-restate}

\usepackage[pdfencoding=auto, psdextra]{hyperref}
\usepackage{cleveref}
\Crefname{figure}{Figure}{Figures}

\usepackage[
backend=biber,
date=year,
style=alphabetic,
sorting=nyt,
giveninits=true,
doi=false,
isbn=false,
sortcites=true
]{biblatex}

\DefineBibliographyStrings{english}{
  prepublished = {preprint},
  mathesis = {Master's Thesis}
}

\AtEveryBibitem{%
	\clearfield{urlyear}%
  \ifentrytype{online}
    {}
    {\clearfield{eprint}%
     \clearfield{eprinttype}%
     \clearfield{eprintclass}}%
  \ifentrytype{misc}{}{\clearfield{url}}
}

\usepackage[shortlabels]{enumitem} 
\setlist[itemize]{noitemsep, topsep=0pt} 
\setlist[enumerate]{noitemsep, topsep=0pt} 

\usepackage{todonotes}
\usepackage{shuffle} 
\DeclareFontFamily{U}{shuffle}{}
\DeclareFontShape{U}{shuffle}{m}{n}{ <-8>shuffle7 <8->shuffle10}{}

\usepackage{bbm}

\newcommand*{\definitionname}{Definition}
\newcommand*{\theoremname}{Theorem}
\newcommand*{\propositionname}{Proposition}
\newcommand*{\corollaryname}{Corollary}
\newcommand*{\lemmaname}{Lemma}
\newcommand*{\remarkname}{Remark}
\newcommand*{\examplename}{Example}

\newcommand*{\notationname}{Notation}
\newcommand*{\conventionname}{Convention}
\newcommand*{\claimname}{Claim}
\newcommand*{\conjecturename}{Conjecture}
\newcommand*{\factname}{Fact}

\let\epsilon\varepsilon
\renewcommand\emptyset{\varnothing}
\let\implies\Rightarrow

\let\ForAll\forall
\renewcommand\forall{\;\ForAll}
\let\Exists\exists
\renewcommand\exists{\;\Exists}

\newcommand{\qq}[1]{\ensuremath{\mathrel{\quad\text{#1}\quad}}}
\newcommand{\abs}[1]{\left\lvert #1 \right\rvert}
\newcommand{\dd}[1]{\mathop{}\!\mathrm{d}#1}

\newcommand*{\on}[1]{\operatorname{#1}}

\newcommand*{\N}{\mathbb{N}} 
\newcommand*{\Z}{\mathbb{Z}} 
\newcommand*{\Q}{\mathbb{Q}} 
\newcommand*{\R}{\mathbb{R}} 
\newcommand*{\C}{\mathbb{C}} 
\newcommand*{\CP}{\mathbb{CP}} 
\newcommand*{\RP}{\mathbb{RP}} 

\newcommand*{\la}{\boldsymbol{a}}

\renewcommand{\Re}{\operatorname{Re}}

\DeclareMathOperator{\can}{can}
\DeclareMathOperator{\RW}{RW}
\DeclareMathOperator{\aux}{aux}

\DeclareMathOperator{\sgn}{sgn}
\let\ker\relax
\DeclareMathOperator{\ker}{ker}
\DeclareMathOperator{\im}{Im}

\DeclareMathOperator{\perm}{perm}
\DeclareMathOperator{\adj}{adj}

\DeclareMathOperator{\sv}{sv}
\DeclareMathOperator{\GL}{GL}
\DeclareMathOperator{\tr}{tr}
\DeclareMathOperator{\rank}{rank}

\DeclareMathOperator{\conf}{Conf}

\DeclareMathOperator{\darg}{darg}
\DeclareMathOperator{\dlog}{dlog}
\DeclareMathOperator{\diag}{diag}
\DeclareMathOperator{\vol}{vol}

\makeatletter
\renewcommand*\env@matrix[1][*\c@MaxMatrixCols c]{%
  \hskip -\arraycolsep
  \let\@ifnextchar\new@ifnextchar
  \array{#1}}
\makeatother

\makeatletter
\newsavebox{\@brx}
\newcommand{\llangle}[1][]{\savebox{\@brx}{\(\m@th{#1\langle}\)}%
  \mathopen{\copy\@brx\mkern2mu\kern-0.9\wd\@brx\usebox{\@brx}}}
\newcommand{\rrangle}[1][]{\savebox{\@brx}{\(\m@th{#1\rangle}\)}%
  \mathclose{\copy\@brx\mkern2mu\kern-0.9\wd\@brx\usebox{\@brx}}}
\makeatother

\theoremstyle{definition}
\newtheorem{definition}{\definitionname}[section]
\newtheorem{eg}[definition]{\examplename}

\theoremstyle{plain}
\newtheorem{theorem}[definition]{\theoremname}
\newtheorem{proposition}[definition]{\propositionname}
\newtheorem{corollary}[definition]{\corollaryname}
\newtheorem{lemma}[definition]{\lemmaname}

\newtheorem{fact}[definition]{\factname}

\theoremstyle{remark}
\newtheorem{remark}[definition]{\remarkname}

\makeatletter
\def\thmt@innercounters{equation,section,definition} 
\makeatother

\usepackage{./tikzit/tikzit}

\tikzstyle{internal vertex}=[fill=black, draw=black, shape=circle, scale=0.6]
\tikzstyle{external vertex}=[fill=white, draw=black, shape=circle, scale=0.6]
\tikzstyle{edge label}=[fill=none, draw=none, shape=circle, font={\scriptsize}]
\tikzstyle{node label}=[fill=none, draw=none, shape=circle, font={\tiny}]
\tikzstyle{subscriptExternal}=[fill=white, draw=black, shape=circle, scale=0.3]
\tikzstyle{subscriptInternal}=[fill=black, draw=black, shape=circle, scale=0.3]

\tikzstyle{standard}=[-]
\tikzstyle{arrow}=[-, arrows={-{Latex}}]
\tikzstyle{ds edge}=[-, dashed]
\tikzstyle{ds arrow}=[-, arrows={-{Latex}}, dashed]
\tikzstyle{standard accent}=[-, draw=orange]
\tikzstyle{arrow accent}=[-, arrows={-{Latex}}, draw=orange]
\tikzstyle{ds edge accent}=[-, dashed, draw=orange]
\tikzstyle{ds arrow accent}=[-, arrows={-{Latex}}, dashed, draw=orange]

\tikzset{every edge quotes/.style = {auto, font=\scriptsize, sloped}}
\title{Graph integrals, Feynman periods, and single-valued multiple zeta values}
\author{Jean-Luc Portner}
\address{Mathematical Institute, Oxford, OX2 6GG, United Kingdom}
\email{jean-luc.portner@maths.ox.ac.uk}

\begin{document}

\begin{abstract}
The Borel classes generating the stable cohomology of the general linear group can be represented by invariant differential forms.
It is known that pulling these forms back along a tropical Torelli map yields canonical convergent integrals associated to graphs, 
which are closely connected to the cohomology of $\mathrm{GL}_n$ and of graph complexes.
A natural question is what numbers these graph integrals are.
We answer this for primitive canonical integrals by showing that they coincide with a family of complex position-space integrals
arising in deformation quantisation. As a consequence, canonical integrals of graphs evaluate to single-valued multiple zeta values.
We further deduce that every single-valued multiple zeta value occurs as a rational linear combination of Feynman periods of graphs 
with massless propagators.
Finally, in the commutative graph complex, our result implies that the two associated cocycles agree. 
\end{abstract}

\vspace*{-0.5cm}
\maketitle
\vspace*{-0.25cm}
\graphicspath{{Images/}}

\section{Introduction}\label{sec:introduction}
Let $n \geq 3$, and let $G = (V_{G},E_{G})$ be an edge-ordered graph with no self-loops,
$V := \abs{V_{G}} = n+1$ vertices and $E := \abs{E_{G}} = 2n$ edges i.e. $E = 2V-2$.
This work shows the equality of two families of integrals associated to $G$: canonical integrals, obtained by
integrating over positive edge parameters, and RW-integrals, obtained by integrating over vertex configurations on $\C$.
The condition $E=2V-2$ ensures that in both constructions the integrand has top-degree on the relevant integration domain.

For the first family, let $\sigma_{2n} \subseteq \RP^{2n-1}$ denote the real positive coordinate simplex with
homogeneous coordinates  $x_1,\ldots,x_{2n} \geq 0$, one for each edge of $G$ induced by the edge-ordering.
For a $n \times  n$ matrix $X$ the primitive canonical form is defined as the $(2n-1)$-form
\begin{equation}\label{eq:canonicalTraceFormula}
	\beta_{X}^{2n-1} = \tr((X^{-1} \dd{X})^{2n-1})
.\end{equation}
Associated to the graph $G$ is its $(n-1) \times (n-1)$ Laplacian matrix $L_{G}$ whose entries are linear forms in $x_{k}$.
The corresponding canonical integral is
\[
I_{\can}(G) = \int_{\sigma_{2n}} \beta_{L_{G}}^{2n-1} \in \R
.\] 
These integrals are finite and arise from the study of the stable cohomology of the general linear group \cite{brown21}.
They may also be interpreted as Feynman integrals with massless propagators.

For the second family, consider the configuration space of $n+1$ pairwise distinct points in $\C$, one for each vertex of $G$,
modulo translation, rotation and positive real scaling,
$Z_{n} = \{z \in \C^{n+1} \mid z_{i}\neq z_{j} \text{ for } i \neq  j\} / (\C^{*} \ltimes \C)$, which
may be identified with $\mathfrak{M}_{0,n+2}$ the moduli space of $n+2$ marked points on $\mathbb{P}^{1}$.
The corresponding RW-integral is
\[
	I_{\RW}(G) = \frac{1}{(2 \pi i)^{n-1}} \int_{Z_{n}} \sum_{l=1}^{2n} \sum_{d\neq l}^{2n} 
	(-1)^{l+d+(l<d)} \log(\abs{z_{l}}^2) \bigwedge_{e \neq l, d} \dlog(\abs{z_{e}}^2) \in \R
,\] 
where $z_{e} = z_{u} - z_{v}$ for $e = (u,v) \in E_{G}$, $(l < d)$ is $1$ if $l < d$ and $0$ otherwise,
and the order of the wedge-product is induced by the edge-order of the graph.
These integrals are finite and real valued: they vanish for even $n$, while for odd $n$ the pre-factor is real.
RW-integrals have been used to study the action of the Grothendieck-Teichmüller group on deformation quantisation \cite{rossi14}.
The main theorem of this paper establishes the equality of the two types of integrals:
\begin{theorem}\label{thm:EqualityOfIntegralsGraphs}
	\[
	I_{\RW}(G) = I_{\can}(G)
	.\] 
\end{theorem}
From the perspective of Feynman integrals, the canonical integrals may be viewed as the parameter-space representation, 
while the RW-integrals correspond to the position-space formulation. The proof makes this correspondence explicit by
constructing an auxiliary integral on the product of position space and parameter
space. Integrating first over the edge-parameters recovers the RW-integral, while
integrating first over position space gives the canonical integral.

\begin{remark}
	For the sake of clarity in the statements of the theorems, 
	the definition of the RW-integrals used in this work differs from the original definition in \cite{rossi14} by a factor of 
	$(2 \pi i)^{n}\, 2^{2n-2}$.
\end{remark}

\subsection{An example}\label{sec:threeWheel}
We illustrate the general constructions with the simplest non-trivial example, the three-wheel graph $W_{3}$ shown in \cref{fig:threeWheel}.
\begin{figure}[htbp]
	\centering
	\begin{tikzpicture}
	\begin{pgfonlayer}{nodelayer}
		\node [style=internal vertex] (0) at (0, 0) {};
		\node [style=internal vertex] (1) at (0, 2.75) {};
		\node [style=internal vertex] (2) at (-2.75, -2) {};
		\node [style=internal vertex] (3) at (2.75, -2) {};
		\node [style=edge label] (4) at (0.25, 1.25) {$1$};
		\node [style=edge label] (5) at (1.5, -0.75) {$2$};
		\node [style=edge label] (6) at (-1.5, -0.75) {$3$};
		\node [style=edge label] (7) at (2, 0.75) {$4$};
		\node [style=edge label] (8) at (-2, 0.75) {$6$};
		\node [style=edge label] (9) at (0, -2.75) {$5$};
		\node [style=node label] (10) at (0, -0.5) {$1$};
		\node [style=node label] (11) at (0, 3.25) {$2$};
		\node [style=node label] (12) at (3.25, -2.25) {$3$};
		\node [style=node label] (13) at (-3.25, -2.25) {$4$};
		\node [style=internal vertex] (14) at (11, 0) {};
		\node [style=internal vertex] (15) at (15, 0) {};
		\node [style=internal vertex] (16) at (13, 3.25) {};
		\node [style=internal vertex] (17) at (9, 3.25) {};
		\node [style=internal vertex] (18) at (7, 0) {};
		\node [style=internal vertex] (19) at (9, -3.25) {};
		\node [style=internal vertex] (20) at (13, -3.25) {};
		\node [style=edge label] (21) at (15.75, 0) {$1$};
		\node [style=edge label] (22) at (13.5, -3.75) {$2$};
		\node [style=edge label] (23) at (8.5, -3.75) {$3$};
		\node [style=edge label] (24) at (6.25, 0) {$4$};
		\node [style=edge label] (25) at (8.5, 3.75) {$n-1$};
		\node [style=edge label] (26) at (13.5, 3.75) {$n$};
	\end{pgfonlayer}
	\begin{pgfonlayer}{edgelayer}
		\draw [style=arrow] (0) to (1);
		\draw [style=arrow] (0) to (3);
		\draw [style=arrow] (0) to (2);
		\draw [style=arrow, bend left=15] (1) to (3);
		\draw [style=arrow, bend left=15] (2) to (1);
		\draw [style=arrow, bend left=15] (3) to (2);
		\draw [style=standard] (15) to (14);
		\draw [style=standard] (15) to (20);
		\draw [style=standard] (20) to (14);
		\draw [style=standard] (20) to (19);
		\draw [style=standard] (19) to (14);
		\draw [style=standard] (19) to (18);
		\draw [style=standard] (18) to (14);
		\draw [style=standard] (17) to (14);
		\draw [style=standard] (16) to (15);
		\draw [style=standard] (16) to (14);
		\draw [style=standard] (16) to (17);
		\draw [style=ds edge] (18) to (17);
	\end{pgfonlayer}
\end{tikzpicture}
	\caption{The three-wheel graph $W_3$ with labeled vertices and labeled directed edges on the left and the general wheel graph $W_{n}$ on the right.}
	\label{fig:threeWheel}
\end{figure}
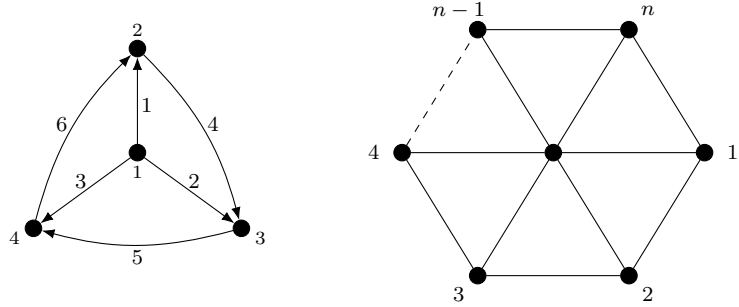
The incidence matrix of $W_{3}$ is given by $\epsilon$ below on the left.
For instance, the first row encodes that edge $1$ is oriented from vertex $1$ to vertex $2$, as depicted in \cref{fig:threeWheel}.
Removing any column from $\epsilon$ yields a reduced incidence matrix, denoted $\widetilde{\epsilon}$, and
the corresponding reduced Laplacian matrix is given by $L_{G} = \widetilde{\epsilon}^{T} \diag(x_1,\ldots,x_{6}) \widetilde{\epsilon}$.
While $L_{G}$ is independent of the orientation of the edges, it depends on the choice of removed column and the labelling of the edge variables, induced
by the edge-order.
For example, removing the first column (vertex $1$) gives the Laplacian matrix $L_{W_{3}}$ shown below on the right:
\[
\epsilon = \begin{pmatrix} 
	-1 &1 &0 &0\\
	-1 &0 &1 &0\\
	-1 &0 &0 &1\\
	0 &-1 &1 &0\\
	0 &0 &-1 &1\\
	0 &1 &0 &-1
\end{pmatrix} \qquad L_{W_3} = \begin{pmatrix} 
	x_1 + x_4 + x_6 &-x_4 &-x_6\\
	-x_4 &x_2 + x_4 + x_5 &-x_5\\
	-x_6 &-x_5 &x_3 + x_5 + x_6
\end{pmatrix}.
\]
The RW-integral for $W_{3}$ is
\[
	I_{\RW}(W_3) = \frac{1}{(2 \pi i)^{2}} \int_{Z_{3}} \sum_{l=1}^{6} \sum_{d\neq l}^{6} 
	(-1)^{l+d+(l<d)} \log(\abs{z_{l}}^2) \bigwedge_{e \neq l,d} \dlog(\abs{z_{e}}^2)
\]
where the edge variables $z_{e}$ are defined by $(z_{e_1},\ldots,z_{e_{6}})^{T} = \epsilon \cdot (z_1,\ldots,z_4)^{T}$.

We turn to the canonical integral.
Although the reduced Laplacian $L_{G}$ depends on the choice of reduced incidence matrix $\widetilde{\epsilon}$,
the canonical form $\beta^{2n-1}_{L_{G}}$ is independent of this choice \cite{brown21}.
A direct computation for $W_{3}$ yields
\begin{equation}\label{eq:canThreeWheel}
	I_{\can}(W_{3}) = \int_{\sigma_{6}} \beta_{L_{W_{3}}}^{5}
	= 10 \int_{\sigma_{6}} \frac{\Omega(x)}{\Psi^2},
\end{equation}
where $\Omega(x) = \sum_{i=1}^{2n} (-1)^{i-1} x_{i} \dd{x_{1}} \ldots \widehat{\dd{x_{i}}} \ldots \dd{x_{2n}}$
denotes the projective volume form, and $\Psi = \det(L_{W_{3}})$ is the first Symanzik polynomial of $G$.

Computing the right integral in \cref{eq:canThreeWheel} gives the classical result $60 \zeta(3)$, 
dating back to calculations by Chetyrkin, Kataev, and Tkachov \cite{chetyrkin80}.
A computation for the RW-integral of $W_{3}$, was sketched in \cite{rossi14}, yielding $60 \zeta(3)$.
More generally, the RW-integrals and the canonical integrals have been computed for all wheel graphs in \cite{portner23,schnetz24}.
For even $n$ both vanish, whereas for odd $n$ they are non-trivial and coincide:
\begin{equation}\label{eq:wheelGraphValues}
	I_{\RW}(W_{n}) = I_{\can}(W_{n}) = n \binom{2n}{n} \zeta(n) 
.\end{equation}

\subsection{Single-valuedness and weight}\label{sec:svAndWeight}
Next, we discuss the values of the RW and canonical integrals.
Consider the  multiple zeta values (MZVs)
\[
	\zeta(n_1,\ldots,n_{k}) := \sum_{0 < j_1 < j_2 < \ldots < j_{k} } \frac{1}{j_1^{n_1} j_2^{n_2} \ldots j_{k}^{n_{k}}},
	\quad n_{i} \in \N, n_{k} \geq 2
,\] 
and let $\mathcal{Z}$ denote the $\Q$-algebra they generate. This algebra contains a subalgebra $\mathcal{Z}^{\sv}$ of so-called 
single-valued multiple zeta values \cite{brown14}.
These numbers arise naturally, for example, in the low-energy expansion of closed string theory \cite{stieberger14,zerbini22}, 
and satisfy the same relations as MZVs as well as many more \cite{brown14}. For example,
\[
	\zeta_{\sv}(2k) = 0, \quad \zeta_{\sv}(2k+1) = 2 \zeta(2k+1) \qq{and} \zeta_{\sv}(5,3) = 14 \zeta(3) \zeta(5)
.\]
The weight of a multiple zeta value $\zeta(n_1,\ldots,n_{k})$ is defined as the sum of its arguments $n_1 + \ldots + n_{k}$.
We denote by $\mathcal{Z}_{k}$ (and $\mathcal{Z}^{\sv}_{k}$) the $\Q$-vector spaces spanned by (single-valued) MZVs of weight $k$.
The algebras $\mathcal{Z}$ and $\mathcal{Z}^{\sv}$ admit decompositions by weight:
\[
	\mathcal{Z} = \sum_{k=2}^{\infty} \mathcal{Z}_{k} \qquad
	\mathcal{Z}^{\sv} = \sum_{k=2}^{\infty} \mathcal{Z}^{\sv}_{k}
.\] 
Conjecturally, there are no $\Q$-linear relations between (single-valued) MZVs of different weight.
Under this conjecture, the above decompositions are direct sums.

From the definition of the RW-integrals and \cite{brown21single,schlotterer19,vanhove22,banks20} (see \cref{prop:RWSingleValuedWeight}) it follows that:
\begin{proposition}\label{prop:RWSingleValued}
	Let $G$ be a graph with $n+1$ vertices and $2n$ edges, i.e. $E = 2 V -2$. 
	Then, the RW-integral and thus the canonical integral, evaluates to single-valued
	MZVs of weight $n$.
	\[
	I_{\can}(G) =  I_{\RW}(G) \in \mathcal{Z}^{\sv}_{n}
	.\] 
\end{proposition}

\begin{eg}
The weight constraint implies that each canonical integral lies in the corresponding piece of $\mathcal{Z}^{\sv}$. 
The relevant pieces in weights $3,5,7$ and $9$ are generated by
$\{\zeta(3)\}$, $\{\zeta(5)\}$, $\{\zeta(7)\}$, and $\{\zeta(9),\ \zeta(3)^3\}$
respectively.
Thus, the three-wheel $W_{3}$, the five wheel $W_5$ and the zigzag $Z_5$,
any graph $G_{7}$ with $7$ loops, $14$ edges and $8$ vertices,
and any graph $G_{9}$ with $9$ loops, $18$ edges and $10$ vertices have
canonical integrals of the form
\begin{gather*}
	I_{\can}(W_{3}) = c_1\zeta(3),\quad I_{\can}(W_{5})= c_2\zeta(5),\quad I_{\can}(Z_{5}) = c_3\zeta(5), \\
	I_{\can}(G_7) = c_4 \zeta(7),\quad I_{\can}(G_{9}) = c_5\zeta(9)+c_6\zeta(3)^3,
\end{gather*}
where $c_1,\ldots,c_6 \in \Q$.
\end{eg}

A partial converse of \cref{prop:RWSingleValued} also holds. Rossi and Willwacher show in \cite{rossi14} that the path-ordered exponential
of a specific element in $\mathfrak{grt}_{1}$, whose coefficients are the RW-integrals, coincides with
the ratio of the KZ and the anti-KZ associator.
Brown shows in \cite[Lemma 5.1]{brown14} that this ratio is the generating series of the single-valued MZVs. This implies:
\begin{proposition}\label{prop:RWGenerateZsv}
	$\mathcal{Z}^{\sv}$ is generated as a $\Q$-algebra by the RW-integrals $I_{RW}(G)$, where we take all graphs $G$ with $n+1$ vertices and $2n$ edges,
	i.e. $E = 2V -2$.
\end{proposition}
However, RW-integrals are not closed under multiplication. For instance, $\zeta(3)$ and $\zeta(5)$ arise from the three- and five-wheel
graphs, respectively, and have weights three and five. 
Their product $\zeta(3) \zeta(5)$ which has weight eight, does not correspond to any single RW-integral.
Thus, RW-integrals only form a set of algebra generators for $\mathcal{Z}^{\sv}$ but not a set of vector space generators.

\subsection{Feynman periods}\label{sec:feynmanPeriods}
For every connected finite graph $G$ one may consider convergent integrals of the form
\begin{equation}\label{eq:feynmanPeriod}
\int_{\sigma_{E}} \frac{P(x)}{\Psi^{k}} \Omega(x)
,\end{equation}
where $k \in \N$, $P(x)$ is a polynomial in the variables $x$, $\Psi$ is the first Symanzik polynomial and $\Omega(x)$ denotes the volume form.
We define the $\Q$-vector space spanned by all such convergent integrals of all graphs to be the space of Feynman periods $\mathcal{P}^{\text{Feyn}}$.
This space has been studied in detail by Brown in \cite{brown17}.
The role of Feynman periods in the study of Feynman integrals is illustrated by the following example.
\begin{eg}
	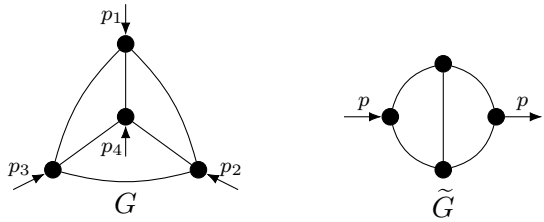
\begin{figure}[htpb]
		\centering
		\begin{tikzpicture}[scale = 0.7]
	\begin{pgfonlayer}{nodelayer}
		\node [style=internal vertex] (0) at (0, 0) {};
		\node [style=internal vertex] (1) at (0, 2.75) {};
		\node [style=internal vertex] (2) at (-2.75, -2) {};
		\node [style=internal vertex] (3) at (2.75, -2) {};
		\node [style=none] (10) at (4.25, -2.75) {};
		\node [style=none] (11) at (-4.25, -2.75) {};
		\node [style=none] (12) at (0, 4.25) {};
		\node [style=none] (13) at (0, -1.5) {};
		\node [style=edge label] (14) at (4, -2) {$p_2$};
		\node [style=edge label] (15) at (-4, -2) {$p_3$};
		\node [style=edge label] (16) at (-0.5, 3.75) {$p_1$};
		\node [style=edge label] (17) at (-0.5, -1.25) {$p_4$};
		\node [style=internal vertex] (18) at (10, 0) {};
		\node [style=internal vertex] (19) at (12, 2) {};
		\node [style=internal vertex] (20) at (12, -2) {};
		\node [style=internal vertex] (21) at (14, 0) {};
		\node [style=none] (22) at (8.25, 0) {};
		\node [style=none] (23) at (15.75, 0) {};
		\node [style=edge label] (24) at (9, 0.5) {$p$};
		\node [style=edge label] (25) at (15, 0.5) {$p$};
		\node [style=none] (26) at (12, -3.25) {$\widetilde{G}$};
		\node [style=none] (27) at (0, -3.25) {$G$};
	\end{pgfonlayer}
	\begin{pgfonlayer}{edgelayer}
		\draw [style=arrow] (11.center) to (2);
		\draw [style=arrow] (12.center) to (1);
		\draw [style=arrow] (10.center) to (3);
		\draw [style=arrow] (13.center) to (0);
		\draw [bend left] (18) to (19);
		\draw [bend left] (19) to (21);
		\draw [bend left] (21) to (20);
		\draw [bend left] (20) to (18);
		\draw (19) to (20);
		\draw [style=arrow] (22.center) to (18);
		\draw [style=arrow] (21) to (23.center);
		\draw (1) to (0);
		\draw (0) to (3);
		\draw (0) to (2);
		\draw [bend left=15] (2) to (1);
		\draw [bend left=15] (1) to (3);
		\draw [bend left=15] (3) to (2);
	\end{pgfonlayer}
\end{tikzpicture}
		\caption{The three-wheel graph $G$ with external legs on the left, and the two-loop master graph $\widetilde{G}$ on the right 
		\cite{bierenbaum03,grozin12}.}
		\label{fig:feynmanGraphs}
	\end{figure}
	Consider the three-wheel graph with four external legs $G$, shown on the left of \cref{fig:feynmanGraphs}.
	In its scalar Feynman integral in dimensional regularisation,  $D = 4 - 2\epsilon$, the Feynman period appears as the pole in $\epsilon$:
	\[
	I_{G}(\{m_{i}^2\}, \{p_{i}^2\},D)
	= \frac{1}{3 \epsilon} \int_{\sigma_{G}} \frac{\Omega(x)}{\Psi^2} + \mathcal{O}(\epsilon^{0})
	.\] 
	On the other hand, removing the external legs of $G$ and cutting an internal edge $e$ yields a two-point graph $\widetilde{G}$,
	shown on the right of \cref{fig:feynmanGraphs}.
	In the massless case, $m_{i} = 0$ for all internal edges, the corresponding two-point integral has the Feynman period 
	as the constant term in the $\epsilon$-expansion:
	\[
		I_{\widetilde{G}}(\{0\}, \{p^2\},D) = \frac{1}{p^2} \int_{\sigma_{G}} \frac{\Omega(x)}{\Psi^2} + \mathcal{O}(\epsilon^{1})
	.\] 
	For details and examples see \cite{brown09,panzer13,baikov10}.
	Numerators $P(x)$ arise, for instance, from non-scalar particles, like fermions, or from IBP reductions, and can be interpreted
	combinatorially as linear combinations of graphs obtained by subdividing edges by two-valent vertices.
\end{eg}

In \cite{brown21} Brown shows that the canonical integrals are Feynman periods. Together with \cref{thm:EqualityOfIntegralsGraphs}, this implies:
\begin{corollary}\label{cor:RWinPFeyn}
	The $\Q$-vector space of RW-integrals is contained in $\mathcal{P}^{\text{Feyn}}$.
\end{corollary}

In general, the structure of $\mathcal{P}^{\text{Feyn}}$ is rather delicate and not well understood. 
Many multiple zeta values are known to occur, yet not all do. For instance, it is conjectured that $\zeta(2)$ does not appear \cite{panzer17}. 
At the same time, Feynman periods are not restricted to multiple zeta values \cite{schnetz12}.
Moreover, certain simple periods, such as $\log(x)$, are known not to appear motivically \cite{brown17}.
Against this background, it is natural to ask which families of values are guaranteed to occur.
We find the following lower bound on the space $\mathcal{P}^{\text{Feyn}}$.
\begin{theorem}
	Every single-valued multiple zeta value can be realized as a $\Q$-linear combination of Feynman periods:
	\[
	\mathcal{Z}^{\sv} \subseteq \mathcal{P}^{\text{Feyn}}
	.\] 
\end{theorem}

\begin{proof}
	By \cref{prop:RWGenerateZsv}, the algebra $\mathcal{Z}^{\sv}$ is generated over $\Q$ by the RW-integrals. By
	\cref{cor:RWinPFeyn}, each such integral lies in $\mathcal{P}^{\text{Feyn}}$. It therefore suffices to show that
	$\mathcal{P}^{\text{Feyn}}$ is closed under multiplication. This closure follows from the two-vertex join 
	\cite[Proposition 40]{brown09}, or via face maps associated to sub- and quotient graphs \cite[Theorem 2] {panzer20}.
\end{proof}

\begin{remark}
	It is also known that not all MZVs which appear in $\mathcal{P}^{\text{Feyn}}$ are single-valued. For example, the 
	non-single-valued MZV $\zeta(12)$ occurs \cite{schnetz10}.
\end{remark}

In the study of Feynman integrals, the weight is also an important aspect, and the weight of the RW-integrals, 
as described in \cref{prop:RWSingleValued}, is rather surprising.

To each graph one can associate a mixed Hodge structure, often referred to as the Feynman motive \cite{bloch06},
such that all integrals of the form as in \cref{eq:feynmanPeriod} are periods of this motive.
The associated weight filtration induces a notion of weight on $\mathcal{P}^{\text{Feyn}}$.
For a graph with $n$ loops, these Hodge structures have weight bounded by $4n-6$ \cite{bloch06}.
Consequently, a generic integral of this type may be expected to evaluate to a period of Hodge weight $4 n-6$.
If this bound is not attained, the phenomenon is referred to as \emph{weight drop} \cite{yeats11}.

\begin{eg}
	Consider the five wheel graph $W_5$ with $n=5$ loops, $6$ vertices and $10$ edges from \cref{fig:threeWheel}.
	Its Feynman period and canonical integral \cite[Section 10.2]{brown21} are given by
	\[
		\int_{\sigma_{10}} \frac{\Omega(x)}{\Psi^2} = 70 \zeta(7) \qq{and} I_{\can}(W_{5}) = 18 \int_{\sigma_{10}} 
	\left(\frac{1}{\Psi^2} + 12 \frac{x_1 x_2 x_3 x_4 x_5}{\Psi^3} \right) \Omega(x)
	= 1260 \zeta(5)
	.\] 
	As the Hodge weight corresponds to twice the MZV weight, 
	the Feynman period $70 \zeta(7)$ has Hodge weight $14$, as expected for a generic integral.
	In contrast, the canonical integral $1260 \zeta(5)$ has Hodge weight $10$.
\end{eg}

Such a weight drop has been observed in all known computations of canonical integrals \cite{schnetz24},
and is unexpected from the viewpoint of Feynman integrals.
\Cref{prop:RWSingleValued} gives an analytic explanation for this phenomenon.

\subsection{Cohomology classes of the even graph complex}
The result of \cref{thm:EqualityOfIntegralsGraphs} also has implications for the cohomology of the even graph complex:
The even graph complex is the $\Q$-vector space $\mathcal{GC}_{2}$ spanned
by isomorphism classes of connected, simple, edge-ordered graphs with vertex valency at least $3$, subject to the relation
\[
	(G,\pi \circ \sigma) = \sgn(\pi) \cdot (G,\sigma) \qq{for} \pi \in \mathbb{S}_{E},
\]
where $\sigma$ denotes the edge-ordering of $G$ and $\mathbb{S}_{E}$ is the symmetric group on $\{1,\ldots,E\}$.
The space $\mathcal{GC}_{2}$ is bigraded by loop number and degree $\deg(G) = 2V -2 - E$. The differential is defined as
\[
	\partial G = \sum_{e \in E_{G}} G / e,
\]
where $G / e$ denotes the graph obtained from $G$ by contracting the edge $e$, equipped with the induced orientation.
Since $\partial^{2} = 0$, $\mathcal{GC}_{2}$ is a complex, and one can consider its homology
\[
H_{\bullet}(\mathcal{GC}_{2}) = \frac{\ker(\partial)}{\im(\partial)}.
\]

The homology of $\mathcal{GC}_{2}$ is of interest in several areas of mathematics.
In algebraic geometry, it computes a small part of the (stable) cohomology of the moduli stack $\mathcal{M}_{g}$ 
of smooth curves of genus $g$ \cite{chan21}. 
In number theory, it is closely related to the cohomology of arithmetic groups such as $\GL_{n}(\Z)$ \cite{brown21}. 
From the viewpoint of geometric topology, graph homology is connected to the cohomology of $\on{Out}(F_{n})$ \cite{conant03}, 
the group of outer automorphisms of the free group on $n$ generators. 
In addition, it plays a central role in deformation quantisation, where cocycles have been computed explicitly 
in order to study Poisson flows \cite{buring18,buring20}, in  operadic deformation theory, 
and in the study of the Grothendieck–Teichmüller group, where its degree-zero cohomology is identified with the 
Grothendieck–Teichmüller Lie algebra \cite{willwacher15}.
Despite these connections, the homology of $\mathcal{GC}_{2}$ remains difficult to compute,
with explicit results known only in very low degrees and for graphs with a small number of loops \cite{willwacher25}.

The connection to the homology of $\mathcal{GC}_{2}$ arises since the integrals $I_{\RW}$ and $I_{\can}$ 
define functions on the degree-zero part of $\mathcal{GC}_{2}$, that is they are degree-zero cochains on $\mathcal{GC}_{2}$.
Moreover, they are cocycles as they vanish on boundaries i.e. $I_{\RW}(\partial G) = I_{\can}(\partial G)  = 0$ \cite{rossi14,brown21}. 
Hence, they determine cohomology classes $[I_{\RW}]$ and $[I_{\can}]$.

Since $\mathcal{GC}_{2}$ is also graded by loop order, the homology decomposes as a direct sum.
Thus, each cohomology class splits into components indexed by the loop number $l$, which we denote by
$[I^{l}_{\RW}]$ and $[I^{l}_{\can}]$.
By \cref{prop:RWSingleValued}, these classes take values in the space of single-valued multiple zeta values of weight $l$:
\[
	[I^{l}_{\RW}], [I^{l}_{\can}] \in H^{0}(\mathcal{GC}_{2}) \otimes \mathcal{Z}_{l}^{\sv}
.\]
Cohomology classes can, for example, be used to detect non-trivial cycles.
Consider the wheel graphs  $W_{l}$ from \cref{fig:threeWheel}, they define cycles in $H_{0}(\mathcal{GC}_{2})$, 
since contracting any edge produces a multi-edge and hence vanishes. 
Their pairing with the cocycle $I^{l}_{\RW}$ is non-trivial as seen in \cref{eq:wheelGraphValues}, which shows that the wheel cycles $W_{l}$ and 
the cocycles $I^{l}_{\RW} = I^{l}_{\can}$ are not exact i.e. they represent non-zero (co)homology classes.

The main result, \cref{thm:EqualityOfIntegralsGraphs}, not only implies that the integrals define the same cohomology class,
but proves the stronger statement:
\begin{corollary}
	The RW cocycle and the canonical cocycle coincide. That is, the integrals $I_{\RW}^{l}$ and $I_{\can}^{l}$ coincide as cocycles.
\end{corollary}

The cocycles also descend to a quotient of the graph complex.
Recall that a graph has a two-vertex cut if the removal of two vertices disconnects it.
Let $\mathcal{GC}_{2}^{\mathrm{tri}}$ denote the quotient of $\mathcal{GC}_{2}$ by the subcomplex spanned by graphs with two-vertex cuts.
Willwacher proved in \cite{willwacher25tri} that $\mathcal{GC}_{2}^{\mathrm{tri}}$ is quasi-isomorphic to $\mathcal{GC}_{2}$.

For canonical integrals, the corresponding vanishing on graphs with two-vertex cuts has been shown by Simone Hu in her PhD thesis via a
quite complicated argument. For the RW-integrals, the same vanishing is more direct; see \cref{rem:RWTwoVertexCut}. This gives
the following consequence.

\begin{corollary}
The cocycles $I_{\RW}^{l}=I_{\can}^{l}$ descend to the quotient complex $\mathcal{GC}_{2}^{\mathrm{tri}}$. 
In particular, they define cohomology classes
\[
[I_{\RW}^{l}] = [I_{\can}^{l}]
\in H^{0}\bigl(\mathcal{GC}_{2}^{\mathrm{tri}}\bigr).
\]
\end{corollary}

\subsection{Effective formulas}\label{sec:effectiveFormulas}
The proof of \cref{thm:EqualityOfIntegralsGraphs} also yields simplified representations of both the RW-integrals and the canonical forms.
Obtaining explicit formulas for canonical forms is difficult as substantial cancellations occur.
Indeed, from the trace formula $\beta^{2n-1}_{L_{G}} = \tr((L_{G}^{-1} \dd{L_{G}})^{2n-1})$
one would naively expect a denominator of the form $\Psi^{2n-1}$, where $\Psi = \det(L_{G})$.
However, Brown shows in \cite[Theorem~2.1]{brown21} that the canonical form can always be written as
\[
	\beta^{2n-1}_{L_{G}} = \frac{P_{G}(x)}{\Psi^{\frac{n+1}{2}}} \Omega(x),
\]
for some polynomial $P_{G}(x)$. Thus, roughly three-quarters of the powers of $\Psi$ expected in the denominator cancel with the numerator.
In the course of proving \cref{thm:EqualityOfIntegralsGraphs}, we obtain an alternative explicit expression of the following shape:
\begin{align*}
	\beta^{2n-1}_{L_{G}} = \frac{Q_{G}(x)}{\Psi^{n}} \frac{\Omega(x)}{x_{d}} \qq{with} Q_{G}(x) = (2n-1)
	\sum_{S,T} \pm \det(\widetilde{\epsilon}_{S d,\bullet}) \det(\widetilde{\epsilon}_{T d,\bullet}) \perm(\adj(M)_{S,T})
.\end{align*}
The precise signs, indexing conventions and definitions are given in \cref{thm:paraToCan}.
Although the expression is initially obtained in the form shown above, the numerator $Q_G(x)$ is divisible by $x_d$. 
Hence the apparent factor $x_d$ in the denominator cancels, and the form has denominator $\Psi^n$. 
While this is not the optimal denominator, it nevertheless accounts for approximately two-thirds of the available cancellations.
Moreover, the proof makes the source of these cancellations explicit.
This provides, to our knowledge, the most efficient explicit general formula currently available for computing canonical forms.

Computing canonical forms explicitly is already difficult, but integrating them is even harder.
As discussed in \cref{sec:feynmanPeriods}, they arise in the context of Feynman motives, 
where complicated periods, such as those associated with a $K3$ surface, may occur \cite{schnetz12}. 
By contrast, RW-integrals are significantly more accessible: their integrands are products of logarithms and dlog-forms, 
and explicit integration algorithms and dedicated software are available for such integrals \cite{banks20,schnetz18}. 
It is therefore useful to simplify the RW-integrals as much as possible.
To this end, we observe that the summations over the distinguished edges $l$ and $d$ in the definition of $I_{\RW}$ may be removed:
\begin{proposition}\label{prop:introNoLsum}
	Let $l,d \in \{1,\ldots, 2n\} $ with $l \neq d$. Then
	\[
		I_{\RW}(G) = (2n-1) \frac{(-1)^{l+d+(l<d)}}{(2 \pi i)^{n-1}} \int_{Z_{n}} \log\Bigg( \frac{\abs{z_{l}}^2}{\abs{z_{d}}^2} \Bigg) 
		\bigwedge_{e\neq l,d} \dlog\Bigg(\frac{\abs{z_{e}}^2}{\abs{z_{d}}^2}\Bigg)
	\]
	where $z_{e} = z_{u} - z_{v}$ for an edge $e = (u,v) \in E_{G}$ and $(l < d)$ is $1$ if $l < d$ and $0$ otherwise.
\end{proposition}
To pass from $Z_{n}$ to an affine chart, one may fix the endpoints of the distinguished edge $e_{d}$ at $0$ and $1$.
In this chart, $\abs{z_{d}} = 1$, the denominators $\abs{z_{d}}^2$ disappear.
Thus, after fixing the scale $\abs{z_d}=1$, the formula reduces the number of terms to be computed from $2n-1$ to one.
A further promising approach would be to express RW-integrals as the single-valued map applied to real $n$-fold integrals, 
using the methods introduced in \cite{brown21single}. We plan to investigate this in future work.

\subsection{Acknowledgements}
The author is grateful for the support and numerous insightful discussions with his supervisors, Erik Panzer and Francis Brown, 
during the course of this work. He also thanks Thomas Willwacher for helpful discussions and comments.
The author is funded through the Royal Society grant URF\textbackslash R1\textbackslash 201473 and the European Research Council (ERC) under
European Union's Horizon Europe programme (grant agreement No. 101167287).
For the purpose of Open Access, the author has applied a CC BY public copyright licence to any Author Accepted Manuscript (AAM) version
arising from this preprint.

\section{Overview}
This work is organised as follows. In the present section, we first extend the definitions of the integrals $I_{\RW}$ and $I_{\can}$
from graphs to matrices $A \in \mathbb{R}^{2n \times n}$. 
Next, we outline the main ideas of the proof of \cref{thm:EqualityOfIntegralsGraphs} and discuss the main technical difficulties.
The remaining sections are organised as follows. In \cref{sec:propsOfInts}, 
we establish several properties of the integrals that will be used in the full proof.
In \cref{sec:proofOfMainTheorems,sec:prodToPos,sec:prodToPara,sec:paraToCan}, we give complete proofs of the steps outlined below. 
Finally, in \cref{app:Dodgson}, we collect a number of standard matrix identities used in the preceding arguments.
For $k \in \N$, we denote by $[k]$ the ordered set $(1,\ldots,k)$.

\subsection{Generalisation to matrices}
Let $n \geq 3$ and  consider a real  $2n \times n$ matrix $A$ of rank $n$, 
such that no row of $A$ is equal to $0$. Denoting the rows of $A$ by $a_{i}$, we regard each $a_{i}$ as a complex-linear functional on $\C^{n}$,
obtaining $2n$ linear forms $\la_{i}: \C^{n} \to \C$ given by $z \mapsto \la_{i}(z) := a_{i} z$.
For brevity and readability we often write $\la_{i}$ for $\la_{i}(z)$.
We define the RW-integral associated with $A$ by
\[
I_{\RW}(A) = \frac{(2i)^{2n-1}}{2 (2 \pi i)^{n}} \int_{C_{A}} \sum_{l=1}^{2n} (-1)^{l-1} \log(\abs{\la_{l}(z)}^2) \bigwedge_{e \neq l} \darg(\la_{e}(z))
,\] 
where, $C_{A} = \conf_{A}(\C) / \R_{> 0}$ and $\conf_{A}(\C) := \{z \in \C^{n} \mid \la_{i}(z) \neq 0 \forall i \in [2n] \}$.
Although the integrand involves logarithms and arguments, the resulting integral $I_{\RW}(A)$ is a well-defined single-valued function of $A$.

This definition of the RW-integral appears to differ from the initial definition given in \cref{sec:introduction}.
To recover $I_{\RW}(G)$ for a given graph $G$, we take $A$ to be a reduced incidence matrix of $G$, so that $a_{e}(z) = z_{u} - z_{v}$ for
$e = (u,v) \in E_{G}$.
Passing from the full to the reduced incidence matrix is geometrically equivalent to fixing the removed vertex at $0$,
which accounts for the translation quotient $\C$ present in the original space $Z_{n}$.
Since the generalised space $C_{A}$ already quotients by positive real scaling $\R_{> 0}$, it only remains to account for rotations.
We achieve this by considering the canonical projection $\pi: C_{A} \to C_{A} / S^{1}$ and integrating along the $S^{1}$-fibers.
The resulting integral over $C_{A} / S^{1}$ agrees with the original definition of $I_{\RW}$, after exchanging the angular differentials with
logarithmic differentials of absolute values; see \cref{prop:ogRWIntegral}.

In the graph case, the condition $\la_{i}(z) \neq  0 \forall i \in [2n]$ reduces to $z_{j} \neq z_{k}$ for all edges $(j,k) \in E_{G}$.
Taking all possible pairs $(j,k) \in [n]^2$ with $j\neq k$, corresponding to the edges of the complete graph $K_{n}$,
yields the classical configuration space $\conf_{n}(\C) = \{z \in \C^{n} \mid z_{j} \neq z_{k} \forall j,k \in [n], j \neq k\}$.
The spaces $\conf_{A}(\C)$ therefore generalise the classical configuration spaces by allowing arbitrary linear forms $\la_{i}$, rather than only
forms of the type $z_{j} - z_{k}$.

\begin{remark}
Despite the notation, $I_{\RW}(A)$ depends only on the linear forms $\la_1(z)$, \ldots, $\la_{2n}(z)$, 
and not on the chosen matrix representation.
Indeed, it is invariant under right multiplication by $\GL_{n}(\R)$, which corresponds to a real change of basis of $\C^{n}$.
Moreover, $I_{\RW}(A)$ is invariant under rescaling of the individual linear forms: for any diagonal $2n \times  2n$ matrix $D$, one has 
$I_{\RW}(D A) = I_{\RW}(A)$. 
Thus, $I_{\RW}(A)$ may equivalently be regarded as a function of the associated arrangement of hyperplanes $H_{i} = \ker(\la_{i}(z))$.
\end{remark}

With the matrix $A$, we can also generalise the canonical integrals.
Let $\sigma_{2n} \subseteq \RP^{2n-1}$ denote the projective simplex with homogeneous coordinates $x = (x_1,\ldots,x_{2n})$, satisfying $x_{i} \geq 0$.
We define a map 
\[
L_{A}(x)\colon  \sigma_{2n} \to \mathcal{P}^{\geq 0}_{n}, \qquad L_{A}(x) = A^{T} \diag(x) A
\]
from the simplex $\sigma_{2n}$ into the space $\mathcal{P}_{n}^{\geq 0}$ of positive semidefinite symmetric $n \times n$ matrices.
We denote its image by $\sigma_{A}$. 
This construction generalises the tropical Torelli map \cite{chan12,caporaso10} from the moduli space of tropical curves to the moduli space of tropical abelian varieties.
In the graph case, when $A$ is a reduced incidence matrix of a graph with $2n$ edges and $n+1$ vertices, the map $L_{A}$ coincides exactly with
the tropical Torelli map.

The canonical integral associated with $A$ is then defined by
\[
I_{\can}(A) = \int_{\sigma_{A}} \beta^{2n-1}_{X} = \int_{\sigma_{2n}} \beta^{2n-1}_{L_{A}}
.\] 
As in the case of the RW-integral, $I_{\can}(A)$ is invariant under right multiplication of $A$ by $\GL_{n}(\R)$.
On the level of $L_{A}$, this corresponds to the conjugation action $L_{A} \mapsto P^{T} L_{A} P$ for $P \in \GL_{n}(\R)$. 
In particular, $I_{\can}(A)$ depends only on the underlying linear forms $\la_{1}, \ldots, \la_{2n}$ and not on the specific matrix representation $A$.
The definition of $I_{\RW}$ and $I_{\can}$ can be extended to all  $2n \times n$-matrices $A$ by setting them equal to zero whenever
$\rank(A) < n$ or $A$ contains a zero row.

In \cref{sec:proofOfMainTheorems} we prove the following generalisation of \cref{thm:EqualityOfIntegralsGraphs}, which
is recovered by taking $A$ to be a reduced incidence matrix $\widetilde{\epsilon}$.
\begin{theorem}\label{thm:EqualityOfIntegrals}
	Let $n \geq 3$ and $A \in \R^{2n \times n}$ of rank $n$, such that no row of $A$ is equal to $0$.
	Then
	\[
	I_{\RW}(A) = I_{\can}(A)
	.\] 
\end{theorem}
Although both canonical and RW-integrals are defined for all $n$, we show in \cref{prop:RWvanish,prop:canVanish} 
that they  vanish whenever $n$ is even. 
In the graph case, this corresponds to the integrals vanishing for graphs with an odd number of vertices.

\begin{remark}
	We note that the canonical integrals admit a further extension to integration of wedge products of forms $\beta_{X}^{2n-1}$. 
	An analogous extension for RW-integrals is currently not known and appears to be more subtle. 
	One obstruction is that such generalised canonical integrals need not evaluate to single-valued polylogarithms. 
	For instance, in the case of the complete graph $K_{6}$, the integral of
	$\beta^{5}_{L_{K_{6}}} \wedge \beta^{9}_{L_{K_{6}}}$ involves the multiple zeta value $\zeta(3,5)$ \cite{brown21}, 
	which is conjecturally not a single-valued multiple zeta value.
\end{remark}

In the following, we restrict to the case of odd $n$, since the integrals vanish for even $n$ as discussed above.
Instead of restricting the canonical integrals to integration over subspaces $\sigma_{A}$, we can extend them to
arbitrary locally finite chains on $\mathcal{P}^{\geq 0}_{n} / \GL_{n}(\Z)$.
Brown shows in \cite{brown23} that these integrals are convergent for all such chains. This is because, after
passing to a suitable compactification,
the canonical form extends to the boundary and restricts to zero.
Consequently, it defines a class in the dual of the locally finite (Borel-Moore) homology,
\[
	[I_{\can}] \in \left(H^{lf}_{k}(\mathcal{P}_{n} / \GL_{n}(\Z);\R)\right)^{\vee}
.\] 
This space may, via integration pairing, be identified with the compactly supported cohomology
$H_{c}^{k}(\mathcal{P}_{n} / \GL_{n}(\Z); \R)$, which in turn is Poincar\'e dual to the group cohomology of $\GL_{n}$:
\[
H^{d_{n} - k}(\GL_{n}(\Z); \R) \cong H^{d_{n}-k}(\mathcal{P}_{n} / \GL_{n}(\Z);\R) \cong \left( H_{c}^{k}(\mathcal{P}_{n} /\GL_{n}(\Z);\R) \right)^{\vee}
,\]
where $d_{n} := n (n+1) / 2$. 
Thus, the cocycle $I_{\can}$, is related to the cohomology of the general linear group.

It is natural to ask what values this cocycle takes. To answer this, we use that locally finite homology is isomorphic to the homology
of the Voronoi complex:
\[
	H_{k}(\on{Vor}_{n}) \cong
	H^{lf}_{k}(\mathcal{P}_{n} / \GL_{n}(\Z);\R)
.\] 
The Voronoi complex $\on{Vor}_{n}$, as described in \cite{elbaz10}, is a cell complex whose cells are polyhedral cones in the space $\mathcal{P}^{\geq 0}_{n}$. Importantly, the rays (generators) of the Voronoi cells are rank-one positive semidefinite matrices.
Using simplicial subdivision (triangulation of the link) of these polyhedral cones, we can
reduce the computation of the cocycle $I_{\can}$ to the computation of the integrals for simplicial cones.
Observe then that the image $\sigma_{A}$ of the map $L_{A}(x)$ is a simplicial cone in $\mathcal{P}^{\geq 0}_{n}$.
Since the rays of the cone $\sigma_{A}$ are rank-one matrices, they can be written as $a_{i}^{T} a_{i}$ for some vector $a_{i}$.
Hence,
\[
	\sigma_{A} = \{ x_1 a_{1}^{T} a_{1} + x_2 a_{2}^{T} a_{2} + \ldots + x_{2n} a_{2n}^{T} a_{2n} \mid (x_1,\ldots,x_{2n}) \in \R^{2n} \}
.\]
Using \cref{thm:EqualityOfIntegrals} we can compute the cocycle $I_{\can}$ on the Voronoi complex via the RW-integrals.
Since the integration of RW-integrals is much better understood than that of canonical integrals, and arguing similar
to the proof sketch of \cref{prop:RWSingleValuedWeight}, we obtain the following description of the values taken by these cocycles.
As the Voronoi complex is rational, the rays can be represented integrally, i.e. with $a_{i} \in \Z^{n}$. 
We can therefore conclude:
\begin{corollary}
	The cocycle $I_{\can}$ on the Voronoi complex $\on{Vor}_{n}$ takes values in the $\Q$-vector space of single-valued polylogarithms
	with integer letters (poles), evaluated at integer points.
\end{corollary}

\subsection{The idea behind proving \texorpdfstring{\cref{thm:EqualityOfIntegrals}}{the theorem}}
Let us next explain the mechanism behind the proof of
\cref{thm:EqualityOfIntegrals}. The goal is not to give a full proof, but
to sketch the correspondence between the RW-integral and the canonical
integral which underlies the argument. To keep the mechanism visible in this sketch
we freely interchange limits, derivatives and integrals, ignore convergence issues, and suppress
boundary/counterterm corrections and signs. These issues are addressed in the full proof
in \cref{sec:proofOfMainTheorems}.

The proof is organised around an auxiliary integral $I_{\aux}(A)$ on the product
space $\C^n \times \sigma_{2n}$. This integral serves as a bridge between
the two sides: integrating first over the simplex $\sigma_{2n}$ produces the
RW-integral, while integrating first over $\C^n$ produces the canonical
integral. Thus the equality is proved by realising both integrals as two
different evaluations of the same object. The diagram below summarises this
strategy:
\begin{figure}[htpb]
	\begin{tikzcd}
	& & I_{\aux} \text{ on } \C^n \times \sigma_{2n}
	\arrow[lldd,"\int_{\sigma_{2n}}", "\text{Feynman parametrisation}"']
	\arrow[rrdd, "\text{Gaussian integration}", "\int_{\C^{n}}"']
	& & \\ \\
	\text{integral on } \C^n
	\arrow[dd, "\text{scale integration}"]
	&  & & &
	\text{integral on } \sigma_{2n}
	\arrow[dd, "\text{matrix identities}"'] \\ \\
	I_{\RW} \text{ on } C_A
	& & & &
	I_{\can} \text{ on } \sigma_{2n}
	\end{tikzcd}
	\caption{The auxiliary integral $I_{\aux}$ relates the RW and the canonical integral.}
	\label{fig:proofSketch}
\end{figure}
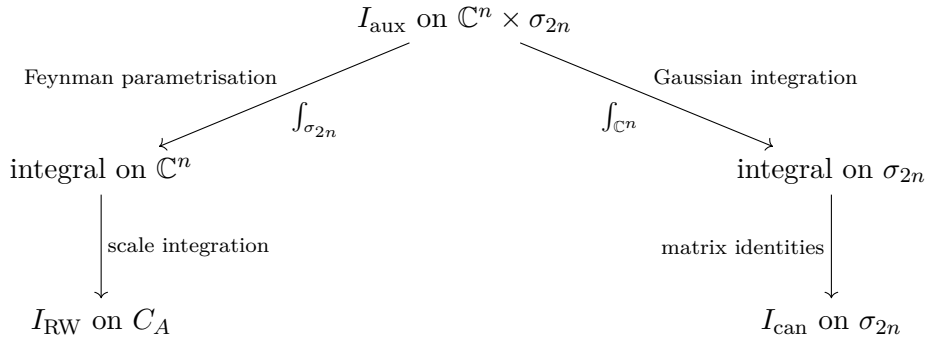

Let $n \geq 3$, and let $d \in [2n]$. We write $z = (z_1,\ldots,z_{n})$ for the coordinates on $\C^{n}$.
The simplex $\sigma_{2n}$ is initially parametrised by coordinates $\alpha = (\alpha_1,\ldots,\alpha_{2n})$.
Later, when passing from the auxiliary integral to the canonical integral, we change variables by setting $\alpha_{i} = x_{i}^{-1}$.
Thus the auxiliary integral is written in the coordinates $z$ and $\alpha$. 
The  RW-integral is written in the $z$-variables, and the canonical integral is written in the dual simplex coordinates $x_{i} = \alpha_{i}^{-1}$.

Let $A \in \R^{2n \times n}$ of rank $n$ such that no row of $A$ is equal to $0$. Let $L(\alpha) = A^{T} \diag(\alpha) A$ and define
\[
P_{d}^{l} := \begin{pmatrix}[c|c]
	(\diag(\partial p) \cdot A)^{ld,\bullet}  & (\diag(\partial \overline{p}) \cdot A)^{ld,\bullet} \\ \hline
	a_{d} & 0\\
	0 & a_{d}
\end{pmatrix}
\qq{and}
M = \begin{pmatrix} 
\diag(x_1,\ldots,x_{2n}) &-A\\ A^{T} & 0
\end{pmatrix}.
\]
Here $\partial p = \left( \frac{\partial}{\partial p_1}, \ldots, \frac{\partial}{\partial p_{2n}} \right)$, and 
$(-)^{ld,\bullet}$ denotes deletion of the rows indexed by $l$ and $d$.
The matrix $P_{d}^{l}$ should be understood as a matrix whose entries are differential operators.
Thus, in expressions such as $\diag(\partial p)\cdot A$, the derivatives do not act on the constant matrix $A$. 
Rather, $A$ supplies the scalar coefficients of the corresponding differential operators. 
Explicitly, the $(i,j)$-entry of $\diag(\partial p)\cdot A$ is $a_{ij}\frac{\partial}{\partial p_i}$.

The auxiliary integral is defined as
\[
	I_{\aux}(A) = \frac{1}{(2 \pi i)^{n}}  \int_{\C^{n} \times \sigma_{2n}} \sum_{l=1}^{2n} \pm \lim_{p \to 0} \det(P_{d}^{l}) \frac{
	\left(\bigwedge_{i \in [n]} \dd{z_{i}} \wedge \dd{\overline{z_{i}}}\right) \Omega(\alpha)}
	{(z^{T} L(\alpha) \overline{z}  + p^{T} A z + \overline{p}^{T} A \overline{z} + \alpha_{l})^{2}}
\]
Here the $(z,\overline{z})$-variables are integrated over $\C^{n}$, while the $\alpha_{i}$ are the coordinates on $\sigma_{2n}$.
The variables $p$ and $\overline{p}$ are auxiliary variables: the differential operator $\det(P_{d}^{l})$ acts on the 
denominator as a function of $(p,\overline{p})$, and the result is then evaluated at $p = \overline{p} = 0$.

\subsubsection{The RW-integral side}
To recover the RW-integral, we must integrate out the simplex variables in the auxiliary integral. 
The mechanism governing this step is the Feynman parametrisation. 
While typically used (via the Schwinger trick) to rewrite a product of powers as a single simplex integral, 
we will use the identity in reverse to evaluate our inner integral. We recall the exact identity:
\begin{proposition}\label{prop:FeynmanPara}
	Let $k \in \N$, let $A_1,\ldots,A_{k} \in \R_{> 0}$ and let $s_1,\ldots,s_{k} \in \R_{>0}$. Then
	\[
	\frac{1}{A_1^{s_1} \ldots A_{k}^{s_{k}}} = \frac{\Gamma(s_1+\ldots+s_{k})}{\Gamma(s_1) \ldots \Gamma(s_{k})} \int_{\sigma_{k}}
	\frac{\alpha_1^{s_1-1}\ldots \alpha_{k}^{s_{k}-1}}
	{( \alpha_{1} A_1+\ldots+ \alpha_{k} A_{k})^{s_1+\ldots+s_{k}}} \Omega(\alpha)
	.\]
\end{proposition}
To apply this to our specific case, consider the auxiliary integral $I_{\aux}(A)$. Applying the differential
operator $\det(P_d^l)$ and setting $p=\overline p=0$ gives
\[
	\frac{\Gamma(2n)}{(2 \pi i)^{n}} \int_{\C^{n}} \int_{\sigma_{2n}} \sum_{l = 1}^{2n} \pm
	\frac{ \bigwedge_{e \neq l,d} ( \overline{\la_{e}} \dd{\la_{e}} - \la_{e} \dd{\overline{\la_{e}}}) 
		\wedge \dd{\la_{d}} \wedge \dd{\overline{\la_{d}}}}
	{(\sum_{e } \alpha_{e} \abs{\la_{e}}^2 + \alpha_{l})^{2n}} \Omega(\alpha)
\]
Observe that the inner integral over the simplex $\sigma_{2n}$ matches the right-hand side of the proposition for $k=2n$, $s_{i} = 1$ for
$1 \leq i \leq 2n$ and $A_{i} = \abs{\la_{i}}^2$ for $i \neq l$ and $A_{l} = \abs{\la_{l}}^2 + 1$.
By running the Feynman parametrisation backwards, we evaluate this simplex integral giving
\[
	\frac{(2 i)^{2n-1}}{2 (2 \pi i)^{n}} \int_{\C^{n}} \sum_{l=1}^{2n} \pm \frac{1}{\abs{\la_{l}}^2 + 1} \dlog(\abs{\la_{d}}^2) 
	\bigwedge_{e \neq l} \darg(\la_{e}) 
\]
where we used the identities
\[
	2 i \darg(\la_{e}) = \frac{\overline{\la_{e}} \dd{\la_{e}} - \la_{e} \dd{\overline{\la_{e}}} }{\abs{\la_{e}}^2} \qq{and}
	i \dlog(\abs{\la_{d}}^2) \wedge \darg(\la_{d}) = \frac{\dd{\la_{d}} \wedge \dd{\overline{\la_{d}}}}{\abs{\la_{d}}^2}
\]
It remains to pass from the integral over $\C^n$ to the integral over
$C_A$. We separate the radial scale using the map $C_A \times \R_{>0} \rightarrow \C^n \setminus \{0\}$ via $(z,r) \mapsto \sqrt{r} z$.
Pulling back the preceding integral gives,
\[
	\frac{(2 i)^{2n-1}}{2 (2 \pi i)^{n}} \int_{C_{A}} \sum_{l=1}^{2n} \pm  
	\left( \int_{0}^{\infty} \frac{1}{r \abs{\la_{l}}^2 + 1} \frac{\dd{r}}{r} \right) 
		\bigwedge_{e \neq l} \darg(\la_{e}) 
\]
The radial integral has a finite and a divergent constant part. Formally, this may be written as
\[
	\frac{(2 i)^{2n-1}}{2 (2 \pi i)^{n}}
	\int_{C_{A}} \sum_{l=1}^{2n} \pm \log(\abs{\la_{l}}^2) \bigwedge_{e \neq l} \darg(\la_{e}) + 
\log(0)  \frac{(2 i)^{2n-1}}{2 (2 \pi i)^{n}} \int_{C_{A}} \sum_{l=1}^{2n} \pm \bigwedge_{e \neq l} \darg(\la_{e})
\]
The first term is the RW-integral. The second term is the divergent constant
term. In the full proof it is cancelled by a corresponding counterterm.
At the level of this formal sketch, it vanishes morally as the integrand is invariant under rotation and hence not a volume form.
This route of integration, therefore,  gives $ I_{\aux}(A) = I_{\RW}(A)$.

\subsubsection{The canonical integral side}
Next, we compute $I_{\aux}$ in the other order, by integrating first over
$\C^n$. This part of the argument has three steps. First, completing the
square turns the integral over $\C^n$ into a standard complex Gaussian-type
integral. Second, the differential operator $\det(P_d^l)$ is expanded, and
its action on the resulting quadratic form is expressed in terms of permanents.
Third, the determinant-permanent expression which remains is identified with
the canonical form by a purely algebraic matrix identity.

For the first step, we replace the exponent $2$ in the denominator of
$I_{\aux}(A)$ by a complex parameter $s$, initially with $\Re(s)>2n$.
This places the integral in a region of absolute convergence, where the
Gaussian integration over $\C^n$ can be performed first. 
The resulting expression is then analytically continued back to $s=2$.
Completing the square in the denominator gives
\[
z^T L \overline z + p^T A z + \overline p^T A \overline z + \alpha_l = w^T \overline w + \alpha_l - \overline p^T A L^{-1}A^T p,
\]
after a change of variables in $\C^n$. The Jacobian of this change of
variables contributes the factor $\det(L)^{-1}$. Hence the standard complex
Gaussian integral gives
\begin{equation}\label{eq:SketchAfterComplexInt}
	I_{\aux}(A) = \int_{\sigma_{2n}} \sum_{l=1}^{2n} \pm \frac{1}{\det(L)} \frac{\Gamma(s-n)}{\Gamma(s)} \lim_{p \to 0} \det(P_{d}^{l}) 
	\frac{1}{(\alpha_{l}-\overline{p}^{T} A L^{-1} A^{T} p)^{s}} \Omega(\alpha)
\end{equation}

Next, we expand the differential operator $\det(P_d^l)$. This produces a sum over pairs $(S,T)$ partitioning $[2n]\setminus\{l,d\}$ into two sets
of equal cardinality. Here $S$ corresponds to choosing rows of $P^{l}_{d}$ giving holomorphic derivatives $\frac{\partial }{\partial p_{j}}$
and $T$ to choosing anti-holomorphic derivatives $\frac{\partial }{\partial \overline{p_{i}}}$.
After the Gaussian integration, the $p$-dependence occurs only through the quadratic form $\overline p^T A L^{-1}A^T p$.
Thus applying the $p$- and $\overline p$-derivatives and setting $p,\overline{p} = 0$ 
amounts to pairing the indices in $S$ with the indices in $T$ in all possible ways. 
Since these pairings do not carry alternating signs, the result is a permanent:
\begin{equation}\label{eq:SketchPermanent}
	\lim_{p \to 0} \prod_{i \in T} \frac{\partial }{\partial \overline{p_{i}}}  \prod_{j \in S} \frac{\partial }{\partial p_{j}} 
	\frac{1}{(\overline{p}^{T} A L^{-1} A^{T} p + \alpha_{l})^{q-n} }
	= \frac{(-1)^{n-1} \Gamma(s-1)}{\Gamma(s-n) \alpha_{l}^{s-1}}
	\perm(A L^{-1} A^{T})_{S,T}
\end{equation}
Here, $(-)_{S,T}$ denotes the submatrix with rows indexed by $S$ and columns indexed by $T$.

Substituting this identity into \cref{eq:SketchAfterComplexInt}, analytically continuing to $s=2$, and changing variables
$\alpha_i \mapsto 1/x_i$, we obtain an expression in terms of minors of the inverse graph matrix  $M^{-1}$.
More precisely, for the index sets $S,T$ appearing here, one has $(A L^{-1} A^{T})_{S,T} = (M^{-1})_{S,T}$ .
Thus the expression for $I_{\aux}(A)$ becomes
\begin{equation}\label{eq:SketchPermDet}
	I_{\aux}(A) = \int_{\sigma_{2n}} \sum_{l=1}^{2n} \sum_{S,T \in \mathcal{I}_{d}^{l}} 
	\pm\det((M^{-1})_{Sl, Tl})  \perm((M^{-1})_{S,T}) \frac{\Omega(x)}{x_{d}}
\end{equation}
Here $\mathcal{I}_{d}^{l} = \{(S,T) \in ([2n] \setminus \{l,d\})^2 \mid \abs{S} = \abs{T} = n-1 \text{ and } S \cap T = \emptyset\}$.

At this stage the analytic part of the argument is finished. What remains is
the algebraic identification of the determinant-permanent expression above
with the canonical integrand. 
The required combinatorial argument allows us to rewrite \cref{eq:SketchPermDet} as
\begin{equation}\label{eq:SketchCanonicalFinal}
	I_{\aux}(A) = \int_{\sigma_{2n}} \pm \frac{\Omega(x)}{x_{d}}
 \sum_{\tau \in \mathbb S_{2n-1}} \sgn(\tau) M^{-1}_{i_{\tau(1)}i_{\tau(2)}} M^{-1}_{i_{\tau(2)}i_{\tau(3)}}
		\cdots M^{-1}_{i_{\tau(2n-2)}i_{\tau(2n-1)}} M^{-1}_{i_{\tau(2n-1)}i_{\tau(1)}}
	\end{equation}
It remains to recognise the right-hand side as the canonical integral. For
this, we use the standard trace formula for the canonical form. Since the
canonical form associated to $L$ agrees with the canonical form associated to
$M$, it is enough to apply the trace formula to $M$. The form
$\tr((M^{-1}\dd M)^{2n-1})$ is the trace of the $(2n-1)$-st power of the matrix with entries $M^{-1}_{ij}\dd{x_j}$.
Expanding this trace gives a signed sum over cyclic products of the entries of $M^{-1}$.
Factoring out the projective volume form $\Omega(x)$ gives precisely the integrand of  \cref{eq:SketchCanonicalFinal},
and hence $I_{\aux}(A)=I_{\can}(A)$.

\subsection{Comments on the difficulty of the proof}
The preceding argument explains the structure of the proof, but it suppresses
several substantial analytic points. First, the integral $I_{\aux}$ over
$\C^n \times \sigma_{2n}$ is not absolutely convergent. In the full
proof, this is handled by truncating the simplex: one replaces
$\sigma_{2n}$ by the region on which all coordinates satisfy
$\alpha_i>\epsilon$, for some $\epsilon>0$. The relevant transformations
are then performed at fixed $\epsilon$, and only afterwards at the two ends of the argument, where the resulting
integrals are well defined for $\epsilon = 0$, is the limit $\epsilon \to 0$ taken.

Furthermore, for the scale integration on the RW-integral side to be well defined, one must introduce a counterterm in the auxiliary integral. 
This counterterm is obtained by replacing $l$ with $d$, while in both the original term and the counterterm the sum over $l$ is restricted to $l \neq d$.

Finally, one has to justify the repeated interchange of limits, derivatives
and integrals. These technical points account for a significant part of the
length of the full proof.

\subsection{Notation}\label{sec:notation}
Before we come to the main sections of this work we need to introduce some notations.
Unless stated otherwise, $x$, $\alpha$, and $z$ denote vectors of coordinates 
\[
	[x_1:\ldots :x_{n}] \in \RP^{n-1}, \quad [\alpha_1 : \ldots : \alpha_{n}] \in \RP^{n-1} \qq{and} 
	z_1,\ldots,z_{n} \in \C^{n}
\]
for some fixed $n \in \N$.
For homogeneous coordinates of $\RP^{n-1}$ (such as $x$ and $\alpha$) and for complex coordinates $z \in \C^{n}$, we define 
the (projective) volume forms
\[
	\Omega(x) = \sum_{i=1}^{n} (-1)^{i-1} x_{i} \dd{x_{1}} \wedge \ldots \wedge \widehat{\dd{x_{i}}} \wedge \ldots \wedge \dd{x_{n}}
	\qq{and} \Omega(z) = \bigwedge_{i=1}^{n} \dd{z} \wedge \dd{\overline{z}}
.\]
We denote by $\diag(x)$ or $\diag(x_1,\ldots,x_{n})$ the diagonal matrix associated to the elements $x_1,\ldots,x_{n}$.
The permanent of a $n \times n$ matrix $A = (a_{ij})$ is defined by
\[
\perm(A) = \sum_{\sigma \in \mathbb{S}_{n}} a_{1\sigma(1)} \ldots a_{k \sigma(k)}
.\] 
where $\mathbb{S}_{n}$ denotes the symmetric group on $n$ elements.
The adjugate of an invertible matrix $A$, that is the matrix $\det(A) \cdot A^{-1}$, is denoted by $\adj(A)$.

We write $[n]$ for the ordered set $(1,\ldots,n)$. For two ordered sets $S = (s_1,\ldots,s_{k})$ and $T = (t_1,\ldots,t_{l})$, we
define their concatenation by
\[
S T := (s_1,\ldots,s_{k},t_1,\ldots,t_{l})
.\] 
By a slight abuse of notation, we write $S  d$ instead of $S (d)$ when appending a single element $d$.
For an ordered set $S$, we denote by $S \setminus \{s_{i_1},\ldots,s_{i_{j}}\}$ the ordered subset obtained by removing
the indicated elements, with the induced ordering inherited from $S$.

Since the elements of $S,T \subseteq \N$ are themselves ordered, 
each such set carries two natural orders: the given order and the standard increasing order on $\N$.
For an ordered subset $A \subseteq \N$ of cardinality $n$, let $\pi \in \mathbb{S}_{n}$  be the permutation that reorders the elements of $A$ in
increasing order. We define $\sgn(A) := \sgn(\pi)$.
For $x \in A$, we write $A_{< x} := \{a \in A \mid a < x\}$ for the set of elements of $A$ smaller than $x$.
For $a,b \in [n]$ we also use the notation $(a < b)$ for the indicator function, which equals $1$ if $a < b$ and $0$ otherwise.

Let $B$ be a $m \times  n$ matrix and let $S \subseteq [m], T \subseteq [n]$ be ordered subsets.
We define
\[
	B_{S,T} = (b_{ij})_{i \in S, j \in T} \qq{and}
	B^{S,T} = (b_{ij})_{i \in [m] \setminus S, j \in [n] \setminus T}
\] 
Thus, $B_{S,T}$ denotes the submatrix of $B$ whose rows are indexed by $S$ and whose columns are indexed by $T$, taken in their prescribed orders.
Similarly, $B^{S,T}$ denotes the submatrix obtained by deleting the rows indexed by $S$ and columns indexed by $T$.
The notation $B_{S,T}$ is intended to generalise the entry notation $b_{ij}$: the subscripts indicate the rows and columns that determine the submatrix, 
while the superscripts in $B^{S,T}$ indicate the rows and columns that are omitted.

To indicate selection or removal of only rows or only columns, we use a placeholder symbol $\bullet$. Specifically,
\[
	\begin{array}{ll}
		B_{S,\bullet} = (b_{ij})_{i \in S, j \in [n]} \qquad 
					  &B_{\bullet,T} = (b_{ij})_{i \in [m], j \in T} \\
B^{S,\bullet} = (b_{ij})_{i \in [m] \setminus S, j \in [n]} \qquad
			  &B^{\bullet,T} = (b_{ij})_{i \in [m], j \in [n] \setminus T}
.\end{array}
\]

\graphicspath{{Images/}}

\section{Properties of canonical and RW-Integrals}\label{sec:propsOfInts}

\subsection{RW-Integrals}
In the overview we defined the generalised RW-Integrals as follows:
\begin{definition}\label{def:RWIntegral}
	Let $n \geq 3$ and  consider a real  $2n \times n$ matrix $A \in \R^{2n \times n}$ of rank $n$, 
	such that no row of $A$ is equal to $0$. Denoting the rows of $A$ by $a_{i}$, we obtain $2n$ linear forms
	$\la_{i}: \C^{n} \to \C$ given by $z \mapsto \la_{i}(z) := a_{i} z$.
	We define the RW-integral associated with $A$ by
	\[
	I_{\RW}(A) = \frac{(2i)^{2n-1}}{2 (2 \pi i)^{n}} \int_{C_{A}} \sum_{l=1}^{2n} (-1)^{l-1} \log(\abs{\la_{l}(z)}^2) \bigwedge_{e \neq l} \darg(\la_{e}(z))
	,\] 
	where, $C_{A} = \conf_{A}(\C) / \R_{> 0}$ and $\conf_{A}(\C) := \{z \in \C^{n} \mid \la_{i}(z) \neq 0 \forall i \in \{1,\ldots,2n\} \}$.
	We equip $C_{A}$ with the orientation induced from $\iota_{R} \Omega(z)$, where $\Omega(z)$ is the standard volume form on  $\C^{n}$ 
	and $R$ is the radial vector field
	\[
		R = \sum_{i=1}^{n} z_{i} \frac{\partial }{\partial z_{i}} + \overline{z_{i}} \frac{\partial }{\partial \overline{z_{i}}} 
	.\] 
	Since $\iota_{R} \Omega$ is rescaled by a positive function under the $\R_{> 0}$ action,
	it induces a well-defined orientation on the quotient $C_{A}$. 
	Thus fixing our convention for integration.
\end{definition}

To show that the integral is indeed well-defined we need to show that the integrand is defined on $C_{A}$, that
is that it is scale invariant. For this we need the following lemma:

\begin{lemma}\label{lem:rotationInvariance}
	The differential form $\omega = \sum_{l=1}^{2n} (-1)^{l-1} \bigwedge_{e \neq l} \darg(\la_{e})$ on $\conf_{A}(\C)$ vanishes identically.
\end{lemma}
\begin{proof}
	Define the vector field $X$ and the top-degree form $\alpha$ on $\conf_{A}(\C)$ as
	\[
	X = i\sum_{i=1}^{n} z_{i} \frac{\partial }{\partial z_{i}} -\overline{z_{i}} \frac{\partial }{\partial \overline{z_{i}}}, 
	\qquad \alpha = \bigwedge_{e} \darg(\la_{e})
	.\] 
	where $X$ is the vector field generating the diagonal $S^{1}$-action.
	Then $\iota_{X} \darg(\la_{e}) = 1$ and thus $\omega = \pm\iota_{X} \alpha$.

	Now let $R$ be the vector field generating real scaling. Since $\darg(\la_{e})$ is invariant under real scaling,
	it follows that $\iota_{R} \alpha = 0$.
	Since $R$ is nowhere vanishing and $\alpha$ has top degree on $\conf_A(\C)$, contraction with
	$R$ is injective on top-degree forms. Therefore $\alpha=0$, and so $\omega = \pm \iota_{X} \alpha = 0$.
\end{proof}
This allows us to show the scale invariance:
\begin{corollary}\label{cor:scaleInvariance}
	Let $\lambda \in \R_{> 0}$ and let $\Phi_{\lambda}(z)\colon  C_{A} \to  C_{A}, z \mapsto \lambda z$ be the scaling action,
	then
	\[
	\Phi_{\lambda}^{*} \omega_{\RW}(A) = \omega_{\RW}(A)
	.\] 
\end{corollary}

\begin{proof}
	Using the fact that $\darg(\lambda \la_{e}) = \darg(\la_{e})$ and \cref{lem:rotationInvariance} we calculate:
\begin{align*}
		\Phi_{\lambda}^{*} \omega_{\RW}(A) &= \sum_{l=1}^{2n} (-1)^{l-1} \log(\lambda^2 \abs{\la_{l}}^2) \bigwedge_{e \neq l} \darg(\la_{e})\\
	&=\omega_{\RW}(A) + \log(\lambda^2) \sum_{l=1}^{2n} (-1)^{l-1} \bigwedge_{e \neq l} \darg(\la_{e}) = \omega_{\RW}(A)
,\end{align*}
which proves the corollary.
\end{proof}

\begin{remark}[Invariance properties]\label{rem:RWGLnInvariance}
	The RW-integrals $I_{\RW}(A)$ enjoy several natural invariance properties:
	\begin{enumerate}
		\item Right $\GL_{n}(\R)$-invariance: for $P \in \GL_{n}(\R)$, we have $I_{\RW}(AP) = I_{\RW}(A)$.
			Indeed, writing $\omega(A)$ for the integrand in \cref{def:RWIntegral}, we compute
	\[
	I_{\RW}(AP) = \int_{C_{AP}} \omega(AP)(z) = \int_{C_{AP}} \omega(A)(Pz) = \int_{C_{A}} \omega(A)(y) = I_{\RW} (A)
	.\] 
	where we used the change of variables $y = P z$ and $(a_{i} P) z = a_{i} y$.
\item Dependence on the ordering: the value of $I_{\RW}(A)$ depends on the ordering of the rows of $A$ only up to a sign,
	corresponding to the permutation of differential forms in the wedge product.
\item Row rescaling: the integral is invariant under rescaling of the rows of $A$, i.e. for $D$ a diagonal $2n \times 2n$ matrix,
	we have $I_{\RW}(D A) = I_{\RW}(A)$. 
	Indeed, let $d_1,\ldots,d_{2n}$ be the diagonal entries of $D$. Then $C_{D A} = C_{A}$, as the conditions $\la_{i}(z) = 0$ are invariant
	under rescaling. Since $(D A)_{i}(z) = d_{i} \la_{i}(z)$ we have $\log(\abs{(D A)_{l}}^2) = \log(\abs{d_{l}}^2) + \log(\abs{\la_{l}}^2)$
	and $\darg((D A)_{i}(z)) = \darg(\la_{i}(z))$. Hence,
	\[
		I_{\RW}(D A) = I_{\RW}(A) + \sum_{l=1}^{2n} (-1)^{l-1} \log(\abs{d_{l}}^2)  \int_{C_{A}} \bigwedge_{e\neq l} \darg(\la_{e}) = I_{\RW}(A)
	.\] 
	Here the additional terms vanish by the Kontsevich vanishing lemma \cite[Section 6.6.1]{kontsevich03}.
\end{enumerate}
\end{remark}

Next, we show that the RW-integral as defined in \cref{sec:introduction,def:RWIntegral} agree. 
The representation from \cref{sec:introduction} is often more suited
for explicit computations and for establishing vanishing properties, whereas the form of \cref{def:RWIntegral} is slightly more natural
and needed for the proof of \cref{thm:EqualityOfIntegrals}.
\begin{proposition}\label{prop:ogRWIntegral}
	Let $Z_{A} = \conf_{A}(\C) / \C^{*}$, where $\C^{*}$ acts diagonally by rotation and scaling and $Z_{A}$ is equipped with the canonical
	orientation induced from $\C^{n-1}$. Then
	\[
	I_{\RW}(A) = \frac{(-1)^{n-1}}{(2 \pi i)^{n-1}} \int_{Z_{A}} \sum_{\substack{l,d\\ l \neq d}} (-1)^{l + d + (l<d)} \log(\abs{\la_{l}}^2) 
	\bigwedge_{e \neq l,d} \dlog(\abs{\la_{e}}^2)
	,\] 
\end{proposition}

\begin{proof}
	Recall the formula for $I_{\RW}$ as in \cref{def:RWIntegral}. 
	Let $\pi: C_{A} \to C_{A} / S^{1}$ denote the canonical projection.
	Since  $C_{A} / S^{1}$ is canonically isomorphic to $Z_{A} = \conf_{A} / \C^{*}$, and $\pi$ is a trivial $S^{1}$-bundle,
	we may integrate along the fibres. Writing $\dd{\varphi}$ for the angular form on $S^{1}$, we obtain
	\begin{align*}
		I_{\RW}(A) &= \frac{(2i)^{2n-1}}{2 (2 \pi i)^{n}} 
		\int_{C_{A}} \sum_{l=1}^{2n} (-1)^{l-1} \log(\abs{\la_{l}}^2) \bigwedge_{e \neq l} (\dd{\varphi} + \darg(\la_{e}))\\
		&= \frac{(2i)^{2n-1}}{2(2 \pi i)^{n}} \int_{Z_{A}} \int_{S^{1}} \dd{\varphi} 
		\sum_{\substack{l,d\\ l \neq d}} (-1)^{l + d + (l<d)} \log(\abs{\la_{l}}^2) 
	\bigwedge_{e \neq l,d} \darg(\la_{e})
.\end{align*}
where we used absolute convergence (\cref{thm:prodToPos}) to justify the integration along the fibre.

Integrating $\dd{\varphi}$ over the fibre yields a factor of $2\pi$, and hence
	\begin{equation}\label{eq:RWdargToLog}
		I_{\RW}(A) = \frac{(2i)^{2n-2}}{(2 \pi i)^{n-1}} \int_{Z_{A}} \sum_{\substack{l,d\\l \neq d}} (-1)^{l + d + (l<d)}
		\log(\abs{\la_{l}}^2) \bigwedge_{e \neq l,d} \darg(\la_{e})
	.\end{equation}
	To replace the $\darg$-forms by $\dlog$-forms, we use a standard trick due to Kontsevich \cite[Section 6.6.1]{kontsevich03}:
	\begin{align*}
		\bigwedge_{e \neq l,d} \darg(\la_{e}) &= \frac{1}{(2 i)^{2n-2}} \bigwedge_{e \neq l,d} \left( \dlog(\la_{e}) - \dlog(\overline{\la_{e}}) \right) \\
	&= \frac{(-1)^{n-1}}{(2 i)^{2n-2}} \bigwedge_{e \neq l,d} \left( \dlog(\la_{e}) + \dlog(\overline{\la_{e}}) \right) 
	= \frac{(-1)^{n-1}}{(2 i)^{2n-2}} \bigwedge_{e \neq l,d} \log(\abs{\la_{e}}^2)
.\end{align*}
In the second equality, we use that only terms with $n-1$ holomorphic and $n-1$ anti-holomorphic differentials contribute, 
as we are considering a top-degree form on an $(n-1)$-dimensional complex manifold.

	Substituting this identity into \cref{eq:RWdargToLog} yields the desired expression and completes the proof.
\end{proof}

We will now sketch the proof of the weight and values the RW-integrals take in the case of graphs, as explained in \cref{sec:svAndWeight}, that is:
\begin{proposition}\label{prop:RWSingleValuedWeight}
	Let $G$ be a graph with $n+1$ vertices and $2n$ edges. Then, the RW-integral and thus the canonical integral, evaluates to single-valued
	MZVs of weight $n$.
	\[
	I_{\can}(G) =  I_{\RW}(G) \in \mathcal{Z}^{\sv}_{n}
	.\] 
\end{proposition}

The proof sketch is based on the theory of \cite{vanhove22,banks20}.
\begin{proof}[Sketch of proof]
Let $X_{k} = \{0,1,z_2,\ldots,z_{k}\}$ denote both a set of points in $\C$ and the corresponding alphabet.
We write $\conf_{k}(\C) = \{z \in \C^{k-1} \mid z_{i} - z_{j} \neq 0 \text{ for } z_{i},z_{j} \in X_{k}\}$ for the configuration space.
Let $\mathcal{A}^{d}(X_{k})$ be the $\Q$-vector space of single-valued $d$-forms on $\conf_{k}(\C)$ generated by wedges 
of $\dlog(z_{i}-z_{j})$ and $\dlog(\overline{z_{i}} - \overline{z_{j}})$ for $z_{i},z_{j} \in X_{k}$. 

Let $\mathcal{V}_{w}(X_{k})$ be the $\Q$-vector space of single-valued polylogarithms of weight $w$ \cite{brown04single}. Is is
generated by products $\mathcal{L}_{v_1}(a_1) \ldots \mathcal{L}_{v_{r}}(a_{r})$ where $v_{i}$ is a word in the alphabet $X_{k}$, $a_{i} \in X_{k}$,
and the weight is the sum of the word lengths.
We set $\mathcal{U}_{w}^{d} := \mathcal{A}^{d} \otimes \mathcal{V}_{w}$.

The key property is that these spaces are stable under integration in the following sense: 
If $\omega \in \mathcal{U}_{w}^{2k}(X_{k})$ is a top-degree form, then there exists a primitive $\alpha \in \mathcal{U}_{w}^{2k-1}(X_{k})$, 
with $\dd{\omega} = \alpha$, which has bidegree $(0,1)$ in $\dd{z_{k}} \wedge \dd{\overline{z_{k}}}$.
Schnetz's residue theorem \cite{schnetz13} gives:
\[
	\frac{1}{2 \pi i} \int_{z_{k} \in \C \setminus X_{k-1}} \omega = \sum_{j=1}^{k-1} \operatorname*{Res}_{\overline{z_{k}} = \overline{z_{j}}} \alpha 
	\in \mathcal{U}_{w+1}^{2k-2}(X_{k-1})
.\]
Thus integrating out one complex variable lowers the degree of the form by $2$ and raises the polylogarithmic weight by $1$.

We now apply this to the RW-integral representation from \cref{prop:ogRWIntegral}, 
where we have identified $Z_{A}$ with $\C^{n-1}$ by fixing $z_1 = 1$.
The integrand lies in $\mathcal{U}_{1}^{2n-2}(X_{n})$, since $\log(\abs{z_{i} - z_{j}}^2) = \mathcal{L}_{z_{j}}(z_{i}) - \mathcal{L}_{0}(z_{j})$ 
is a single-valued
hyperlogarithm of weight $1$.

Integrating successively the variables $z_{n},z_{n-1},\ldots,z_{2}$, using the primitive and residue formula at each step gives 
a sequence of integrands in the following spaces
\[
	\mathcal U^{2n-2}_1(X_n)
	\longrightarrow
	\mathcal U^{2n-4}_2(X_{n-1})
	\longrightarrow \cdots \longrightarrow
	\mathcal U^0_n(X_1).
\]
The final result is therefore a single-valued polylogarithm of weight $n$ in the alphabet $\{0,1\}$, evaluated at $0$ or $1$. 
Since the evaluations at $0$ vanish, the remaining terms are single-valued polylogarithms in the alphabet $\{0,1\}$ evaluated at $1$. 
These are precisely the single-valued MZVs. Hence the RW-integral lies in the space of single-valued MZVs of weight $n$.
\end{proof}

Finally, we record three vanishing properties of the RW-Integrals:
\begin{proposition}\label{prop:RWvanish}
	The RW-integral $I_{RW}(A)$ vanishes whenever one of the following conditions holds:
	\begin{enumerate}[a)]
		\item $n$ is even,
		\item there exists $i \in [2n]$ such that $\rank(A^{i,\bullet}) < n$.
		\item there exists $j \in [n]$ such that $A^{\bullet,j}$ is in block-diagonal form after a suitable permutation of rows and columns.
	\end{enumerate}
\end{proposition}

\begin{proof}
	We first prove a).
	Complex conjugation changes the orientation of the $(n-1)$-dimensional complex manifold $Z_{A} = \conf_{A}(\C) / \C^{*}$  by
	$(-1)^{n-1}$, while leaving the integrand invariant. Hence,
	\[
	I_{\RW}(A) = (-1)^{n-1} I_{\RW}(A)
	,\] 
	and a) follows trivially.

	We now prove b). Suppose that $\rank(A^{i,\bullet}) < n$.  Then the rows $a_1,\ldots,a_{i-1},a_{i+1},\ldots,a_{2n}$
	span a vector space of dimension at most $n-1$. Consequently, there exists a matrix $P \in \GL_{n}(\R)$ such that the first column of 
	$A^{i,\bullet} P$ vanishes identically. Define $\widetilde{A} := A P$.

	By the $\GL_{n}$-invariance of the RW-integral (\cref{rem:RWGLnInvariance}), we have $I_{\RW}(\widetilde{A}) = I_{\RW}(A)$.
	Take the form of the RW-integral from \cref{prop:ogRWIntegral}, and 
	identify $Z_{\widetilde{A}}$ with 
	$Y_{\widetilde{A}} := \{z \in \C^{n-1} \mid \widetilde{\la}_{i}(z,1) \neq 0 \forall i \in [2n]\}$
	by fixing the affine chart $z_{n } = 1$. We obtain
	\[
	I_{\RW}(A) = \frac{(-1)^{n-1}}{(2 \pi i)^{n-1}} \int_{Y_{\widetilde{A}}} \sum_{\substack{l,d\\ l \neq d}} (-1)^{l + d + (l<d)} 
	\log(\abs{\widetilde{\la}_{l}(z,1)}^2) 
	\bigwedge_{e \neq l,d} \dlog(\abs{\widetilde{\la}_{e}(z,1)}^2)
	.\]
	By construction, the first coordinate appears only in the linear form $\widetilde{a}(z,1)$. Hence, among all differential forms
	$\dlog(\abs{\widetilde{\la}_{e}(z,1)}^2)$, only the term corresponding to $e = i$ can contain $\dd{z_1}$ or $\dd{\overline{z_1}}$.
	Therefore, any term appearing in the integrand contains at most one of $\dd{z_1}$ and $\dd{\overline{z_1}}$, but never both. 
	Since a nonzero top-degree form on the complex manifold $Y_{\widetilde{A}} \subseteq \C^{n-1}$ must contain both, 
	the integrand vanishes identically.

	Finally, we prove c). By multiplying $A$ on the right by a suitable permutation matrix $P \in \GL_{n}(\R)$,
	we may assume that the $n$-th column is the one whose removal yields a block-diagonal matrix.
	Thus, we may write
	\[
	A = \begin{pmatrix} 
		U &0 &*\\
		0 &V &*
	\end{pmatrix} 
	.\] 
	where $U \in \R^{i\times j}$ and $V \in \R^{(2n-i) \times (n-1-j)}$.
	In particular, the linear forms $\boldsymbol{u}_{k}: \C^{j} \to \C$ corresponding to the rows of $U$ depend only on the variables $z_1,\ldots,z_{j}$, 
	while the linear forms $\boldsymbol{v}_{k}: \C^{n-1-j} \to \C$ corresponding
	to the rows of $V$ depend only on $z_{j+1},\ldots,z_{n-1}$.

	Using the form of the RW-integral from \cref{prop:ogRWIntegral} and identifying 
	$Z_{A}$ with the affine chart $Y_{A}$ as in the proof of b), we see that $Y_{A}$ decomposes as a product
	$Y_{A} \cong X_{U} \times X_{V}$ where
	\begin{align*}
		X_{U} &:= \{z \in \C^{j} \mid \boldsymbol{u}_{k}(z) \neq 0 \forall k \in \{1,\ldots,i\}\}\\
		X_{V} &:=\{z \in \C^{n-1-j} \mid \boldsymbol{v}_{k}(z,1) \neq 0 \forall k \in \{i+1,\ldots,2n\}\}
	.\end{align*}
	Consequently, the integrand splits into a product of a form depending only on the variables for $U$ ($z_1,\ldots,z_{j}$), and a form
	depending only on the variables for $V$ ($z_{j+1},\ldots,z_{n-1}$), that is
	\begin{align*}
		I_{\RW}(A) &= \sum_{d = 1}^{2n} \sum_{\substack{l = 1\\ l \neq d}}^{i} \pm
		\Bigg( \int_{X_{U}} \log(\abs{\la_{l}}^2) \bigwedge_{\substack{e=1\\ e \neq l,d}}^{i} \darg(\la_{e}) \Bigg)
		\Bigg( \int_{X_{V}} \bigwedge_{\substack{e=i+1\\ e \neq d}}^{2n} \darg(\la_{e}) \Bigg)\\
		&+  \sum_{d = 1}^{2n} \sum_{\substack{l = i\\ l \neq d}}^{2n} \pm
		\Bigg( \int_{X_{U}} \bigwedge_{\substack{e=1\\ e \neq d}}^{i} \darg(\la_{e}) \Bigg)
		\Bigg( \int_{X_{V}} \log(\abs{\la_{l}}^2) \bigwedge_{\substack{e=i+1\\ e \neq l,d}}^{2n} \darg(\la_{e}) \Bigg)
	.\end{align*}
	By the Kontsevich vanishing lemma \cite[Section 6.6.1]{kontsevich03} it follows that the integral over $X_{V}$ in the first sum, and the integral
	over $X_{U}$ in the second sum vanishes, and hence $I_{\RW}(A) = 0$. 
\end{proof}

\begin{remark}\label{rem:RWTwoVertexCut}
	In the case of the RW-integral associated to a graph $G$, c) implies that the integral vanishes whenever $G$ contains a two-vertex cut.
	This also follows from \cite[Lemma 7.5]{rossi14}.

	Indeed, let $G$ be a graph with a two-vertex cut given by vertices $u$, $v$. Choose an incidence matrix  $\epsilon$ of $G$ such that the first
	$k$ edges belong to one connected component of $G \setminus \{u,v\} $ and the remaining edges belong to the other component.
	Moreover, assume that the columns corresponding to $u$ and $v$ are the first and second columns of $\epsilon$, respectively.

	Now consider the reduced incidence matrix $\widetilde{\epsilon} = \epsilon^{\bullet,1}$. 
	Then the RW-integral of $G$ is given by the RW-integral associated to $A = \widetilde{\epsilon}$. 
	Since the first $k$ rows of $\epsilon$ correspond to
	one connected component of $G \setminus \{u,v\} $, and the remaining rows correspond to the other, it follows that $\widetilde{\epsilon}^{\bullet,1}$ 
	has block-diagonal form. Hence, by \cref{prop:RWvanish}~c), we conclude that $I_{\RW}(G) = I_{\RW}(\widetilde{\epsilon}) = 0$.
\end{remark}

\subsection{Canonical Integrals}
In the overview we defined the canonical integral as follows:
\begin{definition}\label{def:canonicalInt}
	Let $n \geq 3$, and let $A \in \R^{2n \times n}$ of rank $n$, such that no row of $A$ is equal to  $0$.
	Further, let $\sigma_{2n} \subseteq \RP^{2n-1}$ denote the simplex with homogeneous coordinates $x = (x_1,\ldots,x_{2n})$, satisfying $x_{i} \geq 0$.
	$L_{A}(x):  \sigma_{2n} \to \mathcal{P}^{\geq 0}_{n}$ given by $L_{A}(x) = A^{T} \diag(x) A$
	from the simplex $\sigma_{2n}$ into the space $\mathcal{P}_{n}^{\geq 0}$ of positive semidefinite symmetric $n \times n$ matrices.
	The canonical integral and the canonical form are defined by
	\[
		I_{\can}(A) = \int_{\sigma_{2n}} \beta^{2n-1}_{L} \qq{and} \beta^{2n-1}_{L} = \tr((L^{-1} \dd{L})^{2n-1})
	.\] 
	We orient $\sigma_{2n}$ by requiring that the projective volume form $\Omega(x)$ evaluates positively 
	on the ordered frame $\partial_{2} \wedge \ldots \wedge \partial_{2n}$.
\end{definition}

\begin{remark}
	The canonical forms are invariant under the right action of $\GL_{n}(\R)$ on $A$, given by $A \mapsto A P$ for $P \in \GL_{n}(\R)$.
	On the level of $L$, this corresponds to invariance under conjugation $L \mapsto P^{T} L P$.
	In fact, this invariance characterises the canonical forms as the unique primitive $\GL_{n}(\R)$-invariant differential forms up to scalar.

	Moreover, the same canonical forms arise from different matrix representations associated to graphs.
	For example, if $M$ denotes the graph matrix from \cref{def:graphMatrices}, then \cite[Prop.~6.23]{brown21} shows that
	$\beta^{2n-1}_{L} = \beta^{2n-1}_{M}$.
\end{remark}

The expression for the canonical form as the trace of a high power of $L^{-1} \dd{L}$ is not particularly well suited for explicit computations, 
since the variables and differentials are heavily intertwined. Using Dodgson polynomials (see \cref{app:Dodgson}), 
one can instead derive a representation as a polynomial times a volume form:
\begin{proposition}\label{prop:canonicalDodgsonFormula}
	Let $d \in [2n]$ and let $I_{d} = (i_1,\ldots,i_{2n-1}) = I \setminus \{d\} $ with the induced ordering.
	Then the canonical form can be written as
	\[
		\beta_{L}^{2n-1} = (-1)^{d+1} \frac{f_{d}(x)}{x_{d}} \Omega(x)
	.\]
	with
	\[
	f_{d}(x) = \frac{1}{\Psi^{2n-1}} \sum_{\tau \in \mathbb{S}_{2n-1}} \sgn(\tau) 
	\Psi^{i_{\tau(1)},i_{\tau(2)}} \Psi^{i_{\tau(2)},i_{\tau(3)}} \ldots \Psi^{i_{\tau(2n-1)},i_{\tau(1)}} 
	.\] 
	where $\Psi^{i,j}$ denotes the Dodgson polynomial associated to $L$ and $\Psi$ is the graph polynomial (\cref{def:graphMatrices}).
\end{proposition}

\begin{proof}
	By  \cite[Corollary 6.24]{brown21} the canonical form $\beta_{L}^{2n-1}$ is equal to $\tr(\eta_{L}^{2n-1})$ where 
	$\eta_{L}$ is defined as the $2n \times 2n$-matrix of $1$-forms given by
	\[
		(\eta_{L})_{ij} = \left( \frac{\Psi^{ij}}{\Psi} \dd{x_{j}} \right), \quad 1 \leq i \leq j \leq 2n
	.\] 
	Expanding the matrix product yields
	\begin{equation} \label{eq:etaExpansion}
		\beta_{L}^{2n-1} = \tr(\eta_{L}^{2n-1}) = \frac{1}{\Psi^{2n-1}} 
		\sum_{k_1,\ldots,k_{2n-1} = 1}^{2n} \Psi^{k_1,k_2} \Psi^{k_2,k_3} \ldots \Psi^{k_{2n-1},k_1} 
		\dd{x_{k_2}} \wedge \ldots \wedge \dd{x_{k_{2n-1}}} \wedge \dd{x_{k_1}}
	.\end{equation}
	Since $\dd{x_{k_{l}}} \wedge \dd{x_{k_{m}}} = 0$ for $l = m$, only terms for which the indices $k_1,\ldots,k_{2n-1}$ are pairwise distinct
	contribute to the sum.
	As there are exactly $2n-1$ such indices, choosing values for $k_1,\ldots,k_{2n-1}$ is equivalent to choosing one element
	$j \in [2n]$ to omit and then choosing a permutation of the remaining elements $I _{j} = [2n] \setminus \{j\}$.
	Thus, \cref{eq:etaExpansion} becomes
	\[
		\beta_{L}^{2n-1} = \frac{1}{\Psi^{2n-1}} \sum_{j \in [2n]}
		\sum_{\tau \in \mathbb{S}_{2n-1}} \Psi^{i_{\tau(1)},i_{\tau(2)}} \Psi^{i_{\tau(2)},i_{\tau(3)}} \ldots \Psi^{i_{\tau(2n-1)},i_{\tau(1)}} 
		\dd{x_{i_{\tau(1)}}} \wedge \ldots \wedge \dd{x_{i_{\tau(2n-1)}}}
	.\]
	Reordering the differential forms into increasing order $\dd{x_{i_1}} \wedge \ldots \wedge \dd{x_{i_{2n-1}}}$ produces precisely
	the sign $\sgn(\tau)$. Hence
	\begin{equation}\label{eq:canEulerField}
		\beta_{L}^{2n-1} = \sum_{j \in [2n]} f_{j}(x) 
		\dd{x_{1}} \wedge \cdots \wedge \widehat{\dd{x_{j}}} \wedge  \cdots \wedge \dd{x_{2n}}
	.\end{equation}
	Finally, \cite[Theorem 2.1]{brown21} shows that $\beta_{L}^{2n-1}$ is a projective differential form.
	Therefore, its contraction with the Euler vector field $E = \sum_{i=1}^{2n} x_{i} \frac{\partial }{\partial x_{i}}$ vanishes,
	\begin{align*}
		\iota_{E} \beta_{L}^{2n-1} 
		= \sum_{\substack{j,k \in [2n]\\ k < j}} \left( (-1)^{j} x_{j} f_{k}(x) - (-1)^{k} x_{k} f_{j}(x) \right)
		\bigwedge_{i \neq j,k} \dd{x_{i}} = 0
	,\end{align*}
	and $(-1)^{k} \frac{f_{k}(x)}{x_{k}} = (-1)^{j} \frac{f_{j}(x)}{x_{j}}$ for all $j,k \in [2n]$.
	We conclude that \cref{eq:canEulerField} is equal to
	\[
	(-1)^{d} \frac{f_{d}(x)}{x_{d}} \sum_{j=1}^{2n} (-1)^{j} x_{j} \dd{x_{1}} \wedge \cdots \wedge \widehat{\dd{x_{j}}} \wedge \cdots \wedge \dd{x_{2n}}
	= (-1)^{d+1} \frac{f_{d}(x)}{x_{d}} \Omega(x). \qedhere 
	\]
\end{proof}

Finally, we record three vanishing properties of the canonical forms:
\begin{proposition}\label{prop:canVanish}
	The canonical form $\beta^{2n-1}_{L}$ vanishes whenever one of the following conditions holds:
	\begin{enumerate}[a)]
		\item $n$ is even,
		\item there exist $i \in [2n]$ such that $\rank(A^{i,\bullet}) < n$,
    	\item $A$ has a repeated row, i.e. there exist $i \neq j$ in $[2n]$ such that $a_i = a_j$.
	\end{enumerate}
\end{proposition}

\begin{proof}
	We first prove a). By Property $3$ of \cite[Lemma 4.3]{brown21},  $\beta^{2n-1}_{L} = (-1)^{n+1} \beta_{L^{T}}^{2n-1}$.
	Since $L$ is symmetric, the vanishing for even $n$ follows trivially.
	
	We now prove b).
	As in the proof of \cref{prop:RWvanish}, after replacing $A$ by $AP$ for a suitable matrix $P \in \GL_{n}(\R)$, 
	we may assume that the first column of $A^{i,\bullet}$ vanishes identically. 
	Then $L$ has block form
	\[
	L = \begin{pmatrix} 
		x_1 &0\\
		0 &R\\
	\end{pmatrix} 
	.\] 
	where $R \in \R[x_2,\ldots,x_{2n}]^{(n-1) \times (n-1)}$ is independent of  $x_1$. Consequently,
	\begin{align*}
		\beta^{2n-1}_{L} = \tr((L^{-1} \dd{L})^{2n-1}) &= \tr(\begin{pmatrix} \frac{\dd{x_1}}{x_1} &0\\ 0 &R^{-1} \dd{R} \end{pmatrix}^{2n-1})\\
		&= \tr(\begin{pmatrix} (\frac{\dd{x_1}}{x_1})^{2n-1} &0\\ 0 &(R^{-1} \dd{R})^{2n-1}\end{pmatrix})
		= \left(\frac{\dd{x_1}}{x_1}\right)^{2n-1} + \beta^{2n-1}_{R} = 0
	.\end{align*}
	where the first term vanished trivially since $\dd{x_{i}} \wedge \dd{x_{i}} = 0$ and the second term also
	vanished because $R$ depends on only $2n-1$ variables, so $\beta_{R}^{2n-1}$ is a $(2n-1)$-form on a space of dimension $2n- 2$.

	Finally, we prove c). By permuting the rows of $A$, we may assume that $a_1 = a_2$, 
	since $\beta_{L}^{2n-1}$ is unaffected except for a relabelling of the coordinates $x_{i}$. Expanding $L$ as a linear function of the coordinates
	gives
	\[
	L = x_1 (a_1^{T} a_1) + \ldots + x_{2n} (a_{2n}^{T} a_{2n})
	= (x_1 + x_2) (a_1^{T} a_1) + x_3 (a_3^{T} a_3) + \ldots + x_{2n} (a_{2n}^{T} a_{2n})
	.\] 
	Let $\widetilde{A}$ be the matrix whose rows are $(a_1,a_3,\ldots,a_{2n})$, and let
	$\widetilde{L} = \widetilde{A}^{T} \diag(x_1,\ldots,x_{2n-1}) \widetilde{A}$.

	Consider the map $\pi: \mathbb{RP}^{2n-1}_{> 0} \setminus \{[1:-1:0:\ldots:0]\}  \to \mathbb{RP}^{2n-2}_{> 0}$ given by
	$(x_1,\ldots,x_{2n}) \to (x_1 + x_2,x_3,\ldots,x_{2n})$.
	Then $L$ is the pullback of $\widetilde{L}$ under $\pi$, and hence
	\[
	\beta^{2n-1}_{L} = \pi^{*} \beta^{2n-1}_{\widetilde{L}}
	.\] 
	Since $\beta^{2n-1}_{\widetilde{L}}$ is a $(2n-1)$-form on the $(2n-2)$-dimensional space $\mathbb{RP}^{2n-2}_{> 0}$,
	it must vanish identically. Therefore, $\beta_{L}^{2n-1} = 0$.
\end{proof}
With these properties of the RW- and canonical integrals established, we can now proceed to prove the main theorems.

\graphicspath{{Images/}}

\section{Proofs of main theorems}\label{sec:proofOfMainTheorems}
We turn to the proof of the main theorem \cref{thm:EqualityOfIntegrals}, in full rigor.
Referring back to \cref{fig:proofSketch}, the proof is divided as follows. In \cref{sec:prodToPos}, we establish the left-hand side of the diagram,
namely the passage from the auxiliary integral to the RW-integral.
In \cref{sec:prodToPara}, we prove the top-arrow on the right-hand side of \cref{fig:proofSketch}, corresponding to the complex integration.
Finally, in \cref{sec:paraToCan}, we prove the final matrix identity needed to match the intermediate result with the canonical form.

In the proof given in this section, we combine these results and introduce the necessary regularising terms to ensure that all integrals
remain convergent. 
Furthermore, we prove \cref{prop:introNoLsum} and show that the auxiliary integral is absolutely convergent, allowing us to integrate in either order.

\begin{theorem}
	Let $n \geq 3$ and let $A \in \R^{2n \times n}$ of rank $n$, such that no row of $A$ is equal to $0$. Then
	\[
		I_{\RW}(A) = I_{\can}(A)
	.\] 
\end{theorem}

\begin{proof}
	By \cref{prop:RWvanish,prop:canVanish} we may restrict to the case where $n$ is odd and $\rank(A_{i,\bullet}) = n$ for all $i \in [2n]$.

	Let $\epsilon > 0$ and define the truncated affine simplices
	\[
		\sigma_{2n}^{\epsilon} := \{\alpha \in \R^{2n}_{\geq \epsilon} \mid \sum_{i=1}^{2n} \alpha_{i} = 1\} 
		\qq{and} \widetilde{\sigma}_{2n}^{\epsilon} := \{x \in \R_{\geq 0}^{2n} \mid 
	x_{i} \leq \frac{1}{\epsilon} \forall  i \in [2n] \text{ and } \sum_{i=1}^{2n} \frac{1}{x_{i}} = 1\}
	.\]
	Then the regularised auxiliary integral is defined as
	\begin{align}
	I_{\aux}^{\epsilon}(A) = \frac{\Gamma(2n)}{(2 \pi i)^{n}}
		\sum_{\substack{l \in [2n]\\ l \neq d}} (-1)^{l+ d + (d<l)} &\left(\int_{\C^{n} \times \sigma_{2n}^{\epsilon}} 
	\frac{\bigwedge_{e \neq l,d} ( \overline{\la_{e}} \dd{\la_{e}} + \la_{e} \dd{\overline{\la_{e}}}) 
		\wedge \dd{\la_{d}} \wedge \dd{\overline{\la_{d}}} }{(\sum_{e \neq l} \alpha_{e} \abs{\la_{e}}^2 + \alpha_{l} \abs{\la_{l}}^2 + \alpha_{l})^{2n}}
		\Omega(\alpha) \right.\notag \\
		- &\left.\int_{\C^{n} \times  \sigma_{2n}^{\epsilon}} \frac{\bigwedge_{e \neq l,d} ( \overline{\la_{e}} \dd{\la_{e}} + \la_{e} \dd{\overline{\la_{e}}})
		\wedge \dd{\la_{d}} \wedge \dd{\overline{\la_{d}}} }{(\sum_{e \neq l} \alpha_{e} \abs{\la_{e}}^2 + \alpha_{l} \abs{\la_{d}}^2 + \alpha_{l})^{2n}}
	  \Omega(\alpha) \right)\label{eq:auxRegInt}
  \end{align}
	Observe that the two integrals differ only in which linear form $\abs{\la_i}^2$ appears as the coefficient of $\alpha_l$. 
	In the first integral this is $\abs{\la_l}^2$, while in the second it is $\abs{\la_d}^2$. 
	Translating this into the language of the matrix $A$, the first term arises from $A$, 
	whereas the second arises from the matrix obtained from $A$ by replacing the $l$-th row with the $d$-th row of $A$.
	To distinguish these cases clearly in the remainder of the proof, for $m \in [2n]$ we denote by $A_m$ the matrix obtained from $A$
	by replacing its $l$-th row with the $m$-th row of $A$. 
	For $m=l$, this is simply the original matrix $A$. 
	Thus, the first integral in \cref{eq:auxRegInt} is associated with $A_l$, and the second with $A_d$.

	All expressions normally derived from $A$ can likewise be defined in terms of $A_{m}$. We denote the Laplacian matrix
	of $A_{m}$ by $L_{m}$, the graph matrix of $A_{m}$ by $M_{m}$, the First Symanzik polynomial of $A_{m}$ by $\Psi_{m}$
	and, for $i,j \in [2n]$, the corresponding Dodgson polynomial associated to $A_{m}$ by $\Psi^{i,j}_{m}$.

	We can now show that, in the limit  $\epsilon \to 0$, $I_{\aux}^{\epsilon}(A)$ is equal to $I_{\RW}(A)$.
	Using the absolute convergence provided by \cref{thm:productSpaceConvergence}, $I_{\aux}^{\epsilon}(A)$ can be rewritten as the integral
	first over $\sigma_{2n}^{\epsilon}$ and then over $\C^{n}$. Applying \cref{thm:prodToPos}, then yields
	\[
	\lim_{\epsilon \to 0} I_{\aux}^{\epsilon}(A) = \frac{(2i)^{2n-1}}{2 (2 \pi i)^{n}}
	\sum_{\substack{l \in [2n]\\ l \neq d}} (-1)^{l-1} \int_{C_{A}} \log\bigg(\frac{\abs{\la_{l}}^2}{\abs{\la_{d}}^2}\bigg) 
	\bigwedge_{e \neq l} \darg(\la_{e})
	.\] 
	Observe next, that the missing summand $l = d$ is equal to $0$ as the logarithm evaluates to $0$. Adding it in we find
	\begin{align*}
		\lim_{\epsilon \to 0} I_{\aux}^{\epsilon}(A) &= \frac{(2i)^{2n-1}}{2 (2 \pi i)^{n}}
	\int_{C_{A}} \sum_{l\in [2n]} (-1)^{l-1} \log\bigg(\frac{\abs{\la_{l}}^2}{\abs{\la_{d}}^2}\bigg)
	\bigwedge_{e \neq l} \darg(\la_{e})\\
	&= \frac{(2i)^{2n-1}}{2 (2 \pi i)^{n}}
	\int_{C_{A}} \sum_{l\in [2n]} (-1)^{l-1} \log( \abs{\la_{l}}^2) \bigwedge_{e \neq l} \darg(\la_{e})
	= I_{\RW}(A)
	.\end{align*}
	where in the second equality, we used the observation from \cref{lem:rotationInvariance}, that the following vanishes
	\[
		\log(\abs{\la_{d}}^2) \sum_{l=1}^{2n} (-1)^{l-1} \bigwedge_{e \neq l} \darg(\la_{e}) = 0
	.\] 
	This concludes one half of the proof.

	To show that, in the limit $\epsilon \to 0$, $I_{\aux}^{\epsilon}(A)$ is equal to $I_{\can}(A)$ we proceed as follows.
	Let $\mathcal{I}_{d}^{l} := \{(S,T) \subseteq  ([2n] \setminus \{l,d\})^2 \mid \abs{S} = \abs{T} = n-1 \text{ and } S \cap T = \emptyset \}$
	denote the set of ordered bipartitions of $[2n] \setminus \{l,d\}$ into two sets of size $n-1$.
	Using the absolute convergence from \cref{thm:productSpaceConvergence}, we may rewrite the integral over 
$\C^{n} \times \sigma_{2n}^{\epsilon}$ as an integral first over $\C^{n}$ and then over $\sigma_{2n}^{\epsilon}$.
	Applying \cref{thm:prodToPara} to each of the two integrals, we find that $I_{\aux}^{\epsilon}(A)$ is equal to
	\begin{align*}
		\sum_{\substack{l\in [2n]\\  l \neq d}} 
		&\left(\int_{\widetilde{\sigma}_{2n}^{\epsilon}}
		\frac{(-1)^{\frac{n + 1}{2}+d + l+(l<d)}}{\Psi^{n}_{l}} \sum_{S,T \in \mathcal{I}_{d}^{l}} \sgn(S,T) 
		\det(A_{Sd,\bullet}) \det(A_{Td,\bullet}) \perm(\adj(M_{l})_{S,T}) \frac{\Omega(x)}{x_{d}} \right.\\
			- &\left. \int_{\widetilde{\sigma}_{2n}^{\epsilon}}
		\frac{(-1)^{\frac{n + 1}{2}+d+l+(l<d)}}{\Psi^{n}_{d}} \sum_{S,T \in \mathcal{I}_{d}^{l}} \sgn(S,T)
		\det(A_{Sd,\bullet}) \det(A_{Td,\bullet}) \perm(\adj(M_{d})_{S,T}) \frac{\Omega(x)}{x_{d}} \right)
.\end{align*}
Moving the sum over $l$ inside each integral and applying \cref{thm:paraToCan} to the integrand, we obtain
\[
	I_{\aux}^{\epsilon}(A) = \left(\int_{\widetilde{\sigma}_{2n}^{\epsilon}}
		 \beta_{L_{l}} - \int_{\widetilde{\sigma}_{2n}^{\epsilon}} \beta_{L_{d}} \right)
.\] 
Both integrals are absolutely convergent for all $\epsilon > 0$. For $\epsilon = 0$, the domain of integration becomes
$\widetilde{\sigma}_{2n}^{0} = \left\{ x \in \R^{2n}_{\geq 0} \mid \sum_{i=1}^{2n} \frac{1}{x_{i}} = 1 \right\} $,
which is just a parametrisation of the projective simplex $\sigma_{2n}$ from the introduction.

Therefore, for $\epsilon = 0$ the integrals coincide with the canonical integrals for $A_{l}$ and $A_{d}$, and are thus absolutely convergent.
Consequently, we may take the limit and obtain
\[
	\lim_{\epsilon \to 0} I_{\aux}^{\epsilon}(A) = I_{\can}(A_{l}) - I_{\can}(A_{d})
.\] 
Finally, that $A_{d}$ contains the $d$-th row twice: once as row $d$ and once as
row $l$. Hence, by \cref{prop:canVanish}, the second integral vanishes.
Thus, we have shown that $I_{\RW}(A) = \lim_{\epsilon \to 0} I_{\aux}^{\epsilon}(A) = I_{\can}(A)$ which concludes the proof.
\end{proof}

Next we prove a generalisation of \cref{prop:introNoLsum} to matrices $A$, 
that is that the sum over the choice of logarithm edge in the definition of $I_{\RW}$ can be removed:
\begin{proposition}
	Let $n \geq 3$, and let $A \in \R^{2n \times n}$ of rank $n$, such that no row of $A$ is equal to $0$.
	Further, let $l,d \in [2n]$ with $l \neq d$, then 
	\begin{equation}\label{eq:RWWithoutL}
	I_{RW}(A) = (2n-1)\frac{ (-1)^{l-1}(2i)^{2n-1}}{2 (2 \pi i)^{n}} \int_{C_{A}} \log\bigg(\frac{\abs{\la_{l}}^2}{\abs{\la_{d}}^2}\bigg) 
	\bigwedge_{e \neq l} \darg(\la_{e})
.\end{equation}
	Moreover, on the reduced space $Z_{A} = \conf_{A}(\C) / \C^{*} $, this gives
	\[
		I_{RW}(A) = (2n-1) \frac{(-1)^{l+d+(l<d)}}{(2 \pi i)^{n-1}} \int_{Z_{A}} \log\bigg(\frac{\abs{\la_{l}}^2}{\abs{\la_{d}}^2}\bigg) 
		\bigwedge_{e \neq l,d} \dlog(\abs{\la}_{e}^2)
	.\] 
\end{proposition}

\begin{proof} Let $l,d \in [2n]$ with $l \neq d$.
	Observe, that in the formula for the canonical form in \cref{prop:canonicalDodgsonFormula} the product of 
	Dodgson polynomials is invariant under the cyclic permutation
	$(1 2 \ldots 2n-1)$. Thus we can fix $i_{\tau(1)}$ to be $l$ and find
	\[
		\beta^{2n-1}_{L} = (2n-1) \frac{(-1)^{d+l+(l<d)}}{\Psi^{2n-1}} \sum_{\tau \in \mathbb{S}_{2n-2}} \sgn(\tau) \Psi^{l,j_{\tau(1)}} 
		\Psi^{j_{\tau(1)},j_{\tau(2)}} \cdots \Psi^{j_{\tau(2n-2)}, l}
	,\]
	where $(j_1,\ldots,j_{2n-2}) = [2n] \setminus \{l,d\}$.
	Tracing back through the proof of \cref{thm:paraToCan}  this is exactly equal to just one term for a fixed $l$ in the sum over $l$, that is
	\[
	\beta^{2n-1}_{L} = (2n-1) \frac{(-1)^{\frac{n + 1}{2}+d+l+(l<d)}}{x_{d} \Psi^{n}} \sum_{S,T \in \mathcal{I}_{d}^{l}} \sgn(S,T)
	\det(A_{Sd,\bullet}) \det(A_{Td,\bullet}) 
	\perm(\adj(M)_{S,T}) \Omega(x)
	\]
	Finally, arguing as in the proof of \cref{thm:EqualityOfIntegrals} one can apply \cref{thm:prodToPara,thm:prodToPos}, 
	which are given for fixed $l$, to the canonical integral
	of $\beta^{2n-1}_{L}$ as given above, to obtain
	\[
	I_{\RW}(A) = I_{\can}(A) = \int_{\sigma_{2n}} \beta^{2n-1}_{L} = 
	\frac{(-1)^{l-1}(2i)^{2n-1}}{2(2 \pi i)^{n}} \int_{C_{A}} \log\bigg(\frac{\abs{\la_{l}}^2}{\abs{\la_{d}}^2}\bigg) 
	\bigwedge_{e \neq l} \darg(\la_{e})
	.\] 
	which proves the first statement.
	To prove the second statement, proceed exactly as in the proof of \cref{prop:ogRWIntegral}, 
	by integrating along the fibres of the canonical projection of  $\pi: C_{A} \to C_{A} / S^{1}$ and then replacing the darg-forms  by dlog-forms.
\end{proof}

\subsection{Absolute convergence of the auxiliary integral}

\begin{theorem}\label{thm:productSpaceConvergence}
	Let $n \geq 3$ odd, let $\epsilon > 0$, let $l,d,m \in [2n]$ with $l \neq d$ and 
	$A \in \R^{2n \times n}$ such that $\rank(A_{i,\bullet}) = n$ for all $i \in [2n]$.
	Then the following integral converges absolutely
	\begin{equation}\label{eq:absConvInt}
		\int_{\C^{n} \times \sigma_{2n}^{\epsilon}} \frac{\bigwedge_{e \neq l,d} (\overline{\la_{e}} \dd{\la_{e}} + \la_{e} \dd{\overline{\la_{e}}})
		\wedge \dd{\la_{d}} \wedge \dd{\overline{\la_{d}}}}
	{(\sum_{e\neq l} \alpha_{e} \abs{\la_{e}}^2 + \alpha_{l} \abs{\la_{m}}^2 + \alpha_{l})^{2n}} \Omega(\alpha)
.\end{equation}
\end{theorem}

\begin{proof}
	We prove the absolute convergence by compactifying the integration domain. Let $\pi: \CP^{n} \setminus \{z_0=0\}  \to \C^{n},
	[z_0 : z_1 : \ldots :z_{n}] \mapsto \left( \frac{z_1}{z_0}, \ldots, \frac{z_{n}}{z_0} \right)$ be the standard affine chart.
Pulling back each term of the wedge product in the numerator in \cref{eq:absConvInt} by $\pi$ gives
	\[
		\pi^{*} \left(\overline{\la_{e}} \dd{\la_{e}} + \la_{e} \dd{\overline{\la_{e}}}\right) = \frac{\overline{\la_{e}} \dd{\la_{e}} + \la_{e} \dd{\overline{\la_{e}}}}{|z_0|^2}
		- \frac{|\la_{e}|^2}{|z_0|^2} \left( \frac{\dd{z_{0}}}{z_0} + \frac{\dd{\overline{z_{0}}}}{\overline{z_0}} \right) 
		\qq{and} \pi^{*}(\dd{\la_{d}}) = \frac{\dd{\la_{d}}}{z_{0}} - \la_{d} \frac{\dd{z_0}}{z_0^2}
	.\]
	Importantly, $\dd{z_0}$ and $\dd{\overline{z_0}}$ can appear at most once and therefore the full numerator pulls back to an expression of the form
	\[
		\pi^{*} \bigg( \bigwedge_{e \neq l,d} (\overline{\la_e}\dd{a_e} + \la_e \dd{\overline{\la_e})} \wedge \dd{\la_d} \wedge \dd{\overline{\la_d}}\bigg)
		=\frac{Q(z,\overline{z})}{|z_0|^{4n}}
		\left( \sum_{i,j=0}^{n}  z_{i} \overline{z_{j}} \left( \dd{z_{j}} \wedge \dd{\overline{z_{i}}} \right) 
			\bigwedge_{e \neq i,j} \dd{z_{e}} \wedge \dd{\overline{z_{e}}}\right)
	\]
	Here, $Q(z,\overline{z})$ is a polynomial in the homogeneous coordinates and the expression in parentheses 
	is the standard volume form on $\CP^{n}$ which we denote by $\Omega([z])$.
	Pulling back the denominator in \cref{eq:absConvInt} by $\pi$ we obtain
	\[
	\pi^{*} \bigg(\sum_{e\neq l} \alpha_{e} |\la_{e}|^2 + \alpha_{l} |\la_{m}|^2 + \alpha_{l}\bigg)^{2n} = 
	\frac{(\sum_{e\neq l} \alpha_{e} |\la_{e}|^2 + \alpha_{l} |\la_{m}|^2 + \alpha_{l})^{2n}}{|z_0|^{4n}}
	.\] 
	Hence, the $z_0$-factor from the numerator exactly cancels the $z_0$-factor coming from the denominator and the whole integrand becomes
	\[
		\int_{\CP^{n} \times \sigma_{2n}^{\epsilon}} \left[ \frac{Q(z,\overline{z})}
		{(\sum_{e\neq l} \alpha_{e} |\la_{e}|^2 + \alpha_{l} |\la_{m}|^2 + \alpha_{l} |z_0|^2 )^{2n}} \right] 
		\Omega([z]) \Omega(\alpha)
	\]
	Let $z = (z_1,\ldots,z_{n})$. We denote the function in brackets by $g( [z_0 :z],\alpha)$.
	It remains to show that $g( [z_0:z],\alpha)$ is continuous on $\CP^{n} \times \sigma_{2n}^{\epsilon}$.
	Since $\alpha_{i} \geq \epsilon$ for all $i \in [2n]$, and the denominator is a sum of non-negative terms, it can vanish only if
	\[
		\la_{e}(z) = 0 \text{ for all } e \neq l,\quad \la_{m}(z) = 0 \qq{and} z_0 = 0
	.\] 
	By assumption $\rank(A_{l,\bullet}) = n$, so the linear forms $(\la_{e})_{e \neq l}$ span $(\C^{n})^{*}$. 
	Hence $\la_{e}(z) = 0$ for all $e \neq l$ implies $z = (z_1,\ldots,z_{n}) = (0,\ldots,0)$.
	Together with $z_0 = 0$, this would imply that all homogeneous coordinates $z_{i}$ ($0 \leq i \leq n$ ) vanish.
	This however is impossible as the origin in not a point in $\CP^{n}$.
	Therefore, the denominator is strictly positive on $\CP^{n} \times  \sigma_{2n}^{\epsilon}$ 
	and $g$ is continuous.

	Consequently $g$ is bounded as $\CP^{n}$ and $\sigma_{2n}^{\epsilon}$ are compact. As the spaces also have finite measure,
	the integral of $\abs{g}$ over their product is finite.
	This proves the absolute convergence.
\end{proof}

\graphicspath{{Images/}}

\section{From product space to position space - Theorem \ref*{thm:prodToPos}}\label{sec:prodToPos}
We show the first step in proving \cref{thm:EqualityOfIntegrals} by showing that the auxiliary integral is equal to 
the RW-integral (up to some constant):
\begin{theorem}\label{thm:prodToPos}
	Let $n \geq 3$ odd, $\epsilon > 0$, $l,d \in [2n]$ such that $l \neq d$ and $A \in \R^{2n \times  n}$ such that $\rank(A_{i,\bullet}) = n$ 
	for all $i \in [2n]$ and no row of $A$ is equal to $0$. 
	Further, let $\sigma_{2n}^{\epsilon} = \{\alpha \in \R^{2n}_{\geq \epsilon} \mid \sum_{i=1}^{2n} \alpha_{i} = 1\}$
	be the truncated unit coordinate simplex. Then
	\begin{multline*}
		\frac{(2i)^{2n-1}}{2} \int_{C_{A}} \log\bigg(\frac{\abs{\la_{l}}^2}{\abs{\la_{d}}^2}\bigg) \bigwedge_{e \neq l} \darg(\la_{e})
=\Gamma(2n) \lim_{\epsilon \to 0}  (-1)^{d + (d<l)}
\int_{\C^{n}} \int_{ \sigma_{2n}^{\epsilon}}\\
	\left( \frac{\bigwedge_{e \neq l,d} ( \overline{\la_{e}} \dd{\la_{e}} + \la_{e} \dd{\overline{\la_{e}}}) 
		\wedge \dd{\la_{d}} \wedge \dd{\overline{\la_{d}}}
}{(\sum_{e \neq l} \alpha_{e} \abs{\la_{e}}^2 + \alpha_{l} \abs{\la_{l}}^2 + \alpha_{l})^{n}}
	-\frac{\bigwedge_{e \neq l,d} ( \overline{\la_{e}} \dd{\la_{e}} + \la_{e} \dd{\overline{\la_{e}}}) 
		\wedge \dd{\la_{d}} \wedge \dd{\overline{\la_{d}}}
}{(\sum_{e \neq l} \alpha_{e} \abs{\la_{e}}^2 + \alpha_{l} \abs{\la_{d}}^2 + \alpha_{l})^{n}}\right)  \Omega(\alpha)
.\end{multline*}	
\end{theorem}

The idea of the proof is to use the Feynman parametrisation to rewrite the integrand. 
However, for this to work properly, we have to regularise it:
\subsection{Regularised Feynman parametrisation}
The classic Feynman parametrisation \cref{prop:FeynmanPara} is only defined when all $A_{i}$ are nonzero.
For our purpose, we need to extend this to allow some $A_{i}$ to go to zero.
For simplicity we reduce to the case where all exponents $s_1,\ldots,s_{n}$ are equal to $1$.
We find the following regularised identity:
\begin{theorem}[Regularised Feynman parametrisation]\label{thm:regFeynmanParam}
	Let $n \in \N$, let $A_1,\ldots,A_{n} \in \R_{\geq 0}$ not all zero and define $A := \sum_{i=1}^{n} A_{i}$.
	Further, let $\epsilon > 0$ and $t := 1 - n \epsilon$ and define the truncated unit coordinate simplex as
	$\sigma^{\epsilon}_{n} = \{ \alpha \in \R_{\geq \epsilon}^{n} \mid \sum_{i=1}^{n} \alpha_{i} = 1\}$.
	Then,
	\[
	\frac{1}{t} \prod_{i=1}^{n} \frac{1}{(A_{i} + \frac{\epsilon}{t} A)} = \Gamma(n) \int_{\sigma^{\epsilon}_{n}}
	\frac{1}
	{(\alpha_{1} A_1+\ldots+ \alpha_{n} A_{n})^{n}} \Omega(\alpha)
	.\] 
\end{theorem}

\begin{proof}
	Since at least one $A_{i}$ is nonzero, $A > 0$ and  the classic Feynman parametrisation applies
	\[
	\frac{1}{t} \prod_{i=1}^{n} \frac{1}{(A_{i} + \frac{\epsilon}{t} A)} =
	\frac{\Gamma(n)}{t} \int_{\sigma^{0}_{n}} \frac{1}{(\sum_{i=1}^{n} \beta_{i} A_{i} + \frac{\epsilon}{t} A)^{n}} \Omega(\beta)
	,\] 
	where we used that $\sum_{i=1}^{n} \beta_{i} = 1$.
	Consider the map $\varphi: \sigma_{n}^{\epsilon} \to \sigma_{n}^{0}, \alpha \mapsto \frac{\alpha - \epsilon}{t}$.
	One computes $\varphi(\sigma^{0}_{n}) = \sigma_{n}^{\epsilon}$ and using $\varphi$ as a change of variables finds
	\begin{equation}\label{eq:genFeynmanPara}
		\frac{1}{t} \prod_{i=1}^{n} \frac{1}{(A_{i} + \frac{\epsilon}{t} A)} = \Gamma(n)
		\int_{\sigma_{n}^{\epsilon}} \frac{t^{n-1}}{(\sum_{i=1}^{n} \alpha_{i} A_{i} )^{n}} \Omega\left(\frac{\alpha - \epsilon}{t}\right)
	.\end{equation}
	On the simplex $\sigma_{n}$, consider the affine chart given by $\alpha_{n} = 1- \sum_{i=1}^{n-1} \alpha_{i}$. In these coordinates,
	\[
		\Omega(\alpha)  = (-1)^{n-1} \dd{\alpha_{1}}\ldots \dd{\alpha_{n-1}}
		\qq{and} S(\alpha) := \sum_{i=1}^{n} (-1)^{n} \bigwedge_{e \neq i} \dd{\alpha_{e}} = n (-1)^{n-1} \dd{\alpha_{1}}\ldots \dd{\alpha_{n-1}}
	.\] 
	A direct computation yields
	\[
		\Omega\left(\frac{\alpha - \epsilon}{t}\right) = \frac{1}{t^{n}} \Omega(\alpha) - \frac{\epsilon}{t^{n}} S(\alpha)
		= \frac{1}{t^{n}} (1-n \epsilon) (-1)^{n-1} \dd{\alpha_1}\ldots\dd{\alpha_{n-1}} = \frac{\Omega(\alpha)}{t^{n-1}}
	.\] 
	Substituting into \cref{eq:genFeynmanPara} completes the proof.
\end{proof}

To prove \cref{thm:prodToPos}, we will also need the following partial fraction decomposition identity:
\begin{lemma}\label{lem:partialFractionDecomp}
	Let  $n \in \N$, let $A_{i} \in \R_{> 0}$ and $b_{i} \in \R_{> 0}$ for $i \in [n]$ and let $m \in \N$, $m < n-1$, then
	\[
		\frac{ r^{m}}{\prod_{i=1}^{n} (A_{i} r + b_{i})}  = (-1)^{m} \sum_{i=1}^{n} \frac{c_{i}}{A_{i} r + b_{i}}
		\qq{where} c_{j} = \frac{(- b_{j})^{m} A_{j}^{n-1-m}}{\prod_{k\neq j} (A_{j} b_{k} - A_{k} b_{j})}
	.\] 
	Further, the  $c_{i}$ satisfy $\sum_{i=1}^{n} \frac{c_{i}}{A_{i}} = 0$.
\end{lemma}

\begin{proof}
	This is mostly an exercise in partial fraction decomposition. That is, we want to find $c_{i} \in \R$ such that
	\begin{equation}\label{eq:partialFractionProof}
	\frac{r^{m}}{\prod_{i=1}^{n} (A_{i} r + b_{i})} = \sum_{i=1}^{n} \frac{c_{i}}{A_{i} r + b_{i}}
	\implies r^{m} = \sum_{i=1}^{n} c_{i} \prod_{k \neq i} (A_{k} r + b_{k})
	,\end{equation}
	where we multiplied both sides by $\prod_{i=1}^{n} (A_{i} r + b_{i})$.
	Evaluating at $r = -\frac{b_{j}}{A_{j}}$ for $j \in [n]$ gives
	\[
		\left( - \frac{b_{j}}{A_{j}} \right)^{m} = \sum_{i=1}^{n} c_{i} \prod_{k \neq i} \left(A_{k} \left( - \frac{b_{j}}{A_{j}} \right)  + b_{k}\right)
		= \frac{c_{j}}{A_{j}^{n-1}} \prod_{k \neq j} (A_{j} b_{k} - A_{k} b_{j})\\
	,\]
	and solving for $c_{j}$ proves the first part.
	Further, take \cref{eq:partialFractionProof} and compare coefficients of $r^{n-1}$ on both sides. As $m < n-1$ this gives
	$0 = \left( \prod_{i=1}^{n} A_{i} \right) \sum_{i=1}^{n} \frac{c_{i}}{A_{i}}$ and thus 
	$\sum_{i=1}^{n} \frac{c_{i}}{A_{i}} = 0$, which concludes the proof.
\end{proof}

\subsection{Proof of \cref{thm:prodToPos}}
For the rest of this section we define the following objects:
\begin{definition}
	Let $n \in \N$, $l,d \in [2n]$, let $\epsilon > 0$ and $\gamma = \frac{\epsilon}{1- 2n \epsilon}$.
	Let $A \in \R^{2n \times n}$ such that  $\rank(A_{i,\bullet}) = n$ for all $i \in [2n]$.
	Further, define the following shorthands:
	\[
		S = \sum_{i \neq l} \abs{\la_{i}}^2, \quad
		\begin{array}{l}
		\widetilde{S}_{l} = S + \abs{\la_{l}}^2\\
		\widetilde{S}_{d} = S + \abs{\la_{d}}^2
		\end{array}
		\qq{ and } 
		\begin{array}{l}
		S_{l} = \widetilde{S}_{l} + 1\\
		S_{d} = \widetilde{S}_{d} + 1 
		\end{array}
	\]
\end{definition}

Starting from the right-hand side (RHS) of \cref{thm:prodToPos}, observe that as $\abs{\la_{l}}^2 + 1$ and $\abs{\la_{d}}^2 + 1$ 
are positive we can apply \cref{thm:regFeynmanParam} to obtain that the RHS is equal to
\begin{align*}
	\lim_{\epsilon \to 0} \frac{(-1)^{d + (d<l)}}{t} \int_{\C^{n}}
		\frac{\bigwedge_{e \neq l,d} ( \overline{\la_{e}} \dd{\la_{e}} + \la_{e} \dd{\overline{\la_{e}}}) \wedge \dd{\la_{d}} \wedge \dd{\overline{\la_{d}}}}
		{(\abs{\la_{l}}^2 +\gamma S_{l}+1) \prod_{e \neq l} (\abs{\la_{e}}^2 + \gamma S_{l})}
	- \frac{\bigwedge_{e \neq l,d} ( \overline{\la_{e}} \dd{\la_{e}} + \la_{e} \dd{\overline{\la_{e}}}) \wedge \dd{\la_{d}} \wedge \dd{\overline{\la_{d}}}}
{(\abs{\la_{d}}^2 + \gamma S_{d}+1)  \prod_{e \neq l} (\abs{\la_{e}}^2 + \gamma S_{d})}
.\end{align*}

Notice, that  $\overline{\la_{e}} \dd{\la_{e}} - \la_{e} \dd{\overline{\la_{e}}} = 2i \abs{\la_{e}}^2 \darg(\la_{e})$ and
$\dd{\la_{d}} \wedge \dd{\overline{\la_{d}}} = i \abs{\la_{d}}^2 \darg(\la_{d}) \wedge \dlog(\abs{\la_{d}}^2)$.
It follows that
\begin{align*}
	\bigwedge_{e \neq l,d} (\overline{\la_{e}} \dd{\la_{e}} + \la_{e} \dd{\overline{\la_{e}}})
	\wedge \dd{\la_{d}} \wedge \dd{\overline{\la_{d}}}
	&=(-1)^{n-1} \bigwedge_{e \neq l,d} (\overline{\la_{e}} \dd{\la_{e}} - \la_{e} \dd{\overline{\la_{e}}})
	\wedge \dd{\la_{d}} \wedge \dd{\overline{\la_{d}}}\\
	&=  (-1)^{d + (l<d)}\frac{(2i)^{2n-1}}{2}\prod_{e \neq l}\abs{\la_{e}}^2 \dlog(\abs{\la_{d}}^2)  \bigwedge_{e \neq l} \darg(\la_{e})
.\end{align*}
Here, in the first equality we used that an equal number of holomorphic and anti-holomorphic components need to be chosen, and thus exactly $n-1$ minus
appear.
The sign in the second equality arises from moving $\dlog(\abs{\la_{d}}^2)$ to the front giving a sign of $(-1)^{2n-1}$ and moving $\darg(\la_{d})$ into its appropriate position contributing a sign of  $(-1)^{2n - d - (d<l)}$, as well as noting that $n$ is odd.
The RHS of \cref{thm:prodToPos} is therefore equal to
\begin{align*}
		\lim_{\epsilon \to 0} \frac{(2i)^{2n-1}}{(-2)t} \int_{\C^{n}} \left(
		\frac{\prod_{e \neq l} \frac{\abs{\la_{e}}^2}{(\abs{\la_{e}}^2 + \gamma S_{l})}}
			{(\abs{\la_{l}}^2 + \gamma S_{l}+1)}
		-\frac{\prod_{e \neq l} \frac{\abs{\la_{e}}^2}{(\abs{\la_{e}}^2 + \gamma S_{d})}}
			{(\abs{\la_{d}}^2 + \gamma S_{d}+1)} \right)
		\dlog(\abs{\la_{d}}^2) \bigwedge_{e \neq l} \darg(\la_{e})
.\end{align*}
In the next step we integrate over the scale. For this, consider the map $\varphi: C_{A} \times  \R_{\geq 0} \to V, (z,r) \mapsto \sqrt{r} z$
where $V = \{z \in \C^{n} \mid \la_{i}(z) \neq 0 \forall i \in [2n]\}$. 
This map is a diffeomorphism onto $V$. The hypersurfaces $\{\la_{i}(z) = 0\} $, as well as the origin corresponding to $r=0$,
have Lebesgue measure zero. Hence replacing $\C^{n}$ by $V$ does not change the value of the integral, and we may use $\varphi$ as a change of variables.
Thus the RHS of \cref{thm:prodToPos} is equal to
\begin{align}\label{eq:preScaleIntegration}
		\lim_{\epsilon \to 0} \frac{(2i)^{2n-1}}{(-2) t} \int_{C_{A}} \int_{0}^{\infty} \left(
		\frac{\prod_{e \neq l} \frac{r^{2n-2}\abs{\la_{e}}^2}{ (r(\abs{\la_{e}}^2 + \gamma \widetilde{S}_{l}) + \gamma)} }
		{(r(\abs{\la_{l}}^2 +\gamma \widetilde{S}_{l}) + 1+ \gamma)}
		-\frac{\prod_{e \neq l} \frac{r^{2n-2}\abs{\la_{e}}^2}{ (r(\abs{\la_{e}}^2 + \gamma \widetilde{S}_{d}) + \gamma)} }
		{(r(\abs{\la_{d}}^2 +\gamma \widetilde{S}_{d}) + 1+\gamma)} 
		\right) \dd{r} \bigwedge_{e \neq l} \darg(\la_{e})
\end{align}
Here we used $\darg(\sqrt{r} \la_{e}) = \darg(\la_{e})$ and hence the term $\dlog(\abs{\la_{d}}^2)$ has to contribute a $\dd{r}$
to obtain a volume form.
For the $r$-integration we use the partial-fraction identity \cref{lem:partialFractionDecomp}.
To simplify notation set
\[
\begin{gathered}
A_1 := |\la_{l}|^2 + \gamma \widetilde S_l,\qquad
A_i := |\la_{k_i}|^2 + \gamma \widetilde S_l\quad(2\le i\le 2n),\\
A_1' := |\la_{d}|^2 + \gamma \widetilde S_d,\qquad
A_i' := |\la_{k_{i}}|^2 + \gamma \widetilde S_d\quad(2\le i\le 2n),
\end{gathered}
\]
where $\{k_1,\dots,k_{2n-1}\}=[2n]\setminus\{l\}$, and set $b_1:=\gamma+1$, $b_i:=\gamma$ for $2\le i\le 2n$.
Let $c_j$ and $c_j'$ denote the residues appearing in the partial-fraction decomposition \cref{lem:partialFractionDecomp}.
Then
\begin{align*}
	&\int_{0}^{\infty} \left(
	\frac{\prod_{e \neq l} \frac{r^{2n-2}\abs{\la_{e}}^2}{ (r(\abs{\la_{e}}^2 + \gamma \widetilde{S}_{l}) + \gamma)} }
	{(r(\abs{\la_{l}}^2 +\gamma \widetilde{S}_{l}) + 1+ \gamma)}
	-\frac{\prod_{e \neq l} \frac{r^{2n-2}\abs{\la_{e}}^2}{ (r(\abs{\la_{e}}^2 + \gamma \widetilde{S}_{d}) + \gamma)} }
	{(r(\abs{\la_{d}}^2 +\gamma \widetilde{S}_{d}) + 1+\gamma)} 
	\right) \dd{r}\\
	=&\prod_{k \neq l} \abs{\la_{k}}^2 \lim_{R \to \infty} \int_{0}^{R} \left (\frac{1}{\prod_{i=1}^{2n} (A_{i} r + b_{i})} - \frac{1}{\prod_{i=1}^{2n} (A_{i}' r +b_{i})} \right) r^{2n-2}
\dd{r}\\
=&\prod_{k \neq l} \abs{\la_{k}}^2 \lim_{R \to \infty}  \sum_{i=1}^{2n} \int_{0}^{R} 
\left(\frac{c_{i}}{A_{i}r + b_{i}} - \frac{c_{i}'}{A_{i}' r + b_{i}} \right) \dd{r}\\
=&\prod_{k \neq l} \abs{\la_{k}}^2 \lim_{R \to \infty} \sum_{i=1}^{2n} \left( \frac{c_{i}}{A_{i}} \log\bigg( \frac{A_{i} R + b_{i}}{b_{i}}\bigg) - 
\frac{c_{i}'}{A_{i}'} \log\bigg( \frac{A_{i}' R + b_{i}}{b_{i}}\bigg)\right) \\
	\overset{R' = \frac{1}{R}}{=}&\prod_{k \neq l} \abs{\la_{k}}^2 \lim_{R' \to 0} \sum_{i=1}^{2n} 
	\left( \frac{c_{i}}{A_{i}} \log\bigg(\frac{A_{i}}{b_{i}} + R'\bigg) - 
\frac{c_{i}'}{A_{i}'} \log\bigg(\frac{A_{i}'}{b_{i}} + R'\bigg) \right) -  \log(R') \sum_{i=1}^{2n} \left( \frac{c_{i}}{A_{i}} - \frac{c_{i}'}{A_{i}'}\right)
\displaybreak[0]
\\
	\overset{(\dagger)}{=}&\prod_{k \neq l} \abs{\la_{k}}^2 \sum_{i=1}^{2n} \left( \frac{c_{i}}{A_{i}} \log\bigg(\frac{A_{i}}{b_{i}}\bigg) - 
\frac{c_{i}'}{A_{i}'} \log\bigg(\frac{A_{i}'}{b_{i}}\bigg) \right)
\end{align*}
Here, $(\dagger)$ follows from the cancellation of the leading-order residues as in \cref{lem:partialFractionDecomp}.
From observing that $\widetilde{S}_{l} - \abs{\la_{l}}^2 = S$ and $\widetilde{S}_{d} - \abs{\la_{d}}^2 = S$ one calculates that
$c_{i} / A_{i}$ and $c_{i}' / A_{i}'$ are equal and given by
\begin{align*}
	\frac{c_{1}}{A_{1}} = \frac{- (1+\gamma)^{2n-2}}{\prod_{k \neq l}( (1+\gamma ) \abs{\la_{k}}^2 + \gamma S) } \qq{and}
	\frac{c_{i}}{A_{i}} =
	\frac{1}{(1+\gamma) \abs{\la_{i}}^2 + \gamma S) \prod_{\substack{k\neq i,l}} (\abs{\la_{i}}^2 - \abs{\la_{k}}^2)} 
.\end{align*}
Applying this to \cref{eq:preScaleIntegration} it follows that the RHS of \cref{thm:prodToPos} is equal to
\begin{multline}\label{eq:posLimExchange}
	\lim_{\epsilon \to 0} \frac{(2i)^{2n-1}}{(-2) t} \int_{C_{A}} 
	\left( \sum_{i \neq l} \frac{1}{\prod_{\substack{k\neq i\\k \neq l}} \abs{\la_{i}}^2 - \abs{\la_{k}}^2} \cdot 
	\frac{\prod_{k \neq l}  \abs{\la_{k}}^2 }{((1+\gamma) \abs{\la_{i}}^2 + \gamma S)} 
	\log\bigg(\frac{\gamma \abs{\la_{l}}^2 + \abs{\la_{i}} + \gamma S}{\gamma \abs{\la_{d}}^2 + \abs{\la_{i}} + \gamma S}\bigg) \right. \\
\left.- \frac{(1+\gamma)^{2n-2} \prod_{k\neq l} \abs{\la_{k}}^2}{\prod_{k \neq l} ((1+\gamma) \abs{\la_{k}}^2 + \gamma S)}
\log\bigg(\frac{(1+\gamma) \abs{\la_{l}}^2 + \gamma S}{(1+\gamma) \abs{\la_{d}}^2 + \gamma S }\bigg) \rule{0cm}{0.9cm} \right) 
\bigwedge_{e \neq l} \darg(\la_{e})
.\end{multline}
In the limit $\epsilon \to 0$ the term in brackets will collapse to $- \log(\abs{\la_{l}}^2 / \abs{\la_{d}}^2)$ showing the result.
To take the limit, we need to exchange it with the integral. The following proposition allows us to do this.
\begin{proposition}\label{prop:absoluteBoundAndConvergence}
	For any $\gamma < 1$ the integrand in \cref{eq:posLimExchange} is bounded absolutely by $g: \{z \in \C^{n} \mid \abs{z} = 1 \text{ and } 
	\la_{i}(z) \neq 0 \forall i \in [2n]\} \to \C$ given by
	 \[
g(z) := C \left\lvert 1 + \log\bigg(\frac{|\la_{l}|^2}{|\la_{d}|^2}\bigg)\right\rvert \left| \bigwedge_{e\neq l} \darg(\la_{e}) \right| 
	.\] 
	for $C \in \R_{> 0}$. The function $g$ is absolutely integrable on its domain.
\end{proposition}
Before proving the proposition we finish the proof of \cref{thm:prodToPos}.
As $ \gamma = \frac{\epsilon}{1-2n \epsilon}$ is invertible we can change the limit from $\epsilon \to 0$ to $\gamma \to 0$.
Under this transformation $t = \frac{1}{1+n \gamma}$. 
By \cref{prop:absoluteBoundAndConvergence} the dominated convergence theorem (\cref{thm:dominatedConvergenceTheorem}) 
applies to the integral in \cref{eq:posLimExchange}
and we can take the limit in $\gamma$ on the integrand of \cref{eq:posLimExchange} to find that the RHS of \cref{thm:prodToPos}
is equal to
\[
	\frac{(2i)^{2n-1}}{2} \int_{C_{A}} \log\bigg(\frac{\abs{\la_{l}}^2}{\abs{\la_{d}}^2}\bigg) \bigwedge_{e \neq l,d} \darg(\la_{e})
,\] 
which concludes the proof of \cref{thm:prodToPos}. \qed

To prove \cref{prop:absoluteBoundAndConvergence} we need the following additional lemma:
\begin{lemma}\label{lem:logBound}
	Let $w \in \R_{\geq 0}$ and $u,v \in \R_{> 0}$ then
	\[
	\abs{\log\left(\frac{u+w}{v+w}\right)} \leq \abs{\log\left(\frac{u}{v}\right)}
	.\] 
\end{lemma}

\begin{proof}
Observe first that
 \[
	 1 \leq \frac{u+w}{v+w} \leq \frac{u}{v} \qq{ if $u > v$ and } \frac{u}{v} \leq \frac{u+w}{v+w} \leq 1 \qquad \text{if } u < v
 \] 
Hence, it follows that
 \[
	 \abs{\log\left(\frac{u+w}{v+w}\right)} = \begin{cases}
		 \log(\frac{u+w}{v+w}) & \text{ if } u >v\\
		 -\log(\frac{u+w}{v+w}) & \text{ if } u < v\\
	 \end{cases} \leq 
	 \begin{cases}
		 \log(\frac{u}{v}) & \text{ if } u >v\\
		 -\log(\frac{u}{v}) & \text{ if } u < v\\
	 \end{cases} = \abs{\log\left(\frac{u}{v}\right)}
 \] 
 as $\log(x)$ is monotonically increasing and $-\log(x)$ is monotonically decreasing.
\end{proof}

\begin{proof}[Proof of \cref{prop:absoluteBoundAndConvergence}]
	First, we identify $C_{A}$ with $Y_{A} := \{z \in \conf_{A}(\C) \mid \abs{z}=1\} \subseteq S^{2n-1}$ by fixing the scale of $z$ as $1$.
	On the chart $Y_{A}$ all $\abs{\la_{i}}^2$ can be bounded by some $c \in \R_{> 0}$ as 
	\[
		\abs{\la_{i}}^2 = \sum_{k=1}^{n} a_{ik}^2 \abs{z_{i}}^2 \leq \abs{z}^2 \max_{k \in [n]} a_{ik}^2 \leq c := \max_{i \in [2n], k \in [n]} a_{ik}^2
	.\]
	To bound the sum over $i$ we are going to use the theory of divided differences (see \cref{ap:divDiff}).
	Define the function $h_{\gamma}(t) : \R_{\geq 0} \to \R$ as
	\[
		h(t) := \frac{1}{(\gamma (t + S)+t)} 
		\log\bigg(\frac{t + \gamma \abs{\la_{l}}^2 + \gamma S}{t + \gamma \abs{\la_{d}}^2 + \gamma S}\bigg) \qq{and} h_{\gamma}(t) := 0 \text{ if } S = 0
	.\] 
	The function $h$ is smooth in $t$ and continuous in the parameters $\gamma$ and $\abs{\la_{k}}^2$.
	Denote $I = \{i_1,\ldots,i_{2n-1}\} := \{1,\ldots,2n\} \setminus \{l\} $ and $x_{j} = |\la_{i_{j}}|^2$ for
	$1 \leq j \leq 2n-1$. Then the absolute value of the first part of the integrand of \cref{eq:posLimExchange} can be rewritten and bounded as:
	\begin{align*}
		\abs{\sum_{i\neq l} \frac{ \log\Big(\frac{t + \gamma \abs{\la_{l}}^2 + \gamma S}{t + \gamma \abs{\la_{d}}^2 + \gamma S}\Big) \prod_{k\neq l} \abs{\la_{k}}^2}
		{((1+\gamma) t + \gamma S ) \prod_{\substack{k\neq i \\k\neq l}} (\abs{\la_{i}}^2 - \abs{\la_{k}}^2)} }
		\leq c^{2n-1} \abs{\sum_{j=1}^{2n-1}  \frac{h(t)}{\prod\limits_{k \neq j} (x_{j} - x_{k})}}
		= c^{2n-1} \abs{h[x_1,\ldots,x_{2n-2}]}
	.\end{align*}
	Here $h_{\gamma}[x_1,\ldots,x_{2n-1}]$ is the divided difference as defined in \cref{def:dividedDifference} and we used
	\cref{prop:dividedDifference}. Further, we observed that $\prod_{k \neq l} \abs{\la_{k}}^2 \leq c^{2n-1}$.
	By the mean value theorem for divided differences (\cref{thm:meanDividedDiff}) there exists $\xi \in \R_{> 0}$ such that
	\[
		h_{\gamma}[x_1,\ldots,x_{2n-1}] = \frac{1}{(2n-2)!} \frac{\partial^{2n-2} }{\partial t^{2n-2}} h_{\gamma}(\xi) 
	.\]
	With this $ c^{2n-1} \abs{h_{\gamma}[x_1,\ldots,x_{2n-1}]}$ is bounded independently of $\gamma$ as follows:
	\begin{align*}
		\abs{ h_\gamma[x_1,\ldots,x_{2n-1}]}
		&\leq \abs{ \frac{(1+\gamma)^{2n-2}}{(\xi(\gamma+1)+\gamma S)^{2n-1}} 
			\log\bigg(\frac{\xi+\gamma \abs{\la_{l}}^2+\gamma S}{\xi+\gamma \abs{\la_{d}}^2+\gamma S}\bigg)}\\
		&\quad{}+ \sum_{k=1}^{2n-2}\frac{(\gamma+1)^{2n-2-k}}{k(\xi(\gamma+1)+\gamma S)^{2n-1-k}}
	\abs{\frac{1}{(\xi+\gamma \abs{\la_{l}}^2+\gamma S)^k}-\frac{1}{(\xi+\gamma\abs{\la_{d}}^2+\gamma S)^k}}\\
		&\le \frac{2^{2n-2}}{\xi^{2n-1}}\abs{\log\bigg(\frac{\abs{\la_{l}}^2}{\abs{\la_{d}}^2}\bigg)}
		+ \sum_{k=1}^{2n-2}\frac{2^{2n-1-k}}{k\,\xi^{2n-1}}.
\end{align*}
The first step comes from calculating the $(2n-2)$-th derivative and the triangle inequality.
The second step uses that  $\xi (\gamma +1) + \gamma S \geq \xi$, $\xi + \gamma \abs{\la_{j}}^2 + \gamma S \geq \xi$,
triangle inequality and \cref{lem:logBound}.

To bound the second term of the integrand in \cref{eq:posLimExchange} we proceed similarly by using \cref{lem:logBound}
and $(1+\gamma) \abs{\la_{k}}^2 + \gamma S \geq \abs{\la_{k}}^2$. To obtain
\[
\abs{\frac{(1+\gamma)^{2n-2} \prod_{k\neq l} \abs{\la_{k}}^2}{\prod_{k \neq l} ((1+\gamma) \abs{\la_{k}}^2 + \gamma S)}
\log\bigg(\frac{(1+\gamma) \abs{\la_{l}}^2 + \gamma S}{(1+\gamma) \abs{\la_{d}}^2 + \gamma S }\bigg)}
\leq 2^{2n-2} \log\bigg(\frac{\abs{\la_{l}}^2}{\abs{\la_{d}}^2}\bigg) 
.\] 
Combining both shows that the integrand of \cref{eq:posLimExchange} can be bounded by the function $g$ as stated
\[
g(z) := C \abs{1 + \log\bigg(\frac{\abs{\la_{l}}^2}{\abs{\la_{d}}^2}\bigg)} \abs{ \bigwedge_{e\neq l} \darg(\la_{e})}
.\] 
where $C = \max(2^{2n-2},c^{2n-1} \frac{2^{2n-2}}{\xi^{2n-1}},c^{2n-1} \sum_{k=1}^{2n-2} \frac{2^{2n-1-k}}{k \xi^{2n-1}})$.
This proves the first part of the proposition.

To prove the absolute convergence of the integral of $g$ over $Y_{A}$ we first compactify $Y_{A}$.
This is done by iteratively performing real-oriented blow-ups along the strata $H_{I} := \bigcap_{i \in I} \{\la_{i}(z) = 0\}$,
starting with the highest codimension strata and proceeding in decreasing codimension.
Locally, this replaces the normal directions to each stratum by their sphere of directions, thereby recording the
relative rates and arguments with which the $\la_{i}$ tend to zero.
This gives a compact manifold with corners $\overline{Y_{A}}$, following the Axelrod--Singer construction \cite{axelrod94}.
On $\overline{Y_{A}}$, each angular map $\la_{e} / \abs{\la_{e}}$ extends continuously to the boundary, so the forms $\darg(\la_{e})$ 
extend with controlled boundary behaviour. Thus we can write
\[
	\bigwedge_{e \neq l} \darg(\la_{e}) = h(z) \vol_{\overline{Y_{A}}}
.\] 
where $h$ is continuous on $\overline{Y_{A}}$ and $\vol_{\overline{Y_{A}}}$ is a smooth volume form.
Since $\overline{Y_{A}}$ is compact, $\abs{h(z)}$ is bounded by a constant $M \in \R_{> 0}$. Hence
\[
\int_{Y_{A}} g(z) \leq C M \int_{\overline{Y_{A}}} \left( 1+ \left| \log(|\la_{l}|^2) \right| +\left| \log(|\la_{d}|^2) \right| \right)
\vol_{\overline{Y_{A}}}
.\] 
It remains only to note that the logarithmic singularities are integrable.. 
In local coordinates near the boundary, we can choose variables $r_{i} \in (0,c_{i}]$ that measure the distance to the boundary, 
so that $\abs{\la_{e}}$ is, up to multiplication by a bounded non-zero function, of the form $r_1 \cdots r_{k}$. 
Thus the logarithmic terms are locally bounded by finite sums of the form $\sum_{i=1}^{k} \abs{\log(r_{i})}$ near $r_{i}=0$, 
and it suffices to check that such functions are integrable, which implies the absolute integrability of $g$,
\[
\int_{0}^{c_{i}} \log(r_{i}) \dd{r_{i}} = c_{i} (\log(c_{i}) - 1) < \infty. \qedhere 
\]
\end{proof}

\graphicspath{{Images/}}

\section{From product space to parameter space - Theorem \ref*{thm:prodToPara}}\label{sec:prodToPara}
We prove the second step of \cref{thm:EqualityOfIntegrals} by showing that, in each term in the auxiliary integral,
the complex variables can be integrated out, yielding an expression purely in terms of the graph matrix $M_{m}$.

For this we will again need the matrix  $A_{m}$, obtained from $A$ by replacing its $l$-th row with the $m$-th row of $A$, for fixed 
$l,m \in[2n]$. For $m = l$, this is simply the original matrix $A$.
From $A_{m}$ we define the Laplacian matrix $L_{m} = A_{m}^{T} \diag(\alpha) A_{m}$, the to $A_{m}$ associated graph matrix $M_{m}$,
and the first Symanzik Polynomial $\Psi_{m} = \det(M_{m})$; see \cref{def:graphMatrices}.

\begin{theorem}\label{thm:prodToPara}
	Let $n \geq 3$ odd, $l,d,m \in [2n]$ with $l \neq d$, and let $A \in \R^{2n \times  n}$ such 
	that $\rank(A_{i,\bullet}) = n$ for all $i \in [2n]$ and no row of $A$ is equal to $0$.
	Let $\epsilon > 0$ and define the truncated affine simplices
	\[
		\sigma_{2n}^{\epsilon} := \{\alpha \in \R^{2n}_{\geq \epsilon} \mid \sum_{i=1}^{2n} \alpha_{i} = 1\} 
		\qq{and} \widetilde{\sigma}_{2n}^{\epsilon} := \{x \in \R_{\geq 0}^{2n} \mid 
	x_{i} \leq \frac{1}{\epsilon} \forall  i \in [2n] \text{ and } \sum_{i=1}^{2n} \frac{1}{x_{i}} = 1\}
	.\]
	Then:
	\begin{multline*}
		\Gamma(2n) \int_{\sigma_{2n}^{\epsilon}} \int_{\C^{n}}
		\frac{ \bigwedge_{e \neq l,d} ( \overline{\la_{e}} \dd{\la_{e}} + \la_{e} \dd{\overline{\la_{e}}}) 
		\wedge \dd{\la_{d}} \wedge \dd{\overline{\la_{d}}}}
		{(\sum_{e \neq l} \alpha_{e} \abs{\la_{e}}^2 + \alpha_{l} \abs{\la_{m}}^2 + \alpha_{l})^{2n}} \Omega(\alpha)\\
		=\int_{\widetilde{\sigma}_{n}^{\epsilon}} \frac{ (-1)^{\frac{n + 1}{2}}(2\pi i)^{n}}{\Psi_{m}^{n}} \frac{1}{x_{d}}
		\sum_{S,T \in \mathcal{I}_{d}^{l}} \sgn(S,T) 
		\det(A_{Sd,\bullet}) \det(A_{Td,\bullet}) \perm((\adj(M_{m}))_{S,T}) \Omega(x)
	.\end{multline*}
	Here, the indexing set is given as
	$\mathcal{I}_{d}^{l} = \{(S,T) \in ([2n] \setminus \{l,d\})^2 \mid \abs{S} = \abs{T} = n-1 \text{ and } S \cap T = \emptyset\}$.
\end{theorem}

To prove \cref{thm:prodToPara} we start on the left-hand side (LHS) and rewrite the integrand using two lemmata:
\begin{lemma}\label{lem:wedgeToDeriv}
Let $L_{m} = A^{T}_{m} \diag(\alpha) A_{m}$ and $p = (p_1,\ldots,p_{2n})^{T}$. Then
\begin{multline*}
	\Gamma(2n)
	\frac{ \bigwedge_{e \neq l,d} ( \overline{\la_{e}} \dd{\la_{e}} + \la_{e} \dd{\overline{\la_{e}}}) 
	\wedge \dd{\la_{d}} \wedge \dd{\overline{\la_{d}}}}
	{(\sum_{e \neq l} \alpha_{e} \abs{\la_{e}}^2 + \alpha_{l} \abs{\la_{m}}^2 + \alpha_{l})^{2n}}\\
	=\bigg[
	\bigwedge_{e\neq l,d} (\frac{\partial }{\partial \overline{p_{e}}} \dd{\la_{e}} - \frac{\partial }{\partial p_{e}}  \dd{\overline{\la_{e}}}) 
	\wedge \dd{\la_{d}} \wedge \dd{\overline{\la_{d}}} \bigg]
	\left. \left[ \frac{1} 
	{(z^{T} L_{m} \overline{z}  + p^{T} A_{m} z + \overline{p}^{T} A_{m} \overline{z} + \alpha_{l})^{2}}
	\right] \right|_{p,\overline{p} = 0}
\end{multline*}
\end{lemma}

\begin{proof}
	To move a factor from the numerator to the denominator observe that for $D \in \C$ and $k > 1$ 
	\[
		\frac{\la_{e}}{D^{k}} =  -\frac{1}{k-1} \frac{\partial }{\partial p} 
		\frac{1}{\left(D + p \la_{e}\right)^{k-1}} \bigg|_{p,\overline{p}=0}
	.\] 
	Whenever no $\frac{\partial }{\partial p_{e}}$ or 
	$\frac{\partial }{\partial \overline{p_{e}}} $ appears $ p_{e} \la_{e}$ or 
	$\overline{p_{e}}\, \overline{\la_{e}}$ can still be added to the denominator as setting $p_{e}$ or $\overline{p_{e}}$ to zero removes these terms.
	Applying the above identity to each $\la_{e}$ and adding $p_{l} \la_{m}$ and $\overline{p_{l}} \overline{\la_{m}}$ to the denominator
	the LHS of \cref{lem:wedgeToDeriv} is equal to
	\begin{multline*}
		\frac{\Gamma(2n)}{\prod_{k=1}^{2n-2} (2n-k)} 
		\Bigg[\bigwedge_{e\neq l,d} 
		\left(\frac{\partial }{\partial \overline{p_{e}}} \dd{\la_{e}} - \frac{\partial }{\partial p_{e}}  \dd{\overline{\la_{e}}}\right) 
		\wedge \dd{\la_{d}} \wedge \dd{\overline{\la_{d}}} \Bigg]\\
		\cdot \Bigg[\frac{1}
		{(\sum\limits_{e \neq l} \alpha_{e} \abs{\la_{e}}^2 +\alpha_{l} \abs{\la_{m}}^2 +  \sum\limits_{e \neq l} p_{e} \, \la_{e} + p_{l} \la_{m} +
		\sum\limits_{e \neq l} \overline{p_{e}} \, \overline{\la_{e}} + \overline{p_{l}}\, \overline{\la_{m}}+ \alpha_{l})^{2}}\Bigg]
		\Bigg|_{p,\overline{p} = 0}
	\end{multline*}
	where taking $2n-2$ derivatives reduces the exponent from $2n$  to $2$ and the sign disappears as the number of derivatives is even.
	To finish the proof observe that
	\[
		\sum_{e \neq l} \alpha_{e} \la_{e}(z) \overline{\la_{e}(z)} + \alpha_{l} \la_{m}(z) \overline{\la_{m}(z)}  = z^{T} A_{m}^{T} D_{m} A_{m} \overline{z} 
		= z^{T} L_{m} \overline{z},\quad
		\sum_{e \neq l}  p_{e} \la_{e}(z) + p_{l} \la_{m}(z) = p^{T} A_{m} z
	\] 
	as well as $\prod_{k=1}^{2n-2} (2n-k) = \Gamma(2n)$.
\end{proof}

\begin{lemma}\label{lem:wedgeToDet}
	\begin{align*}
		\bigwedge_{\substack{e \neq l,d}} \left(\frac{\partial }{\partial \overline{p_{e}}}\dd{\la_{e}} + 
		\frac{\partial}{\partial p_{e}} \dd{\overline{\la_{e}}} \right)
		\wedge \dd{\la_{d}} \wedge \dd{\overline{\la_{d}}}
		= (-1)^{\frac{n^2 - n}{2}} \det(P_{d}^{l}) \Omega(z) \qq{where} \Omega(z) = \bigwedge_{i \in [n]} \dd{z_{i}} \wedge \dd{\overline{z_{i}}}
	\end{align*}
	where $P_{d}^{l}$ is the matrix
	\[
		P_{d}^{l} := \begin{pmatrix}[c|c]
			(\diag(\partial p) \cdot A)^{ld,\bullet}  & (\diag(\partial \overline{p}) \cdot A)^{ld,\bullet} \\ \hline
			a_{d} & 0\\
			0 & a_{d}
		\end{pmatrix} 
	\]
	where the rows $l$ and $d$ have been removed from $\diag(\partial p ) \cdot A$
	and $\partial p $ and $\partial \overline{p}$ are the vectors  
	$ \partial p = (\frac{\partial }{\partial p_1},\ldots,\frac{\partial }{\partial p_{2n}})^{T}$
	and $\partial \overline{p} = (\frac{\partial }{\partial \overline{p_1}},\ldots,\frac{\partial }{\partial \overline{p_{2n}}})^{T}$.
\end{lemma}

\begin{proof}
	Let $I = (i_1,\ldots,i_{2n-2}) := [2n] \setminus \{l,d\}$. Denote the $j$-th row of $P_{d}^{l}$ by $P_{j}$ 
	and let $Z$ denote the vector  $(\dd{z_1},\ldots,\dd{z_{n}}, \dd{\overline{z}_{1}},\ldots, \dd{\overline{z}_{n}})^{T}$.
	Then
	\[
		\frac{\partial }{\partial \overline{p_{i_{j}}}}  \dd{\la_{i_{j}}(z)} + \frac{\partial }{\partial p_{i_{j}}}  
		\dd{\overline{\la_{i_{j}}(z)}} =  P_{j} Z
		\qq{and} \dd{\la_{d}(z)} = P_{2n-1} Z
		\qq{and} \dd{\overline{\la_{d}(z)}} = P_{2n} Z
	.\] 
	From this it follows that,
	\[
		\bigwedge_{e \neq l,d} \left(\frac{\partial }{\partial \overline{p_{e}}}\dd{\la_{e}} + 
		\frac{\partial}{\partial p_{e}} \dd{\overline{\la_{e}}} \right)
		\wedge \dd{\la_{d}} \wedge \dd{\overline{\la_{d}}}
		= \bigwedge_{i \in [2n]} P_{i} Z
		= \sum_{\sigma \in \mathbb{S}_{2n}} \bigwedge_{i \in [2n]} P_{i \sigma(i)} Z_{\sigma(i)}
	,\]
	Here we use multilinearity of the wedge product in the components of $Z$. 
	Expanding the product amounts to summing over all possible ways of selecting, for each factor in the wedge product, 
	one element of $Z$, under the condition that the chosen elements are pairwise distinct. 
	Such selections are naturally indexed by permutations $\sigma \in \mathbb{S}_{2n}$, 
	where in the $i$-th factor we choose $Z_{\sigma(i)}$, whose coefficient is $P_{i \sigma(i)}$.

	Next, we may factor out the scalar coefficients and reorder the wedge factors so that the terms $Z_{\sigma(i)}$ appear 
	in the order induced by $Z$. The required reordering is given by $\sigma^{-1}$, which contributes the sign $\sgn(\sigma)$. 
	Consequently, we obtain
	\[
		\sum_{\sigma \in \mathbb{S}_{2n}} \bigwedge_{i \in [2n]} P_{i \sigma(i)} Z_{\sigma(i)} = 
		\sum_{\sigma \in \mathbb{S}_{2n}} \sgn(\sigma) \prod_{i \in [2n]} P_{i \sigma(i)} \bigwedge_{i \in [2n]} Z_{i}
		= (-1)^{\frac{n^2 - n}{2}} \det(P_{d}^{l}) \Omega(z)
	,\]
	where the sign $(-1)^{\frac{n^2 - n}{2}}$ accounts for rewriting $\bigwedge_{i \in [2n]} Z_{i}$ in the order defining $\Omega(z)$
\end{proof}

Using \cref{lem:wedgeToDeriv,lem:wedgeToDet} and the fact that for $n$ odd $(-1)^{\frac{n^2 - n}{2}} = (-1)^{\frac{n+1}{2}}$,
we get the following intermediate result.
\begin{proposition}
	Let $s = 2$, then the LHS of \cref{thm:prodToPara} is equal to 
\begin{equation}\label{eq:prodToParaRegulator}
	(-1)^{\frac{n+1}{2}} \int_{\sigma_{2n}^{\epsilon}} \int_{\C^{n}} \left. \det(P_{d}^{l}) \frac{1}
	{(z^{T} L_{m} \overline{z}  + p^{T} A_{m} z + \overline{p}^{T} A_{m} \overline{z} + \alpha_{l})^{s}}
	\right|_{p,\overline{p} = 0} \Omega(z)
	\Omega(\alpha)
.\end{equation}
\end{proposition}
To move the limits $p, \overline{p} = 0$ and the derivatives past the complex integral we need to replace the exponent $s=2$ of the denominator
by a regulator $s \in \C$ such that $\Re(s) > 2n$. All calculations can then be done in that open neighbourhood of $s$ and in the end
analytic continuation can be used to set $s=2$ again.
To use the dominated convergence theorem (\cref{thm:dominatedConvergenceTheorem}) and the Leibniz integral rule (\cref{thm:leibnizRule})
we bound the integrand:
\begin{lemma}\label{lem:prodToParaBound}
	Let $P$ be a differential operator in $\{p,\overline{p}\}$ of degree $(k,l)$ ($k$ holomorphic and $l$ anti-holomorphic derivatives).
	Let $\abs{p_{i}}^2 < \frac{\epsilon^2}{8 n}$, then the following bound holds:
	\[
		\abs{P \frac{1}
		{(z^{T} L_{m} \overline{z}  + p^{T} A_{m} z + \overline{p}^{T} A_{m} \overline{z} + \alpha_{l})^{s}}}
		\leq g(z) := \left( \frac{2}{\epsilon} \right)^{s+k+l} 
		\frac{\Gamma(s+k+l) \abs{F(z,\overline{z})}}{(|A_{m} z|^2 + 1)^{s+k+l}}
	,\] 
	where $F$ is some homogeneous polynomial in $\{z,\overline{z}\}$ of degree $(k,l)$.
	Further, $g: \C^{n} \to \R$ is absolutely integrable.
\end{lemma}

\begin{proof}
Apply the differential operator to the fraction. For the numerator one gets some homogeneous polynomial in $\{z_{i},\overline{z_{i}}\}$ 
of degree $(k,l)$, where the coefficients arise from $A_{m}$. We denote it by $F(z,\overline{z})$:
\begin{equation}\label{eq:derivativesTakenF}
	\abs{P \frac{1}
	{(z^{T} L_{m} \overline{z}  + p^{T} A_{m} z + \overline{p}^{T} A_{m} \overline{z} + \alpha_{l})^{s}}}
	= \frac{\Gamma(s+k+l) \abs{F(z,\overline{z})}}
	{\abs{z^{T} L_{m} \overline{z}  + p^{T} A_{m} z + \overline{p}^{T} A_{m} \overline{z} + \alpha_{l}}^{s+k+l}}
.\end{equation}
To bound the denominator independently of $p$ and $\overline{p}$ we complete the square for each line of $A_{m}$ separately. 
For this denote $y = A_{m} z$ and $\beta = (\alpha_1,\ldots,\alpha_{l-1},\alpha_{m},\alpha_{l+1},\ldots,\alpha_{2n})^{T}$:
\begin{align*}
	z^{T} A_{m}^{T} D_{m} A_{m} \overline{z}  + p^{T} A_{m} z + \overline{p}^{T} A_{m} \overline{z} + \alpha_{l}
	&= \alpha_{l} + \sum_{i=1}^{2n} \beta_{i} y_{i} \overline{y_{i}} + p_{i} y_{i} + \overline{p_{i}} \overline{y_{i}}\\
	&= \alpha_{l} + \sum_{i = 1}^{2n} \frac{\beta_{i}}{2} \abs{y_{i} + \frac{2 \overline{p_{i}}}{\beta_{i}}}^2
	+ \sum_{i=1}^{2n} \frac{\beta_{i}}{2} \abs{y_{i}}^2 - \sum_{i=1}^{2n} \frac{2\abs{p_{i}}^2}{\beta_{i}}\\ 
	&\geq \epsilon + \sum_{i=1}^{2n} \frac{\beta_{i}}{2} \abs{y_{i}}^2 - 2\sum_{i=1}^{2n} \frac{\abs{p_{i}}^2}{\beta_{i}} \\
	&\geq \frac{\epsilon}{2} |y|^2 + \epsilon  - 4n \frac{\epsilon^2}{8n \epsilon} = \frac{\epsilon}{2} (|A_{m} z|^2) + 1)
,\end{align*}
where we used that $\alpha_{l} > \epsilon$, $\beta_{i} > \epsilon$ and $\abs{p_{i}}^2 < \frac{\epsilon^2}{16n} $ for all $i \in [2n]$.
Thus we bound \cref{eq:derivativesTakenF} as
\[
	\frac{\Gamma(s+k+l) \abs{F(z,\overline{z})}}{\abs{z^{T} L_{m} \overline{z}  + p^{T} A_{m} z + \overline{p}^{T} A_{m} \overline{z} + \alpha_{l}}^{s+k+l}}
	\leq g(z) := \left( \frac{2}{\epsilon} \right)^{s+k+l}  \frac{\Gamma(s+k+l) \abs{F(z,\overline{z})}}{(|A_{m} z|^2 + 1)^{s+k+l}}
.\] 
This finishes the first part. What remains is to show that $g$ is absolutely integrable on $\C^{n}$.

To this end, consider the integral of $g$ over $\C^{n}$ and pull it back to complex projective space via $[y_0:y] \mapsto \frac{y}{y_0}$:
\begin{multline*}
	\int_{\C^{n}} \left( \frac{2}{\epsilon} \right)^{s+k+l}  \frac{\Gamma(s+k+l) \abs{F(z,\overline{z})}}{(|A_{m} z|^2 + 1)^{s+k+l}} \Omega(z)\\
	= \int_{\CP^{n}} \left(\frac{2}{\epsilon}\right)^{s+k+l} \frac{\Gamma(s+k+l) \abs{F(y,\overline{y})}}{(|A_{m} y|^2 + |y_0|^2)^{s+k+l}}
	\frac{|y_0|^{2s+2k+2l}}{y_0^{k} \overline{y_0}^{l} |y_0|^{2n+2}} 
	\left( \sum_{i,j=0}^{n}  y_{i} \overline{y_{j}} \left( \dd{y_{j}} \wedge \dd{\overline{y_{i}}} \right) 
	\bigwedge_{e \neq i,j} \dd{y_{e}} \wedge \dd{\overline{y_{e}}}\right)
.\end{multline*}
The expression in parentheses is the standard volume form on $\CP^{n}$.
The crucial observation is that the integrand is smooth. Indeed,
$|A_{m} y |_{2}^2 + \abs{y_{0}}^2 \neq 0$ for all $ [y_0,y] \in \CP^{n}$,
since not all homogeneous coordinates $y_0,\ldots,y_{n}$ vanish simultaneously (as $0 \not\in \CP^{n}$ ), and since $A_{m}$ has rank $n$ 
as $\rank(A_{l,\bullet}) = n$ by definition of $A$. Consequently, the denominator never vanishes.

Moreover $F(y,\overline{y})$ is a polynomial, hence smooth and
\[
	\frac{|y_0|^{2s+2k+2l}}{y_0^{k} \overline{y_0}^{l} |y_0|^{2n+2}} 
	= y_0^{s-n-1+l}\ \overline{y_0}^{s-n-1+k}
,\] 
which is also smooth, since $\Re(s) > n+1$. Hence, the entire integrand defines a smooth function on $\CP^{n}$.
As $\CP^{n}$ is compact and the integrand is smooth, it is bounded. Since $\CP^{n}$ has finite volume,
the integral is finite. This proves the second part of the statement.
\end{proof}
By dominating the integrand of \cref{eq:prodToParaRegulator} with the function $g$ from \cref{lem:prodToParaBound}, 
corresponding to the differential operator $P = P_{d}^{l}$, we may apply the Dominated Convergence \cref{thm:dominatedConvergenceTheorem} 
to interchange the limits and the integral.
Subsequently, using the Leibniz integral rule (\cref{thm:leibnizRule}) with the appropriate dominating functions $g$, 
the derivatives can be passed through the integral one at a time. Hence, the left-hand side of \cref{thm:prodToPara} is equal to
\[
	(-1)^{\frac{n+1}{2}} \int_{\sigma_{2n}^{\epsilon}}\lim_{p \to 0} \det(P_{d}^{l}) \int_{\C^{n}} 
	\frac{1}{(z^{T} L_{m} \overline{z}  + p^{T} A_{m} z + \overline{p}^{T} A_{m} \overline{z} + \alpha_{l})^{s}}
	\Omega(z) \Omega(\alpha)
.\] 
As $L_{m}$ is real symmetric we can consider its Cholesky decomposition $L_{m} = C^{T} C$.
Applying the change of variables $y := C z$, we obtain
\[
	\frac{(-1)^{\frac{n+1}{2}}}{\det(L_{m})} \int_{\sigma_{2n}^{\epsilon}} \lim_{p \to 0} \det(P_{d}^{l}) \int_{\C^{n}} \frac{1}
	{(y^{T} \overline{y}  + p^{T} A_{m} C^{-1} y +\overline{p}^{T} A_{m} C^{-1} \overline{y} + \alpha_{l})^{s}}
	\Omega(y) \Omega(\alpha)
,\]
where we used that $\abs{\det(C)}^2 = \det(C^{T} C) = \det(L_{m})$. Next, we complete the square in the denominator:
\begin{multline*}
	y^{T} \overline{y}  +  p^{T} A_{m} C^{-1} y + \overline{p}^{T} A_{m} C^{-1} \overline{y} + \alpha_{l} \\
	= (y + (C^{-1})^{T} A_{m}^{T} \overline{p})^{T} (\overline{y} + (C^{-1})^{T} A_{m}^{T} p) + 
	(\alpha_{l} - \overline{p}^{T} A_{m} C^{-1} (C^{-1})^{T} A_{m}^{T} p)
,\end{multline*}
where we used that $p^{T} A_{m} C^{-1} y = y^{T} (C^{-1})^{T} A_{m}^{T} p$.
Moreover, since the transpose commutes with the inverse, we have  $C^{-1} (C^{-1})^{T} = (C^{T} C)^{-1} = L_{m}^{-1}$.
Substituting the completion of the square and changing variables $w := y + (C^{-1})^{T} A_{m}^{T} \overline{p}$, we obtain
\begin{equation}\label{eq:prodToParaIntegralToIntegrate}
	\frac{(-1)^{\frac{n+1}{2}}}{\det(L_{m})} \int_{\sigma_{2n}^{\epsilon}} \lim_{p \to 0} \det(P_{d}^{l}) \int_{\C^{n}} \frac{1}
	{(w^{T} \overline{w} + (\alpha_{l} - \overline{p}^{T} A_{m} L_{m}^{-1} A_{m}^{T} p))^{s}}
	\Omega(w) \Omega(\alpha)
.\end{equation}
The complex integral is then a repeated application of the following one-dimensional integral:
\begin{lemma}\label{lem:complexFracInt}
	Let $B \in \C$ such that $\Re(B) > 0$ and  $k > 1$, then
	\[
		\int_{\C} \frac{1}{(\abs{z}^2 + B)^{k}} \dd{z} \wedge \dd{\overline{z}}
		= \frac{2 \pi i}{k-1} \frac{1}{B^{k-1}}
	.\] 
\end{lemma}
\begin{proof}
	To integrate change to polar coordinates $z = r e^{i \varphi}$ to obtain
	\[
		\int_{0}^{\infty} \int_{0}^{2\pi} \frac{2i r}{r^2 + B} \dd{\varphi} \wedge \dd{r} = 2 \pi i \int_{0}^{\infty} \frac{1}{(u+B)^{k}} \dd{u}
		= \frac{2 \pi i}{k-1} \frac{1}{B^{k-1}}
	\] 
	where the first step used the substitution $u = r^2$ and the second step used the fact that $\Re(y) > 0$ to rule out the existence
	of poles along $u \in [0,\infty)$.
\end{proof}
For our situation note that $\overline{p} A_{m} L_{m}^{-1} A_{m}^{T} p$ is a homogeneous quadratic polynomial in $p$ with coefficients in $\Q[\alpha]$.
Hence, there exists $0 < \delta < 1$ such that, whenever $\abs{p_{i}}^2 < \delta$, we have
\[
	\Re(\overline{p} A_{m} L_{m}^{-1} A_{m}^{T} p) < \frac{\alpha_{l}}{2}
.\]
Since we consider the limit $p \to 0$, we may restrict $p$ to this neighbourhood.

Under these conditions, \cref{lem:complexFracInt} can be applied repeatedly to the integral in \cref{eq:prodToParaIntegralToIntegrate},
as $s > 2n$ and the bound described above, ensures that 
the real part of the remaining terms in the denominator is positive. Consequently, the LHS of \cref{thm:prodToPara} is equal to
\begin{equation}\label{eq:prodToParaPostIntegration}
	(-1)^{\frac{n + 1}{2}} \frac{(2 \pi i)^{n}}{\det(L_{m})} \frac{\Gamma(s-n)}{\Gamma(s)}
	\int_{\sigma_{2n}^{\epsilon}} \lim_{p \to 0} \det(P_{d}^{l}) 
	\frac{1}{(\alpha_{l} - \overline{p} A_{m} L_{m}^{-1} A_{m}^{T} p)^{s-n}} \Omega(\alpha)
,\end{equation}
Next we apply the differential operator and subsequently take the limits.
Recall the structure of $P_{d}^{l}$ 
\[
	P_{d}^{l} := \begin{pmatrix}[c|c]
		(\diag(\partial p) \cdot A)^{ld,\bullet}  & (\diag(\partial \overline{p}) \cdot A)^{ld,\bullet} \\ \hline
		a_{d} & 0\\ 0 & a_{d}
	\end{pmatrix} 
.\] 
Observe that the first $2n-2$ rows of $P_{d}^{l}$ can be written in the form $(\partial p_{i} a_{i}, 0) + (0, \partial \overline{p}_{i}a_{i})$.
Let $I := (i_1,\ldots,i_{2n-2}) = [2n] \setminus \{l,d\}$.
By multilinearity of the determinant, we may expand $\det(P_{d}^{l})$ accordingly to obtain
\[
	\det(P_{d}^{l}) = \sum_{S,T \in \mathcal{I}_{d}^{l}} \det \begin{pmatrix}[c|c]
		\partial p_{i_1} \cdot  a_{i_1} \cdot  \mathbbm{1}_{i_1 \in S} & \partial \overline{p}_{i_1} \cdot  a_{i_1} \cdot  \mathbbm{1}_{i_1 \in T}\\
		\vdots & \vdots\\
		\partial p_{i_{2n-2}} \cdot  a_{i_{2n-2}} \cdot  \mathbbm{1}_{i_{2n-2} \in S} & 
		\partial \overline{p}_{i_{2n-2}} \cdot  a_{i_{2n-2}} \cdot  \mathbbm{1}_{i_{2n-2} \in T}\\
		a_{d} & 0\\
		0 & a_{d}
	\end{pmatrix} 
,\]
where $\mathcal{I}_{d}^{l} = \{(S,T) \in ([2n] \setminus \{l,d\})^2 \mid \abs{S} = \abs{T} = n-1 \text{ and } S \cap T = \emptyset\}$.
Crucially in each of the first $2n-2$ rows, at most one of the two characteristic functions can be equal to $1$.
Consequently, after reordering rows, each matrix in the sum can be written in block form: we permute the rows corresponding to $S$, together with the row
$(a_{d}, 0)$, to the top, and the rows corresponding to $T$ to the bottom.
The sign of this permutation is $\sgn(S,T) (-1)^{n-1}$. Using the multiplicativity of the determinant for block matrices, we obtain
\[
	\det(P_{d}^{l}) = \sum_{S,T \in \mathcal{I}_{d}^{l}} \sgn(S,T) (-1)^{n-1} \det \begin{pmatrix}
		\partial p_{s_1} \cdot  a_{s_1}\\
		\vdots\\
		\partial p_{s_{n-1}} \cdot  a_{s_{n-1}}\\
		a_{d}
		\end{pmatrix} \det \begin{pmatrix}
		\partial \overline{p}_{t_1} \cdot  a_{t_1}\\
		\vdots\\
		\partial \overline{p}_{t_{n-1}} \cdot  a_{t_{n-1}}\\
		a_{d}
	\end{pmatrix} 
.\] 
By linearity of the determinant, the factors $\partial p_{s_{j}}$ and $\partial \overline{p}_{t_{j}}$ may be extracted, yielding
\[
	\det(P_{d}^{l}) = \sum_{S,T \in \mathcal{I}_{d}^{l}} \sgn(S,T) (-1)^{n-1} 
	\det(A_{S d,\bullet}) \det(A_{T d,\bullet}) 
	\prod_{i \in T} \frac{\partial }{\partial \overline{p_{i}}} \prod_{j \in S} \frac{\partial }{\partial p_{j}} 
.\]
Looking back at \cref{eq:prodToParaPostIntegration} we can now apply the differential operator to the rational function 
$1 /(\alpha_{l} - \overline{p} A_{m} L_{m}^{-1} A_{m}^{T} p)^{k}$ for $k >0$.
Observe that taking the derivative with respect to a variable $p_{j}$ produces
\[
	\frac{\partial }{\partial p_{j}} \frac{1}{(\overline{p}^{T} A_{m} L_{m}^{-1} A_{m}^{T} p + \alpha_{l})^{k} }
	= -k \frac{\overline{p}( A_{m} L_{m}^{-1} A_{m}^{T})_{\bullet,j}}
	{(\overline{p}^{T} A_{m} L_{m}^{-1} A_{m}^{T} p + \alpha_{l})^{k+1} }
.\] 
Such a term can survive in the limit $p\to 0$ only if it is differentiated with respect to some $p_{i}$,
since otherwise the entire expression vanishes in the limit. 
Hence, taking derivatives and then limits amounts to summing over all perfect matchings between elements of $S$ and elements of $T$.
Each perfect matching corresponds to a permutation of $\sigma \in \mathbb{S}_{n}$, so the result can be written as a sum over all permutations.
This mechanism turns a differential operator applied to a quadratic expression into a permanent:
\begin{lemma}
\begin{align*}
	\lim_{p \to 0} \prod_{i \in T} \frac{\partial }{\partial \overline{p_{i}}}  \prod_{j \in S} \frac{\partial }{\partial p_{j}} 
	\frac{1}{(\overline{p}^{T} A_{m} L_{m}^{-1} A_{m}^{T} p + \alpha_{l})^{s-n} }
	&= \frac{(-1)^{n-1} \Gamma(s-1)}{\Gamma(s-n) \alpha_{l}^{s-1}}
	\sum_{\sigma \in \mathbb{S}_{n}} \prod_{i \in [n]} \left(A_{m} L_{m}^{-1} A_{m}^{T}\right)_{s_i t_{\sigma(i)}} \\
	&= \frac{(-1)^{n-1} \Gamma(s-1)}{\Gamma(s-n) \alpha_{l}^{s-1}}
	\perm\left(A_{m} L_{m}^{-1} A_{m}^{T}\right)_{S,T}
.\end{align*}
\end{lemma}
Here, the sign and the Gamma-factor arise from taking $n-1$ derivatives.
Substituting the expansion of $\det(P_{d}^{l})$ into the integrand of \cref{eq:prodToParaPostIntegration} and combining it with the above, we arrive at
\begin{multline}\label{eq:prodToParaPermRegulator}
	\frac{\Gamma(s-n)}{\Gamma(s)} \lim_{p \to 0} \det(P_{d}^{l}) 
	\frac{1}{(\alpha_{l} - \overline{p} A_{m} L_{m}^{-1} A_{m}^{T} p)^{s-n}}\\
	= \frac{\Gamma(s-1)}{\Gamma(s) \alpha_{l}^{s-1}} 
	\sum_{S,T \in \mathcal{I}_{d}^{l}} \sgn(S,T) 
	\det(A_{Sd,\bullet}) \det(A_{Td,\bullet}) \perm\left(A_{m} L_{m}^{-1} A_{m}^{T}\right)_{S,T}
\end{multline}
We now use analytic continuation to return to the case $s=2$.
Observe that both the expression in \cref{eq:prodToParaRegulator} and the expression in \cref{eq:prodToParaPermRegulator} are well defined at $s=2$. 
Moreover, we have shown that these two expressions agree for all $s$ with $\Re(s) > 2n$.
By analytic continuation, it follows that they must also coincide at $s=2$. 
Substituting $s=2$, we conclude that the left-hand side of \cref{thm:prodToPara} is equal to
\begin{equation}\label{eq:prodToParaPerm}
	\int_{\sigma_{n}^{\epsilon}} \frac{(-1)^{\frac{n +1}{2}}(2 \pi i)^{n}}{\det(L_{m})} \frac{1}{\alpha_{l}} 
	\sum_{S,T \in \mathcal{I}_{d}^{l}} \sgn(S,T) 
	\det(A_{Sd,\bullet}) \det(A_{Td,\bullet}) \perm\left(A_{m} L_{m}^{-1} A_{m}^{T}\right)_{S,T} \Omega(\alpha)\\
.\end{equation}

Finally, we pass to the dual representation. Consider the involution $\varphi: \sigma_{n} \to \sigma_{n}$:
\[
	\varphi(x) = \left(\frac{1}{x_1},\ldots,\frac{1}{x_{2n}}\right)^{T} \qq{under which}
	\varphi^{*} \Omega(\alpha) = \frac{(-1)^{2n-1}}{\prod_{k=1}^{2n} x_{k}^2} \Omega(x)
\] 
The truncated simplex $\sigma_{n}^{\epsilon}$ is being mapped to
\[
	\widetilde{\sigma}_{n}^{\epsilon} := \varphi(\sigma_{n}^{\epsilon}) 
	= \left\{x \in \R^{2n}_{\geq 0} \mid x_{i} \leq \frac{1}{\epsilon} \text{ for all } i \in [2n] \text{ and } 
	\sum_{i=1}^{2n} \frac{1}{x_{i}} = 1  \right\} 
.\] 
and the matrix $L_{m}$ transforms as $\widetilde{L}_{m} := \varphi^{*} L_{m} = A_{m}^{T} D_{x} A_{m}$.
Applying $\varphi$ to \cref{eq:prodToParaPerm}, we obtain
\begin{equation}\label{eq:paraToCanDualRep}
	\int_{\widetilde{\sigma}_{n}^{\epsilon}} \frac{(-1)^{\frac{n +1}{2}+1}}{\det(\widetilde{L}_{m})(2\pi i)^{-n}}
	\sum_{S,T \in \mathcal{I}_{d}^{l}} \sgn(S,T) 
	\det(A_{Sd,\bullet}) \det(A_{Td,\bullet}) \perm\left(A_{m} \widetilde{L}_{m}^{-1} A_{m}^{T}\right)_{S,T} 
	\frac{x_{l} \Omega(x)}{\prod_{k=1}^{2n} x_{k}^2}
.\end{equation}
Next, we rewrite the permanent using \cref{lem:eLeToM} which states that the matrix $A_{m} \widetilde{L}_{m}^{-1} A_{m}^{T}$ 
and  $-D_{x} M_{m}^{-1} D_{x}$ agree off the diagonal.
Since, $S$ and $T$ are disjoint, diagonal entries do not contribute to the relevant submatrix. Hence,
\[
	\perm\left(A_{m} \widetilde{L}_{m}^{-1} A_{m}^{T}\right)_{S,T}
	= \frac{(-1)^{n-1}}{\Psi_{m}^{n-1}} \perm(\adj(M_{m})_{S,T}) \prod_{i \in [2n] \setminus \{l,d\} } x_{i}
,\]
where we used the definition of the permanent together with the fact that $S \sqcup T = [2n] \setminus \{l,d\}$ and that $\det(M_{m}) = \Psi_{m}$.
Recall from \cref{def:graphMatrices} that
\[
	\det(\widetilde{L}_{m}) = \frac{\det(M_{m})}{\prod_{k \in [2n]} x_{k}}= \frac{\Psi_{m}}{\prod_{k \in [2n]} x_{k}}
.\] 
Substituting both identities into \cref{eq:paraToCanDualRep} yields that the left-hand side of \cref{thm:prodToPara} is equal to
\begin{equation*}\label{eq:paraToCanGraphMatrixRep}
\int_{\widetilde{\sigma}_{n}^{\epsilon}} \frac{(-1)^{\frac{n + 1}{2}} (2\pi i)^{n}}{\Psi_{m}^{n}} \frac{1}{x_{d}}
	\sum_{S,T \in \mathcal{I}_{d}^{l}} \sgn(S,T) 
	\det(A_{Sd,\bullet}) \det(A_{Td,\bullet}) \perm(\adj(M_{m})_{S,T}) \Omega(x)
,\end{equation*}
which concludes the proof of \cref{thm:prodToPara}.

\graphicspath{{Images/}}

\section{A new formula for canonical forms - Theorem \ref*{thm:paraToCan}}\label{sec:paraToCan}
In this section we identify the parameter space integrand obtained in \cref{thm:prodToPara} with the canonical forms \cref{def:canonicalInt}, that is
\begin{theorem}\label{thm:paraToCan}
	Let $n \geq 3$ odd, $d\in [2n]$ and $A \in \R^{2n \times  n}$.
	Define for  $l \in [2n] \setminus \{d\}$ the indexing set 
	$\mathcal{I}_{d}^{l} = \{(S,T) \in ([2n] \setminus \{l,d\})^2 \mid \abs{S} = \abs{T} = n-1 \text{ and } S \cap T = \emptyset\}$.
	Then:
	\begin{align*}
	\beta^{2n-1}_{L} = \sum_{\substack{l \in [2n]\\ l \neq  d}} \sum_{S,T \in \mathcal{I}_{d}^{l}} 
	\frac{(-1)^{\frac{n + 1}{2} + d+l+(l<d)}}{\Psi^{n}} \sgn(S,T) 
		\det(A_{Sd,\bullet}) \det(A_{Td,\bullet}) \perm(\adj(M)_{S,T}) \frac{\Omega(x)}{x_{d}}
	\end{align*}
\end{theorem}
Before proving the theorem, we illustrate the new formula in the simplest case:

\begin{eg}\label{eg:threeWheelCanForm}
	Let us compute the canonical form for the three-wheel graph as in \cref{fig:threeWheel}. In this case $n = 3$ and
	there are $6$ pairs $(S,T)$. Modulo the symmetry exchanging $S$ and $T$, three cases remain.
	Their contributions are as given below, where the product of Dodgson polynomials comes from the permanent and the pre-factor
	is given by the sign and the two determinants:
\[
	\begin{array}{rcl}
(S,T) = (\{2,3\},\{4,5\})
&\rightsquigarrow&
\phantom{-}1 \cdot \bigl(\Psi^{2,4}\Psi^{3,5}+\Psi^{3,4}\Psi^{2,5}\bigr),\\
(S,T) = (\{2,4\},\{3,5\})
&\rightsquigarrow&
\phantom{-}0\cdot (\Psi^{2,3} \Psi^{4,5} + \Psi^{3,4} \Psi^{2,5}),\\
(S,T) = (\{2,5\},\{3,4\})
&\rightsquigarrow&
-1 \cdot \bigl(\Psi^{2,3}\Psi^{4,5}+\Psi^{3,5}\Psi^{2,4}\bigr).
\end{array}
\]
Here $\Psi^{i,j}$ denotes the Dodgson polynomials of \cref{app:Dodgson}, and satisfies $\Psi^{i,j} = \Psi^{j,i}$.

	Combining these contributions gives
	\begin{align*}
		\beta_{L_{W_{3}}}^{5} = -10 \frac{\Psi^{3,4} \Psi^{5,2} - \Psi^{3,2} \Psi^{5,4}}{\Psi^{3}} \frac{\Omega(x)}{x_1} 
		= -10 \frac{\Psi \Psi^{(3,5),(2,4)}}{\Psi^3}
		\frac{\Omega(x)}{x_1} = 10\frac{\Omega(x)}{\Psi^2}
	.\end{align*}
	In the second equality we used the Dodgson identity  
	$\Psi^{i_1,i_3} \Psi^{i_2,i_4} - \Psi^{i_1,i_4} \Psi^{i_2,i_3} = \Psi \Psi^{(i_1,i_2),(i_3,i_4)}$ \cite[Eq. 1.33]{golz19} 
	and in the last equality we evaluated 
	$\Psi^{(3,5),(2,4)} = - x_1$.
	Overall this agrees with the result from \cite[Example 4.7]{brown21}.
\end{eg}

To prove the theorem, we first express the two determinants in terms of $M$. This is established by the following lemma.
\begin{lemma}\label{lem:paraToCanDetProd}
	Let $S,T \in \mathcal{I}_{d}^{l}$ and $\Psi = \det(M)$. Then,
	\[
	\det(A_{Sd,\bullet}) \det(A_{Td,\bullet})
	= (-1)^{n+1} \frac{\det(\adj(M)_{Sl, Tl})}{\Psi^{n-1}}	
	.\] 
\end{lemma}
\begin{proof}
	Since, $S \sqcup T \sqcup \{l,d\} = [2n]$ the complimentary minors of $A$ satisfy
	\[
		\det(A_{Sd,\bullet}) = (-1)^{\abs{S_{> d}}} \det(A^{Tl,\bullet}) \qq{and}
		\det(A_{Td,\bullet}) = (-1)^{\abs{T_{> d}}} \det(A^{Sl,\bullet})
	\]
	Consequently,
	\begin{align*}
		\det(A_{Sd,\bullet}) \det(A_{Td,\bullet}) &= (-1)^{\abs{S_{> d}} + \abs{T_{> d}}} \det(A^{Tl,\bullet}) \det(A^{Sl,\bullet})
		= (-1)^{\abs{S_{> d}} + \abs{T_{> d}}} \Psi^{Sl, Tl}\\
		&= (-1)^{\abs{S_{> d}} + \abs{T_{> d}}}
		 (-1)^{\sum_{i \in (S l)} i + \sum_{j \in (T l)} j} \sgn(\pi_{I}) \sgn(\pi_{J}) \frac{\det(\adj(M)_{Sl, Tl})}{\Psi^{n-1}}
	.\end{align*}
	In the second equality we used \cref{lem:detProductAsDodgson} with $\abs{R}= 0$, and 
	in the third equality \cref{lem:dodgsonToAdjugate} with $I = S l$ and  $J = T l$.
	It remains to simplify the overall sign. First observe that
	\[
		(-1)^{\abs{S_{> d}} + \abs{T_{> d}}} = (-1)^{2n - d - (l > d)}
	\qq{and} \sgn(\pi_{I}) \sgn(\pi_{J}) = (-1)^{\abs{S_{> l}} + \abs{T_{> l}}} = (-1)^{2n - l - (l < d)}
	.\]
	Further,
	\[
		(-1)^{\sum_{i \in (S l)} i + \sum_{j \in (T l)} j} = (-1)^{\left(\sum_{i=1}^{2n} i \right)+ l - d} = (-1)^{2 n^2 + n + l + d} = (-1)^{n + l + d}
	\]
	Combining these contributions yields
	\[
		\det(A_{Sd,\bullet}) \det(A_{Td,\bullet})
		= (-1)^{n+1} \frac{\det(\adj(M)_{Sl, Tl})}{\Psi^{n-1}}
	\] 
	which proves the claim..
\end{proof}

Applying \cref{lem:paraToCanDetProd} to the right-hand side of \cref{thm:adjToDodgsonProd} and using that $n$ is odd, gives
\begin{equation}\label{eq:canonicalPermDet}
\sum_{\substack{l \in [2n]\\ l \neq  d}} \sum_{S,T \in \mathcal{I}_{d}^{l}} 
	\frac{(-1)^{\frac{n + 1}{2} + d+l+(l<d)}}{\Psi^{2n-1}} \sgn(S,T) 
	\det(\adj(M)_{Sl, Tl})  \perm(\adj(M)_{S,T}) \frac{\Omega(x)}{x_{d}}
.\end{equation}
We can then apply the following combinatorial identity:
\begin{theorem}\label{thm:adjToDodgsonProd} 
	Let $n \geq 2$, $m > n/2$ and $d \in [2n]$. Let $B$ be a $m \times m$ matrix and denote its entries by $b_{ij}$.
	Further let $I = (i_1,\ldots,i_{2n-1}) := [2n] \setminus \{d\}$. Then
	\begin{multline*}
		(-1)^{\frac{n^2 + n+2}{2}}\sum_{\substack{l \in [2n]\\ l \neq  d}} \sum_{S,T \in \mathcal{I}_{d}^{l}} \sgn(S,T) (-1)^{l+(d<l)} 
		\det(B_{Sl,Tl}) \perm(B_{S,T})\\
			= \sum_{\tau \in \mathbb{S}_{2n-1}} \sgn(\tau) b_{i_{\tau(1)} i_{\tau(2)}}\, b_{i_{\tau(2)} i_{\tau(3)}}
	\cdots b_{i_{\tau(2n-2)} i_{\tau(2n-1)}}\, b_{i_{\tau(2n-1)} i_{\tau(1)}}
		.\end{multline*}
\end{theorem}
Before we prove this identity we finish the proof of \cref{thm:paraToCan}:
We apply \cref{thm:adjToDodgsonProd} to \cref{eq:canonicalPermDet}.
In this setting $m = 3n$ and $B = M^{-1}$. The $(i,j)$-th entry of $M^{-1}$ is then given as $\Psi^{i,j} / \Psi$,
where $\Psi^{i,j} = \det(M^{i,j})$ is a Dodgson polynomial and $\Psi = \det(M)$.
This gives that the right-hand side of \cref{thm:paraToCan} is equal to
\[
	\frac{(-1)^{d+1}}{\Psi^{2n-1}} \sum_{\tau \in \mathbb{S}_{2n-1}} \sgn(\tau) \Psi^{i_{\tau(1)},i_{\tau(2)}}\, \Psi^{i_{\tau(2)},i_{\tau(3)}}
	\cdots \Psi^{i_{\tau(2n-1)},i_{\tau(1)}} \frac{\Omega(x)}{x_{d}}
.\]
Here we observed, that as $n$ is odd, so that $(-1)^{\frac{n^2+n+2}{2}} = (-1)^{\frac{n+1}{2}}$.
Applying \cref{prop:canonicalDodgsonFormula} concludes the proof.

\subsection{Proof of \cref{thm:adjToDodgsonProd}}
For the entire section fix $d \in [2n]$ and let $I  = (i_1,\ldots,i_{2n-1})= [2n] \setminus \{ d\}$.
Denote the ordered set $S\,l$ by $\widetilde{S}$ and the ordered set $T\,l$ by $\widetilde{T}$.
Then, expanding the permanent and the determinant on the left-hand side of \cref{thm:adjToDodgsonProd} we obtain
\begin{equation}\label{eq:dodgsonSum}
	\sum_{l \in [2n] \setminus \{d\}} \sum_{S,T \in \mathcal{I}_{d}^{l}}\sum_{\sigma \in \mathbb{S}_{n}}\sum_{\varrho \in \mathbb{S}_{n-1}} 
	W_{d}(S,T,l,\sigma,\varrho)
\end{equation}
where,
\[
W_{d}(S,T,l,\sigma,\varrho) :=
(-1)^{\frac{n^2 + n+2}{2} + l + (d<l)}\sgn(S,T) \sgn(\sigma)
\prod_{i \in [n]} b_{\widetilde{t}_{i} \widetilde{s}_{\sigma(i)}}
\prod_{j \in [n-1]} b_{s_{j} t_{\varrho(j)}}
.\] 
as the entries of $B$ are given by Dodgson polynomials $b_{i,j}$.
First, we want to understand which terms in this large sum cancel.
To this end we encode each summand by a directed graph:
\begin{definition}
	Let $G(S,T,l,\sigma,\varrho) = (V_{G},E_{G})$ be the directed graph with vertices $ V_{G} := S \sqcup T \sqcup \{l\}$ 
	and edges $E_{G} := \{(\widetilde{t}_i,\sigma(\widetilde{s}_i) \mid i \in [n]\} \cup \{(s_{i},t_{\varrho(i)}) \mid i \in [n-1] \}$.
	Then,
	\[
	\prod_{(u,v) \in E_{G}} b_{u v} = \prod_{i \in [n]} b_{\widetilde{t}_{i} \widetilde{s}_{\sigma(i)}}
	\prod_{j \in [n-1]} b_{s_{j} t_{\varrho(j)}}
	.\] 
\end{definition}

By construction, each vertex of $G$ has one incoming and one outgoing edge and is thus a disjoint union of directed cycles.
Moreover, the induced subgraph on $S \sqcup T$ is bipartite, with edges oriented alternatively from $S \to T$ and from $T \to S$;
hence any cycle contained fully in $S \sqcup T$ has even length.
The vertex $l$ has exactly one incoming edge from $T$ and one outgoing edge to $S$.
If one removes $l$ and replaces these two edges by a single edge, the resulting cycle lies in the bipartite subgraph $S \sqcup T$ and therefore
has even length. It follows that the original cycle containing $l$ has odd length.
We conclude:
\begin{fact}\label{fct:GCycles}
	The graph $G$ is a disjoint union of cycles. It contains exactly one odd cycle, namely the one containing $l$.
	All other cycles (if any) are even and lie in the bipartite subgraph $S \sqcup T$.
\end{fact}

To illustrate this consider the following example:
\begin{eg}
Let $S=\{s_1,s_2\}$, $T=\{t_1,t_2\}$ and let $l$ be the additional vertex.
Take $\sigma=(23)$ and $\varrho=\mathrm{id}$. The corresponding graph $G$ is shown on the left in \cref{fig:egCycleRepr}.
If instead $\widetilde{\sigma}=(23)$ and $\varrho=(12)$, the resulting graph is shown on the right.

\begin{figure}[htbp]
	\centering
	\begin{tikzpicture}
	\begin{pgfonlayer}{nodelayer}
		\node [style=internal vertex] (0) at (-7, -2) {};
		\node [style=internal vertex] (1) at (-5, -2) {};
		\node [style=internal vertex] (2) at (-3, 0) {};
		\node [style=internal vertex] (3) at (-5, 2) {};
		\node [style=internal vertex] (4) at (-7, 2) {};
		\node [style=none] (5) at (-2, 0) {$l$};
		\node [style=none] (6) at (-5, 3) {$s_2$};
		\node [style=none] (7) at (-5, -3) {$t_2$};
		\node [style=none] (8) at (-7, 3) {$s_1$};
		\node [style=none] (9) at (-7, -3) {$t_1$};
		\node [style=internal vertex] (10) at (1, -2) {};
		\node [style=internal vertex] (11) at (3, -2) {};
		\node [style=internal vertex] (12) at (5, 0) {};
		\node [style=internal vertex] (13) at (3, 2) {};
		\node [style=internal vertex] (14) at (1, 2) {};
		\node [style=none] (15) at (6, 0) {$l$};
		\node [style=none] (16) at (3, 3) {$s_2$};
		\node [style=none] (17) at (3, -3) {$t_2$};
		\node [style=none] (18) at (1, 3) {$s_1$};
		\node [style=none] (19) at (1, -3) {$t_1$};
	\end{pgfonlayer}
	\begin{pgfonlayer}{edgelayer}
		\draw [style=arrow, bend left] (0) to (4);
		\draw [style=arrow, bend left] (4) to (0);
		\draw [style=arrow] (1) to (2);
		\draw [style=arrow] (2) to (3);
		\draw [style=arrow] (3) to (1);
		\draw [style=arrow] (12) to (13);
		\draw [style=arrow] (13) to (10);
		\draw [style=arrow] (10) to (14);
		\draw [style=arrow] (14) to (11);
		\draw [style=arrow] (11) to (12);
	\end{pgfonlayer}
\end{tikzpicture}
	\caption{Graphs with distinct cycle structure arising from different $\sigma$ and $\varrho$.}
	\label{fig:egCycleRepr}
\end{figure}
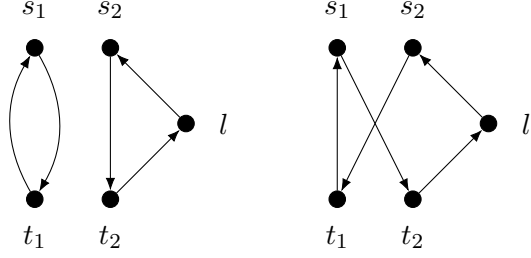

In the first case the graph decomposes into two cycles: an odd cycle containing $l$, and an even cycle entirely contained in $S\sqcup T$.
In contrast, the graph on the right consists of a single cycle of odd length.
\end{eg}

The cancellations in \cref{eq:dodgsonSum} are characterised as follows:
\begin{proposition}\label{prop:cycleCancellation}
	The sum of all terms $W_{d}(S,T,l,\sigma,\varrho)$ in \cref{eq:dodgsonSum} for which $G(S,T,l,\sigma,\varrho)$ contains an even cycle
	vanishes.
\end{proposition}

\begin{proof}
	To prove the result we define an involution $f$ on the space spanned by $(S,T,l,\sigma,\varrho)$ such that $G(S,T,l,\sigma,\varrho)$
	contains an even cycle, and show that $W_{d}(f(S,T,l,\sigma,\varrho)) = - W_{d}(S,T,l,\sigma,\varrho)$.
	It follows that in the sum over all such tuples $(S,T,l,\sigma,\varrho)$ containing even cycles, the terms cancel pairwise and the sum vanishes.

	Let $c = (c_1,\ldots,c_{2r})$ be the even cycle of $G$ whose smallest vertex is minimal amongst all even cycles and 
	orient $c$ such that $c_1$ is the smallest vertex and $c_2 < c_{2r}$.

	Since $c$ is an even cycle, its vertices alternate between $S$ and $T$. Define new ordered sets $S'$ and $T'$ by flipping the vertices of $c$ 
	between $S$ and $T$ by advancing them along the cycle by one. Concretely, if in $c$ a vertex $x \in c \cap S$ is followed by $y \in c \cap T$,
	then, in $T$, $y$ is replaced by $x$ and analogous if $x \in c \cap T$ is followed by $y \in \cap S$, then, in $S$, $y$ is replaced by $x$.
	This is illustrated in \cref{fig:cycleAdvance}.
	
	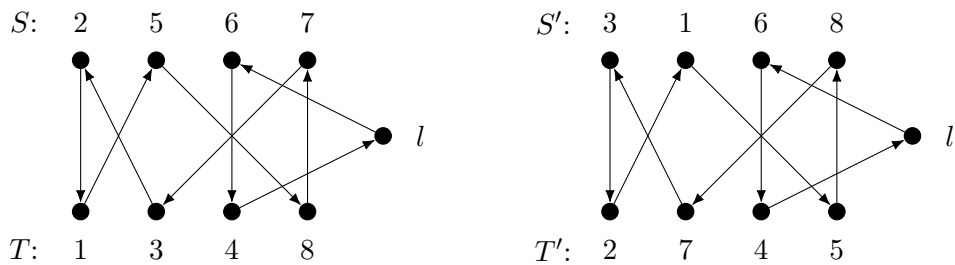
\begin{figure}[htpb]
		\centering
		\begin{tikzpicture}
	\begin{pgfonlayer}{nodelayer}
		\node [style=internal vertex] (0) at (-7, -2) {};
		\node [style=internal vertex] (1) at (-5, -2) {};
		\node [style=internal vertex] (3) at (-5, 2) {};
		\node [style=internal vertex] (4) at (-7, 2) {};
		\node [style=none] (6) at (-5, -3) {3};
		\node [style=none] (7) at (-5, 3) {5};
		\node [style=none] (8) at (-7, -3) {1};
		\node [style=none] (9) at (-7, 3) {2};
		\node [style=internal vertex] (10) at (-1, -2) {};
		\node [style=internal vertex] (11) at (-1, 2) {};
		\node [style=internal vertex] (12) at (-3, -2) {};
		\node [style=internal vertex] (13) at (-3, 2) {};
		\node [style=internal vertex] (14) at (1, 0) {};
		\node [style=none] (15) at (-3, -3) {4};
		\node [style=none] (16) at (-3, 3) {6};
		\node [style=none] (17) at (-1, -3) {8};
		\node [style=none] (18) at (-1, 3) {7};
		\node [style=none] (19) at (2, 0) {$l$};
		\node [style=none] (38) at (-8.5, 3) {$S$:};
		\node [style=none] (39) at (-8.5, -3) {$T$:};
		\node [style=internal vertex] (40) at (7, -2) {};
		\node [style=internal vertex] (41) at (9, -2) {};
		\node [style=internal vertex] (42) at (9, 2) {};
		\node [style=internal vertex] (43) at (7, 2) {};
		\node [style=none] (44) at (7, 3) {3};
		\node [style=none] (45) at (13, -3) {5};
		\node [style=none] (46) at (9, 3) {1};
		\node [style=none] (47) at (7, -3) {2};
		\node [style=internal vertex] (48) at (13, -2) {};
		\node [style=internal vertex] (49) at (13, 2) {};
		\node [style=internal vertex] (50) at (11, -2) {};
		\node [style=internal vertex] (51) at (11, 2) {};
		\node [style=internal vertex] (52) at (15, 0) {};
		\node [style=none] (53) at (11, -3) {4};
		\node [style=none] (54) at (11, 3) {6};
		\node [style=none] (55) at (13, 3) {8};
		\node [style=none] (56) at (9, -3) {7};
		\node [style=none] (57) at (16, 0) {$l$};
		\node [style=none] (58) at (5.5, 3) {$S'$:};
		\node [style=none] (59) at (5.5, -3) {$T'$:};
	\end{pgfonlayer}
	\begin{pgfonlayer}{edgelayer}
		\draw [style=arrow] (0) to (3);
		\draw [style=arrow] (3) to (10);
		\draw [style=arrow] (10) to (11);
		\draw [style=arrow] (11) to (1);
		\draw [style=arrow] (1) to (4);
		\draw [style=arrow] (4) to (0);
		\draw [style=arrow] (40) to (42);
		\draw [style=arrow] (42) to (48);
		\draw [style=arrow] (48) to (49);
		\draw [style=arrow] (49) to (41);
		\draw [style=arrow] (41) to (43);
		\draw [style=arrow] (43) to (40);
		\draw [style=arrow] (50) to (52);
		\draw [style=arrow] (52) to (51);
		\draw [style=arrow] (51) to (50);
		\draw [style=arrow] (12) to (14);
		\draw [style=arrow] (14) to (13);
		\draw [style=arrow] (13) to (12);
	\end{pgfonlayer}
\end{tikzpicture}
		\caption{Passing from $(S,T)$ to $(S',T')$ by advancing the cycle $(1 5 8 7 3 2)$ by one.}
		\label{fig:cycleAdvance}
	\end{figure}

	Let $\alpha,\beta \in \mathbb{S}_{n-1}$ be the permutations that sort $S'$ and $T'$ increasingly and extend them to permutations on $[n]$ by
	fixing  $n$. Then define new sets and permutations
	\[
		\widehat{S} = \alpha(S'), \quad \widehat{T} = \beta(T') \qq{as well as} \widehat{\varrho} := \beta^{-1} \circ \varrho \circ \alpha^{-1},
		\quad \widehat{\sigma} = \alpha^{-1} \circ \sigma \circ \beta^{-1}
	.\]
	This defines a map
	\[
	f(S,T,l,\sigma,\varrho) = (\widehat{S},\widehat{T},l,\widehat{\sigma},\widehat{\varrho})
	.\] 
	To see that $f$ is an involution observe, that when applying $f$ a second time the same cycle $c$ is being chosen and its elements
	are being moved forward another step, that is they are being flipped back to their original positions. The resulting sorting permutations are exactly
	the inverses of the ones from the first application. Thus, $f$ is a self-inverse bijection.

	What remains to show is that
	\[
	W_{d}(\widehat{S},\widehat{T},l,\widehat{\sigma},\widehat{\varrho}) = - W_{d}(S,T,l,\sigma,\varrho)
	.\] 
	Indeed, the transformation does not change the directed graph itself, it only relabels the cycle decomposition. Hence the monomial
	\[
	\prod_{(u,v) \in E_{G}} b_{u v} = \prod_{i \in [n]} b_{\widetilde{t}_{i} \widetilde{s}_{\sigma(i)}}
	\prod_{j \in [n-1]} b_{s_{j} t_{\varrho(j)}}
	.\] 
	is unchanged. It remains to compare signs. Observe that
	\[
	\sgn(\widehat{S},\widehat{T}) \sgn(\widehat{\sigma}) = \sgn(S',T') \sgn(\sigma)
	.\] 
	as the only differences come from the sorting permutations $\alpha$ and $\beta$ which appear twice and thus cancel.
	It remains to consider the sign change from replacing $(S,T)$ by $(S',T')$. If we concatenate the ordered sets $S$ and $T$ 
	then passing to $S'$ and $T'$ amounts to applying the cycle $(c_1 c_2 \ldots c_{2r})$. As this is an even cycle its sign is  $-1$.
	Therefore
	\[
		\sgn(S,T) = - \sgn(S',T') \qq{and hence} 
	W_{d}(\widehat{S},\widehat{T},l,\widehat{\sigma},\widehat{\varrho}) = - W_{d}(S,T,l,\sigma,\varrho)
	,\]
	which concludes the proof.
\end{proof}

We can now turn to proving \cref{thm:adjToDodgsonProd}.
\begin{proof}[Proof of \cref{thm:adjToDodgsonProd}]
	By \cref{prop:cycleCancellation} and \cref{fct:GCycles}, only those tuples $(S,T,l,\sigma,\varrho)$ survive for which $G$ consists of a single 
	odd cycle of length $2n-1$. Let $I = (i_1,\ldots,i_{2n-1}) = [2n] \setminus \{d\}$ and let
	\[
	H = \{(S,T,l,\sigma,\varrho) \mid G(S,T,l,\sigma,\varrho) \text{ is a single cycle}\} 
	.\] 
	To prove the theorem we construct a bijection $f$ between $\mathbb{S}_{2n-1}$ and $H$ and show that $\tau$ and $f(\tau)$ give the same
	product of Dodgson polynomials in their respective settings.

	For $\tau \in \mathbb{S}_{2n-1}$ set
	\[
	l = i_{\tau(1)}, \quad S' = (i_{\tau(2)}, i_{\tau(4)}, \ldots, i_{\tau(2n-2)}), \text{ and } T' = (i_{\tau(3)}, i_{\tau(5)}, \ldots, i_{\tau(2n-1)})
	.\] 
	Let $\alpha, \beta \in \mathbb{S}_{n-1}$ be the permutations that sort $S'$ and $T'$ increasingly and extend them to permutations on $[n]$ 
	by fixing $n$, and let $\gamma = (1 2 \ldots n) \in \mathbb{S}_{n}$. Then define 
	\[
		S = \alpha(S'), \quad T = \beta(T') \qq{as well as} \varrho = \beta \alpha^{-1}, \quad \sigma = \alpha \gamma \beta^{-1}
	.\] 
	We then define the map $f(\tau) = (S,T,l,\sigma,\varrho)$.
	In fact $f$ is a bijection $f: \mathbb{S}_{2n-1} \to H$. First observe that $f(\tau) \in H$: By construction the edges of $G(f(\tau))$ are precisely
	\[
	i_{\tau(1)} \to i_{\tau(2)} \to i_{\tau(3)} \to \cdots \to i_{\tau(2n-1)} \to i_{\tau(1)}
	.\]
	Thus $G(f(\tau))$ is a single directed cycle.
	
	Conversely let $(S,T,l,\sigma,\varrho) \in H$. Since $G$ is a single directed cycle, there is a unique cycle ordering
	\[
	l \to s_1' \to t_1' \to s_2' \to t_2' \to \cdots \to s_{n-1}' \to t_{n-1}' \to l
	,\] 
	obtained by traversing the cycle from $l$. Reading of the vertices gives a unique permutation $\tau \in \mathbb{S}_{2n-1}$ such that
	\[
		i_{\tau(1)} = l, \quad i_{\tau(2j)} = s_{j}', \text{ and } i_{\tau(2j+1)} = t_{j}' \qquad (1 \leq j \leq n-1)
	.\] 
	Applying $f$ to $\tau$ recovers $(S,T,l,\sigma,\varrho)$ showing that $f$ is bijective.

	What remains to show is that
	\begin{equation}\label{eq:WdEqualTau}
		W_{d}(f(\tau)) = \sgn(\tau) b_{i_{\tau(1)} i_{\tau(2)}} b_{i_{\tau(2)} i_{\tau(3)}} \cdots b_{i_{\tau(2n-1)} i_{\tau(1)}}
	.\end{equation}
	Observe that ignoring the sign
	\[
	\pm W_{d}(f(\tau)) = \prod_{i \in [n]} b_{\widetilde{t}_{i} \widetilde{s}_{\sigma(i)}}
	\prod_{j \in [n-1]} b_{s_{j} t_{\varrho(j)}} = \prod_{(u,v) \in E_{G(f(\tau))}} b_{u v}
	.\]
	which is equal (up to the sign) to the right-hand side of \cref{eq:WdEqualTau} as the graph $G(f(\tau))$ is exactly the directed cycle
	\[
	i_{\tau(1)} \to i_{\tau(2)} \to \ldots \to i_{\tau(2n-1)} \to i_{\tau(1)}
	.\] 
	To compute the sign let $(S,T,l,\sigma,\varrho) = f(\tau)$ and $S'$ and $T'$ as well as  $\alpha$ and $\beta$ as in the definition of $f$ above.
	Then
	\begin{align*}
		\sgn(\tau) &= \sgn(l,s'_{1},t'_{1},\ldots,s'_{n-1},t'_{n-1}) = (-1)^{\frac{n^2 - n}{2}+1} \sgn(l,s'_1,\ldots,s'_{n-1},t'_1,\ldots,t'_{n-1}) \\
		&=(-1)^{\frac{n^2 - n}{2}+1} \sgn(l, s_1,\ldots,s_{n-1},t_1,\ldots,t_{n-1}) \sgn(\alpha) \sgn(\beta)\\
		&=(-1)^{\frac{n^2 - n}{2}+1} \sgn(l I ) \sgn(S,T) \sgn(\gamma) \sgn(\sigma) \\
		&= (-1)^{\frac{n^2 + n + 2}{2}} (-1)^{l + (d < l)} \sgn(S,T) \sgn(\sigma)
	\end{align*}
	Here we used that $S = \alpha(S')$, $T = \beta(T')$, $\sigma = \alpha \gamma \beta^{-1}$ and that $l I$ and $l S T$ differ exactly by the permutation
	whose sign is given by $\sgn(S,T)$ as well as that $\sgn(\gamma) = (-1)^{n-1} = 1$ as $\gamma$ is an $n$-cycle.
	This concludes the proof of \cref{thm:adjToDodgsonProd}.
\end{proof}

\appendix
\crefalias{section}{appendix}
\graphicspath{{Images/}}

\section{Dodgson polynomials}\label{app:Dodgson}
Let $n,m \in \N$. In this section, we introduce several matrices and polynomials associated to graphs, 
and more generally to the $2n \times n$-matrices considered throughout this work, or even arbitrary real $n \times m$ matrices.

Originally, these objects originate in graph theory and extend naturally to arbitrary matrices of the above form. 
For this reason, we retain the terminology from the graph-theoretic setting, even though, for a general matrix $A$, 
the objects need not have any direct graph interpretation.

The correspondence between graphs $G$ and matrices $A$ is given by taking $A$ to be a reduced incidence matrix of $G$. 
Roughly speaking, the rows of $A$ correspond to edges of the graph, while the columns correspond to vertices.

\begin{definition}\label{def:graphMatrices}
	Let $A$ be a $n \times m$ matrix, and let $x = (x_1,\ldots,x_{2n})$ be homogeneous coordinates on the positive
	coordinate simplex $\sigma_{n} \in \RP^{n-1}$.

	The Laplacian matrix $L_{A}(x)$ and the graph matrix $M_{A}(x)$ associated to $A$ are defined by
	\[
	L_{A}(x) = A^{T} \diag(x) A \qquad
		M_{A}(x) = \begin{pmatrix}[c|c]
			\diag(x) & - A\\ \hline
			A^{T} & 0
		\end{pmatrix} 
	.\] 
	Furthermore, we define the polynomial
	\[
		\Psi_{A}(x) = \det(M_{A}(x)) =  \det(L_{A}(x^{-1})) \prod_{e \in [n]} x_{e}
	.\] 
	In the graph-theoretic setting $\Psi_{A}(x)$ is the first Symanzik polynomial.
\end{definition}

\begin{remark}
	For brevity, we will often omit the dependence on $A$ and $x$ from the notation and simply write $L$, $M$, and $\Psi$
instead of  $L_{A}(x)$, $M_{A}(x)$, and $\Psi_{A}(x)$. 
The parameters $A$ and $x$ will only be included when necessary for clarity or emphasis.
\end{remark}

As a generalisation of $\Psi$, one may also consider minors of the graph matrix $M$.

\begin{definition}\label{def:dodgsons}
	Let $I, J \subseteq [n + m]$ with $\abs{I} = \abs{J}$. The determinant of the submatrix of $M$ 
	obtained by removing the rows indexed by $I$ and the columns indexed by $J$,
	\[
		\Psi^{I,J} = \det(M^{I,J})
	\]
	is called a Dodgson polynomial of $M$. 
	In the special case where $I = \{i\} $ and $J = \{j\} $ consist of single elements, we write
	$\Psi^{i,j}$ instead of $\Psi^{\{i\},\{j\}}$.

	Recall that the matrix $M$ naturally decomposes into two types of rows and columns: the first $n$ correspond to rows of $A$,
	while the remaining $m$ correspond to the columns of $A$. Accordingly we call an index $i \in [n+m]$, a row index if $i \leq n$, and
	a column index if $i > n$.
	Row indices will be denoted by $e$, and column indices by $v$, following the graph-theoretic interpretation 
	in which they correspond to edges and vertices, respectively.
\end{definition}
The following symmetry properties follow immediately from the block structure of $M$ :
\begin{fact}
	If $I$ and $J$ consist entirely of row indices or entirely of column indices, then $\Psi^{I,J} = \Psi^{J,I}$.
	However if $e$ is a row index and $v$ is a column index, then $\Psi^{e,v} = - \Psi^{v,e}$.
\end{fact}

Next, we discuss several key identities relating Dodgson polynomials, the Laplacian $L$, and the graph matrix $M$. 
All of these identities are already known in the literature, and we indicate where they were originally proved. 
However, since many of them were established only in the graph-theoretic setting, 
we also provide proofs in the more general context of arbitrary $n\times m$ matrices $A$.

First we can express the entries of $L_{A}^{-1}(x^{-1})$ in terms of Dodgson polynomials:
\begin{lemma}[{\cite[Lemma 3.9]{brown21}}]\label{lem:laplToDodgson}
	Let $(u,v) \in [m]^2$, then
	\[
		(L_{A}^{-1}(x^{-1}))_{u,v} = (-1)^{u+v} \frac{\Psi^{n+u,n+v}}{\Psi}
	.\] 
\end{lemma}

\begin{proof}
	Denote the $k \times k$ identity matrix by $I_{k}$ and define the diagonal matrix $D := \diag(x_1,\ldots,x_{n})$. 
	The minor $M^{n+u,n+v}$ can be decomposed as $M = F G H$, where
	 \[
		 F = \begin{pmatrix}[c|c]
			 I_{n}& 0\\ \hline
			 (A^{T})^{u,\bullet} D^{-1} & I_{m-1}
		 \end{pmatrix} \qquad
		 G = \begin{pmatrix}[c|c]
			 D& 0\\ \hline
			 0 & (L_{A}(x^{-1}))^{u,v}
		 \end{pmatrix} \qquad
		 H = \begin{pmatrix}[c|c]
			 I_{n}& - D^{-1} A^{\bullet,v}\\ \hline
			 0 & I_{m-1}
		 \end{pmatrix} 
	.\] 
	Thus we find that
	\[
		\frac{\Psi^{n+u,n+v}}{\Psi} = \frac{\det(M^{n+u,n+v})}{\Psi} = \frac{\det(F G H)}{\Psi} 
		= \det((L_{A}(x^{-1}))^{u,v}) \frac{\prod_{e \in [n]} x_{e}}{\Psi} = (L_{A}^{-1}(x^{-1}))_{u,v}
	\]
	where in the last equality we used that $\det(L_{A}(x^{-1})) = \Psi / \prod_{e \in [n]} x_{e}$ as well as 
	$L_{ij}^{-1} = \frac{\det(L^{ij})}{\det(L)}$  for an invertible matrix $L$.
\end{proof}

Next, we may reformulate an identity due to Jacobi to
write $\Psi^{I,J}$ as sums over products of Dodgson polynomials $\Psi^{u,v}$.
\begin{lemma}[{\cite[Lemma 48]{schnetz24}}] \label{lem:dodgsonToAdjugate}
	Let $I, J \subseteq [n]$ be sets of row indices of equal cardinality. Then
	\[
	\Psi^{I,J} = \det(M^{I,J}) = \frac{\sigma(I) \sigma(J)}{\det(M)^{\abs{I} - 1}} \det(\adj(M)_{I,J})
	\]
	where $\sigma(I) = (-1)^{i_1+\ldots+i_{k}} \sgn(I)$.
\end{lemma}
A proof of this identity can be found in \cite[Lemma~28]{brown09}.
There is also an explicit expression for the Dodgson polynomials $\Psi^{I,J}$ in terms of the matrix $A$:
\begin{lemma}[{\cite[Proof of Proposition 23]{brown09}}]\label{lem:detProductAsDodgson}
	Let $I,J \subseteq [n]$ of equal cardinality and
	let $Q$ be the set of subsets of  $[n] \setminus \{I \cup  J\}$ of cardinality $n - m - \abs{I}$,
	Then
	\[
		\Psi^{I,J} = \sum_{R \in Q} (-1)^{\sum_{e \in R} \abs{I_{< e}} + \abs{J_{<e}}}
		\det(A^{R + I,\bullet}) \det(A^{R + J, \bullet}) \prod_{e \in R} x_{e}
	.\]
	In the case $R = \emptyset$, the product $\prod_{e \in R} x_{e}$ is understood to equal $1$.
\end{lemma}

\begin{proof}
	Since the lower-right $m \times m $ submatrix of $M$ is zero, any nonzero term in  $\det(M^{I,J})$ must 
	choose $m$ rows/columns from each submatrix $A$. 
	After removing the rows indexed by $I$ and the columns indexed by $J$, this leaves $n-m-\abs{I}$
	diagonal entries from the submatrix $D$ to be chosen.

	Let $Q$ denote the set of all subsets $R \subseteq [n]\setminus (I\cup J)$ of cardinality $n-m-\abs{I}$. 
	Expanding along the diagonal submatrix $D$ gives
	\[
	\Psi^{I,J} = \det(M^{I,J}) = \sum_{R \in Q} (-1)^{\sum_{e \in R} \abs{I_{< e}} + \abs{J_{<e}}}\begin{vmatrix} 
		0 & -A^{J+R,\bullet}\\ (A^{I+R,\bullet})^{T} & 0
	\end{vmatrix}  \prod_{e \in R} x_{e} 
	.\]
	where the sign is the usual Laplace sign arising from selecting the diagonal entries,
	giving $(-1)^{e - \abs{I_{< e}} + e - \abs{J_{< e}}}$ for each $x_{e}$.
	Finally, to bring the remaining matrix into block diagonal form to write it as the product of determinants, one permutes rows and columns. 
	The sign of this permutation is cancelled by the minus sign in the upper-right block, and the formula follows.
\end{proof}

\begin{remark}
	In the case of $S$ having cardinality $n - m$, $R$ is just an empty set and the formula reduces to
	\[
		\Psi^{I,J} = \det(A^{I,\bullet}) \det(A^{J,\bullet})
	,\] 
	so that the Dodgson polynomial $\Psi^{I,J}$ is actually just $\pm 1$.

	Moreover, for the case of graphs,
	only the terms where $R + S$ and $R + T$ are spanning tree compliments contribute as else one of the determinants of $A$ will vanish.
\end{remark}

Finally, we have an identity relating column indexed to row indexed Dodgson polynomials. For this we first need the following additional lemma:

\begin{lemma}[{\cite[Lemma A.11]{balduf24}}]\label{lem:dodgsonVtoE}
	Let $e \in [n]$ be a row index and let $v \in [n+m]$, be a row or column index distinct from $e$. Then
	\[
		\sum_{j \in [m] } (-1)^{j} a_{ej} \psi^{j,v} = (-1)^{e + n-1} x_{e} \Psi^{e,v}
	.\] 
\end{lemma}

\begin{proof}
	The proof of this theorem is basically an exercise in Laplace expansion after one observes that the structure of the $e$-th column is given
	by one entry $x_{e}$ in the $e$-th row and then the entries $a_{e j}$ in the rows $n+1$ to $n + m$.

	First consider the minor of $M^{e,v}$ corresponding to $\Psi^{e,v}$ and expand it along the $e$-th column to obtain:
	\[
		\Psi^{e,v} = \sum_{j=1}^{m} (-1)^{e + (n-1+j)}  a_{e j} \Psi^{\{e,j\},\{e,v\}} \qq{thus} 
		(-1)^{e+n-1} x_{e} \Psi^{e,v} = \sum_{j=1}^{m} (-1)^{j}  a_{e j} x_{e} \Psi^{\{e,j\},\{e,v\}}
	.\]
	Let $j \in [m]$, then we can consider the minor $M^{j,v}$ corresponding to $\Psi^{j,v}$ and expand it along the $e$-th column to obtain:
	 \[
		 \Psi^{j,v} = (-1)^{e+e} x_{e} \Psi^{\{e,j\},\{e,v\}} +
		 \sum_{k \in [m] \setminus \{j\}}  (-1)^{e + n+k-(j < k)} a_{e k} \Psi^{\{j,k\},\{e,v\}}
	,\] 
	where the $(j < k)$ in the sign arises from the fact that if  $j$ is smaller than $k$ then in the submatrix $M^{j,v}$ the $k$-th row of
	 $M$ appears in position $k-1$, whereas if  $j > k$ then the  $k$-th row  is still the $k$-th row.
	Let us then consider
	\[
		 \sum_{j \in [m]} (-1)^{j} a_{e j} \Psi^{j,v} = \sum_{j \in [m]} (-1)^{j} a_{e j} x_{e} \Psi^{\{e,j\}, \{e,v\} }
			 \sum_{\substack{j,k =1\\ j \neq k}}^{m} (-1)^{e + n + k +j - (j<k)} a_{e j}
		 a_{e k} \Psi^{\{j,k\}, \{e,v\}}
	.\]
	Here the second sum vanishes as the term $(j,k)$  and  $(k,j)$ appear with opposite sign.
	Comparing the remaining term in the expansion of the sum over $\Psi^{j,v}$ with the expansion of $\Psi^{e,v}$ proves the lemma.
\end{proof}
Using the previous lemma, the identity between row indexed and column indexed Dodgson polynomials is given by:
\begin{lemma}[{\cite[Lemma A.12]{balduf24}}]\label{lem:eLeToM}
	Let $e_1, e_2 \in [n]$ be distinct row indices. Then
	\[
		\sum_{j,k \in [m]} (-1)^{j+k} a_{e_1 j} a_{e_2 k} \Psi^{j,k} = (-1)^{e_1+e_2+1} x_{e_1} x_{e_2} \Psi^{e_1,e_2}
	.\]
	which can be summarised as
	\[
		(A L(x)^{-1} A^{T})_{e_1 e_2} = - x_{e_1} x_{e_2} M(x)^{-1}_{e_1 e_2}
	,\] 
\end{lemma}

\begin{proof}
	Consider the right hand side and apply \cref{lem:dodgsonVtoE} once with  $e = e_1$ and $v = e_2$ to obtain:
	\[
		(-1)^{e_1+e_2+1} x_{e_1} x_{e_2} \Psi^{e_1 e_2} = \sum_{j \in [m]} a_{e_1 j} (-1)^{j+e_2+n} x_{e_2} \Psi^{j,e_2}
	.\] 
	Now take the $(-1)^{e_2 + n -1} x_{e_2} \Psi^{j,e_2}$ in every summand and apply \cref{lem:dodgsonVtoE} with $e = e_2$ and $v = j$.
	Notice that as  $\Psi^{e,v} = - \Psi^{v,e}$ we obtain a negative sign,
	\[
		\sum_{j \in [m]} a_{e_1 j} (-1)^{j+e_2+n} x_{e_2} \Psi^{j,e_2} 
		= \sum_{j \in [m]} \sum_{\substack{k \in [m]}} (-1)^{j+k} a_{e_1 j} a_{e_2 k } \Psi^{k,j}
	.\] 
	As both $j$ and $k$ are column indices  $\Psi^{j,k} = \Psi^{k,j}$, which proves the first identity.

	Let $v \in [m]$, then
	\[
		(L(x^{-1})^{-1} A^{T})_{v e_2} = \sum_{k = 1}^{m} a_{e_2 k} (L^{-1}(x^{-1}))_{v k}
	\]
	and for the full expansion:
	\[
		(A L(x^{-1})^{-1} A^{T})_{e_1 e_2} = \sum_{j,k \in [m]} a_{e_1 j} a_{e_2 k} 
	(L^{-1}(x^{-1}))_{j k} = \sum_{j,k \in [m]} (-1)^{j+k}  a_{e_1 j} a_{e_2 k} \frac{\Psi^{j,k}}{\Psi}
	,\]
	where in the second equality we used \cref{lem:laplToDodgson}.
	Applying the first identity we obtain:
	\[
		\sum_{j,k \in [m]} (-1)^{j+k}  a_{e_1 j} a_{e_2 k} \frac{\Psi^{j,k}}{\Psi}
		= (-1)^{e_1 + e_2 + 1} x_{e_1} x_{e_2} \frac{\Psi^{e_1,e_2}}{\Psi} = - x_{e_1} x_{e_2} (M(x)^{^{-1}})_{e_1 e_2} 
	.\] 
	which proves the second identity.
\end{proof}

\graphicspath{{Images/}}

\section{Fubinis theorem, Leibniz integral rule and dominated convergence}
We summarise three classic results from measure theory which we frequently use in the main text.
They allow us to exchange limits and derivatives with integrals as well as exchange integrals themselves.

\begin{theorem}[Fubini's Theorem]\label{thm:fubinisTheorem}
	Let $X,Y$ be $\sigma$-finite measure spaces and suppose that $X \times Y$ has the product measure.
	If $f$ is a measureable function and its integral over $X \times Y$ converges absolutely, that is
	\[
		\int_{X \times Y} \abs{f(x,y)} \dd{(x,y)} < \infty
	\]
	then
	\[
		\int_{X} \int_{Y} f(x,y) \dd{y} \dd{y} = 
		\int_{Y} \int_{X} f(x,y) \dd{x} \dd{y} =
		\int_{X \times Y} \abs{f(x,y)} \dd{(x,y)}
	.\] 
\end{theorem}

\begin{theorem}[Dominated Convergence Theorem (DCT)]\label{thm:dominatedConvergenceTheorem}
	Let $X$ be an open subset of $\R$, $\Omega$ a measure space and fix $x_{0} \in X$.  
	Let $f: X \times \Omega \to \C$ be a complex-valued measureable function such that:
	\begin{enumerate}
		\item For almost all $\omega \in \Omega$
			\[
			\lim_{x \to x_0} f(x,\omega) = f(x_0,\omega)
			.\] 
		\item There exists a integrable function $g: \Omega \to \R$ such that for all $x \in X$ and almost all  $\omega \in \Omega$ 
			$\abs{f(x,\omega)} < g(\omega)$.
	\end{enumerate}
	then 
	\[
		\lim_{x \to x_0} \int_{\Omega} f(x,\omega) \dd{\omega} = \int_{\Omega} f(x_0,\omega) \dd{\omega}
	.\] 
\end{theorem}

\begin{theorem}[Leibniz integral rule]\label{thm:leibnizRule}
	Let $X$ be an open subset of $\R$, $\Omega$ a measure space and  $f: X \times \Omega \to \C$ a function such that:
	\begin{enumerate}
		\item For each $x \in X$ $f(x,\omega)$ is an integrable function of $\omega$.
		\item For almost all $\omega \in \Omega$, the partial derivative $\frac{\partial f}{\partial x}$ exists for all $x \in X$.
		\item There exists an integrable function $g: \Omega \to \R$ such that for all $x \in X$ and almost all $\omega \in \Omega$
			$\abs{\frac{\partial f}{\partial x}(x,\omega)} < g(\omega)$.
	\end{enumerate}
	Then for $x \in X$
	\[
		\frac{\partial }{\partial x} \int_{\Omega} f(x,\omega) \dd{\omega} = \int_{\Omega} \frac{\partial f}{\partial x}(x,\omega) \dd{\omega}
	.\] 
\end{theorem}

\section{Divided differences}\label{ap:divDiff}
We also briefly use a key theorem from the theory of divided differences.
As a reference for this we use \cite{jamesonXX}.

\begin{definition}\label{def:dividedDifference}
	Let $n \in \N$ and $I \subseteq \R$ and interval.
	Let $f: I \to \R$ be $n$-times differentiable and $x_0,\ldots,x_{n}$ be $n+1$ points in $I$ then the 
	\emph{divided difference} of $f$ is defined as
	\begin{align*}
		f[x_{k}] &:=f(x_{k}), & &k \in \{0,\ldots ,n\}\\
		f[x_{k},\ldots ,x_{j}]&:={\frac {f[x_{k+1},\ldots ,x_{j}]-f[x_{k},\ldots ,x_{j-1}]}{x_{j}-x_{k}}},
							  &&k\in \{0,\ldots ,n-1\},\ j\in \{k+1,\ldots ,n\}.
	\end{align*}
	By defining
	\[
		f[x_0,\ldots,x_0] = \frac{f^{(n)}(x_0)}{n!}
	\] 
	this definition can be continuously extended to the points where some $x_{j}$ coincide.
\end{definition}
An explicit non-recursive formula for the divided difference (whenever no points coincide) is given by
\begin{proposition}\label{prop:dividedDifference}
	Let $f$ as above and let $x_0,\ldots,x_{n}$ be distinct. Then
	\[
		f[x_0,\ldots,x_{n}] = \sum_{j=0}^{n} \frac{f(x_{j})}{\prod\limits_{\substack{k=0\\ k \neq j}}^{n} (x_{j} - x_{k})}
	\] 
\end{proposition}
The key theorem we use in this work is the following:
\begin{theorem}[Mean value theorem for divided differences]\label{thm:meanDividedDiff}
	There exists a point $\xi$ in the interior of $I$ such that 
	\[
		f[x_0,\ldots,x_{n}] = \frac{f^{n}(\xi)}{n!}
	.\] 
\end{theorem}

\section*{References}
\printbibliography[heading=none]

\end{document}